\documentclass[11pt]{article}
\usepackage{fullpage,amsfonts,amsmath,amssymb,amsthm,bm, xspace}
\usepackage{multirow,graphicx, algorithm,algorithmic,caption,subcaption,url,cite,float}
\usepackage{mathtools}
\usepackage[colorlinks=true,urlcolor=blue,citecolor=blue]{hyperref}
\usepackage{todonotes}
\usepackage{pgfplots}
\usepackage{array} 

\newtheorem{theorem}{Theorem}[section]
\newtheorem{lemma}[theorem]{Lemma}
\newtheorem{corollary}[theorem]{Corollary}
\newtheorem{proposition}[theorem]{Proposition}
\newtheorem{definition}[theorem]{Definition}

\theoremstyle{remark}

\numberwithin{equation}{section}

\newcommand{\A}{\mathcal{A}}
\newcommand{\C}{\mathbb{C}}
\newcommand{\R}{\mathbb{R}}
\newcommand{\E}{\mathbb{E}}
\newcommand{\Z}{\mathbb{Z}}
\renewcommand{\Pr}{\mathbb{P}}

\newcommand{\vct}[1]{\bm{#1}}

\newcommand{\norm}[1]{\left\lVert{#1}\right\rVert}

\newcommand{\defeq}{\vcentcolon=}

\newcommand{\tmop}{\operatorname*}

\newcommand{\cfg}[1]{\textcolor{blue}{#1}}
\newcommand{\Id}{\text{\em I}}
\newcommand{\MAT}[1]{\begin{bmatrix} #1 \end{bmatrix}}

\newcommand{\abs}[1]{\left| #1 \right|}
\newcommand{\keys}[1]{\left\{ #1 \right\}}
\newcommand{\sqbr}[1]{\left[ #1 \right]}
\newcommand{\brac}[1]{\left( #1 \right) }
\newcommand{\ml}[1]{\mathcal{ #1 } }
\newcommand{\op}[1]{ \operatorname{#1} }
\newcommand{\normInf}[1]{\left|\left| #1 \right|\right| _{\infty}}
\newcommand{\normTwo}[1]{\left|\left| #1 \right|\right| _{2}}
\newcommand{\normOne}[1]{\left|\left| #1 \right|\right| _{1}}
\newcommand{\normTV}[1]{\left|\left| #1 \right|\right| _{\op{TV}}}
\newcommand{\normF}[1]{\left|\left| #1 \right|\right| _{F}}
\newcommand{\atomicnorm}[1]{\left|\left| #1 \right| \right|_{\mathcal{A}}}
\newcommand{\der}[2]{\frac{\text{d}#2}{\text{d}#1}}
\newcommand{\derTwo}[2]{\frac{\mathrm{d} ^2#2}{\mathrm{d}#1^2}}
\newcommand{\set}[2]{ \keys{ #1 \; | \; #2 } }
\newcommand{\PROD}[2]{\left \langle #1, #2\right \rangle}
\newcommand{\diff}[1]{ \, \operatorname{d}#1 }

\newcommand{\signx}{\vct{h}}
\newcommand{\signz}{\vct{r}}

\newcommand{\mindisthalf}{1.26}
\newcommand{\mindist}{2.52}
\newcommand{\gammaOne}{0.247}
\newcommand{\gammaTwo}{0.339}
\newcommand{\gammaThree}{0.414}

\title{Demixing Sines and Spikes:\\ Robust Spectral Super-resolution in the Presence of Outliers}

\author{Carlos
  Fernandez-Granda\thanks{Courant Institute of Mathematical Sciences and Center for Data Science,
    NYU, New York, NY} , Gongguo Tang\thanks{Department of Electrical Engineering and Computer Science, Colorado School of Mines, Golden, CO} , Xiaodong Wang\thanks{Electrical Engineering Department, Columbia University, New York, NY} \,  and Le Zheng\footnotemark[3]}

\date{September 2016}

\begin{document}

\maketitle

\vspace{-0.3in}

\begin{abstract}
We consider the problem of super-resolving the line spectrum of a multisinusoidal signal from a finite number of samples, some of which may be completely corrupted. Measurements of this form can be modeled as an additive mixture of a sinusoidal and a sparse component. We propose to demix the two components and super-resolve the spectrum of the multisinusoidal signal by solving a convex program. Our main theoretical result is that-- up to logarithmic factors-- this approach is guaranteed to be successful with high probability for a number of spectral lines that is linear in the number of measurements, even if a constant fraction of the data are outliers. The result holds under the assumption that the phases of the sinusoidal and sparse components are random and the line spectrum satisfies a minimum-separation condition. We show that the method can be implemented via semidefinite programming and explain how to adapt it in the presence of dense perturbations, as well as exploring its connection to atomic-norm denoising. In addition, we propose a fast greedy demixing method which provides good empirical results when coupled with a local nonconvex-optimization step.  
\end{abstract}


\section{Introduction}


The goal of \emph{spectral super-resolution} is to estimate the spectrum of a multisinusoidal signal from a finite number of samples. This is a problem of crucial importance in signal-processing applications, such as target identification from radar measurements~\cite{berni1975target,carriere1992high}, digital filter design~\cite{smith2008introduction}, underwater acoustics~\cite{beatty1978use}, seismic imaging~\cite{borcea2002imaging}, nuclear-magnetic-resonance spectroscopy~\cite{viti1997prony} and power electronics~\cite{leonowicz2003advanced}. In this paper, we study spectral super-resolution in the presence of perturbations that completely corrupt a subset of the data. The corrupted samples can be interpreted as \emph{outliers} that do not follow the same multisinusoidal model as the rest of the measurements and complicate significantly the task of super-resolving the spectrum of the signal of interest. Depending on the application, outliers may appear due to sensor failures, interference from other signals or impulsive noise. For instance, radar measurements can be corrupted by lightning discharges, spurious radio emissions or telephone switching transients~\cite{faxin2001effective,lu2010impulsive}. 

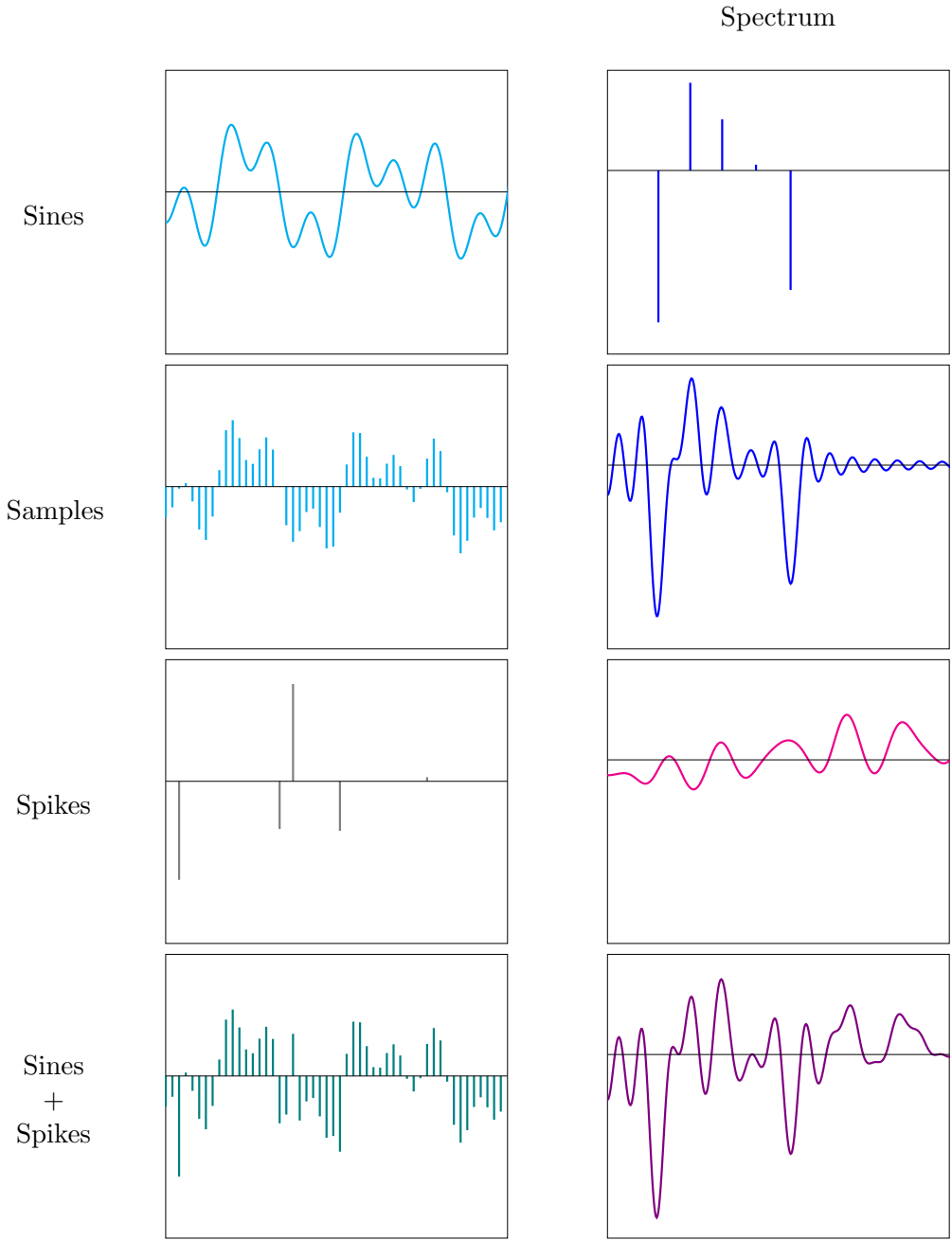
\begin{figure}[tp]
\centering
\begin{tabular}{  >{\centering\arraybackslash}m{0.08\linewidth} >{\centering\arraybackslash}m{0.35\linewidth} >{\centering\arraybackslash}m{0.35\linewidth}  }
 && Spectrum\\ \\
Sines &
\begin{tikzpicture}[scale=0.7]
\begin{axis}[ticks=none,ymin=-8,ymax=6,xmin=0,xmax=51]
\addplot[very thick,cyan] file {sines.dat};
\addplot[black] coordinates {(0,0) (100,0)};
\end{axis}
\end{tikzpicture}  
&  
\begin{tikzpicture}[scale=0.7]
\begin{axis}[ticks=none,ymin=-1.1,ymax=0.6,xmin=0,xmax=0.3]
\addplot+[ycomb,mark=none,very thick,blue] file {sines_spectrum.dat};
\addplot[black] coordinates {(0,0) (0.3,0)};
\end{axis}
\end{tikzpicture} 
\\
Samples &
\begin{tikzpicture}[scale=0.7]
\begin{axis}[ticks=none,ymin=-8,ymax=6,xmin=0,xmax=51]
\addplot+[ycomb,mark=none,very thick,cyan] file {sines_samples.dat};
\addplot[black] coordinates {(0,0) (100,0)};
\end{axis}
\end{tikzpicture}  
&  
\begin{tikzpicture}[scale=0.7]
\begin{axis}[ticks=none,ymin=-1.1,ymax=0.6,xmin=0,xmax=0.3]
\addlegendimage{ no markers, blue, very thick}
\addplot[very thick,blue] file {sines_lowres.dat};
\addplot[black] coordinates {(0,0) (0.3,0)};
\end{axis}
\end{tikzpicture} 
\\
Spikes &
\begin{tikzpicture}[scale=0.7]
\begin{axis}[ticks=none,ymin=-8,ymax=6,xmin=0,xmax=51]
\addplot+[ycomb,mark=none,very thick,gray] file {spikes.dat};
\addplot[black] coordinates {(0,0) (51,0)};
\end{axis}
\end{tikzpicture} 
&
\begin{tikzpicture}[scale=0.7]
\begin{axis}[ticks=none,ymin=-1.1,ymax=0.6,xmin=0,xmax=0.3]
\addplot[very thick,magenta] file {spikes_spectrum.dat};
\addplot[black] coordinates {(0,0) (0.3,0)};
\end{axis}
\end{tikzpicture}  
\\
Sines \hspace{2cm} 
+ 
Spikes 
& 
\begin{tikzpicture}[scale=0.7]
\begin{axis}[ ticks=none,ymin=-8,ymax=6,xmin=0,xmax=51]
\addplot+[ycomb,mark=none,very thick,teal] file {sinesspikes.dat};
\addplot[black] coordinates {(0,0) (51,0)};
\end{axis}
\end{tikzpicture}  
&
\begin{tikzpicture}[scale=0.7]
\begin{axis}[ ticks=none,ymin=-1.1,ymax=0.6,xmin=0,xmax=0.3]
\addplot[very thick,violet] file {sinesspikes_spectrum.dat};
\addplot[black] coordinates {(0,0) (0.3,0)};
\end{axis}
\end{tikzpicture}  
\end{tabular}
\caption{The top row shows a multisinusoidal signal (left) and its sparse spectrum (right). The minimum separation of the spectrum is $2.8 / (n - 1)$ (see Section~\ref{sec:minimum_separation}). On the second row, truncating the signal to a finite interval after measuring $n:= 101$ samples at the Nyquist rate (left) results in aliasing in the frequency domain (right). The third row shows some impulsive noise (left) and its corresponding spectrum (right). The last row shows the superposition of the multisinusoidal signal and the sparse noise, which yields a mixture of \emph{sines} and \emph{spikes} depicted in the time (left) and frequency domains (right). For ease of visualization, the amplitudes of the spectrum of the sines and of the spikes are real (we only show half of the spectrum and half of the spikes because their amplitudes and positions are symmetric).}
\label{fig:sines_spikes}
\end{figure}

Figure~\ref{fig:sines_spikes} illustrates the problem of performing spectral super-resolution in the presence of outliers. The top row shows a superposition of sinusoids and its corresponding sparse spectrum. In the second row, the multisinusoidal signal is sampled at the Nyquist rate over a finite interval, which induces spectral aliasing and makes it challenging to resolve the individual spectral lines. The sparse signal in the third row represents an additive perturbation that corrupts some of the samples. Finally, the bottom row shows the available measurements: a mixture of \emph{sines} (samples from the multisinusoidal signal) and \emph{spikes} (the sparse perturbation). Our objective is to \emph{demix} these two components and super-resolve the spectrum of the sines.  

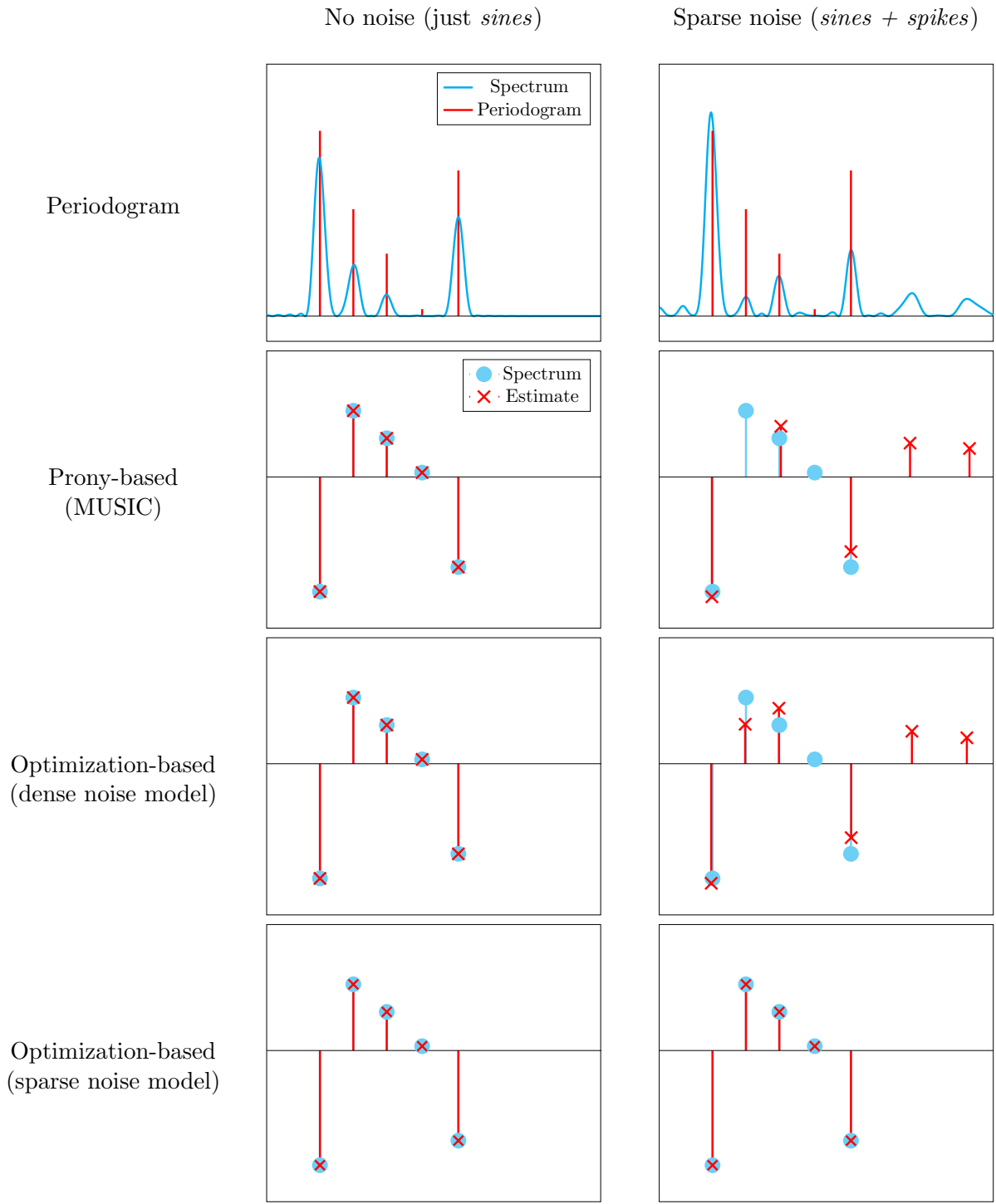
\begin{figure}[tp]
\centering
\begin{tabular}{  >{\centering\arraybackslash}m{0.21\linewidth} >{\centering\arraybackslash}m{0.36\linewidth} >{\centering\arraybackslash}m{0.35\linewidth}  }
 & No noise (just \emph{sines}) & Sparse noise (\emph{sines + spikes}) \\ \\
Periodogram &
\begin{tikzpicture}[scale=0.78]
\begin{axis}[ticks=none,ymax=19,xmin=0,xmax=0.28, legend pos=north east]
\addplot+[ycomb,mark=none,very thick,red,forget plot] file {periodogram_amp.dat};
\addplot[very thick,cyan,forget plot] file {periodogram_clean.dat};
\addplot[black,forget plot] coordinates {(0,0) (0.3,0)};
\addlegendentry{Spectrum}
\addlegendimage{very thick,cyan}
\addlegendentry{Periodogram}
\addlegendimage{very thick, red}
\end{axis}
\end{tikzpicture} 
&
\begin{tikzpicture}[scale=0.78]
\begin{axis}[ticks=none,ymax=19,xmin=0,xmax=0.28]
\addplot[very thick,cyan] file {periodogram.dat};
\addplot+[ycomb,mark=none,very thick,red] file {periodogram_amp.dat};
\addplot[black] coordinates {(0,0) (0.3,0)};
\end{axis}
\end{tikzpicture} 
\\
Prony-based (MUSIC) &
\begin{tikzpicture}[scale=0.78]
\begin{axis}[ticks=none, legend pos=north east,
ymin=-1.2,ymax=1,xmin=0,xmax=0.28]
\addplot+[ycomb,mark=*, cyan!50!white,very thick,mark options={solid,scale=2}]  file {sines_spectrum.dat};
\addplot+[ycomb,mark=x, red,mark options={fill=red,scale=2.5},very thick]  file {music_clean.dat};
\addplot[black,forget plot] coordinates {(0,0) (0.3,0)};
\addlegendentry{Spectrum}
\addlegendimage{mark=x, dashed,blue,very thick,mark options={solid,scale=2.5},line width=4pt}
 \addlegendentry{Estimate}
\addlegendimage{mark=*, red,mark options={ fill=red,scale=2.5},very thick}
\end{axis}
\end{tikzpicture} 
&
\begin{tikzpicture}[scale=0.78]
\begin{axis}[ticks=none,legend style={font=\small,at={(0,0)},anchor=south west,legend columns=-1,/tikz/every even column/.append style={column sep=0.25cm},draw=none},ymin=-1.2,ymax=1,xmin=0,xmax=0.28]
\addplot+[ycomb,mark=*, cyan!50!white,very thick,mark options={solid,scale=2}]  file {sines_spectrum.dat};
\addplot+[ycomb,mark=x, red,mark options={fill=red,scale=2.5},very thick]  file {music.dat};
\addplot[black,forget plot] coordinates {(0,0) (0.3,0)};
\end{axis}
\end{tikzpicture} 
 \\ Optimization-based (dense noise model)
  &
\begin{tikzpicture}[scale=0.78]
\begin{axis}[ticks=none,legend style={font=\small,at={(0,0)},anchor=south west,legend columns=-1,/tikz/every even column/.append style={column sep=0.25cm},draw=none},ymin=-1.2,ymax=1,xmin=0,xmax=0.28]
\addplot+[ycomb,mark=*, cyan!50!white,very thick,mark options={solid,scale=2}]  file {sines_spectrum.dat};
\addplot+[ycomb,mark=x, red,mark options={fill=red,scale=2.5},very thick]  file {sdp_clean.dat};
\addplot[black,forget plot] coordinates {(0,0) (0.3,0)};
\end{axis}
\end{tikzpicture} 
 &
\begin{tikzpicture}[scale=0.78]
\begin{axis}[ticks=none,legend style={font=\small,at={(0,0)},anchor=south west,legend columns=-1,/tikz/every even column/.append style={column sep=0.25cm},draw=none},ymin=-1.2,ymax=1,xmin=0,xmax=0.28]
\addplot+[ycomb,mark=*, cyan!50!white,very thick,mark options={solid,scale=2}]  file {sines_spectrum.dat};
\addplot+[ycomb,mark=x, red,mark options={scale=2.5},very thick]  file {sdp.dat};
\addplot[black,forget plot] coordinates {(0,0) (0.3,0)};
\end{axis}
\end{tikzpicture} 
 \\
 Optimization-based (sparse noise model)
 &
\begin{tikzpicture}[scale=0.78]
\begin{axis}[ticks=none,legend style={font=\small,at={(0,0)},anchor=south west,legend columns=-1,/tikz/every even column/.append style={column sep=0.25cm},draw=none},ymin=-1.2,ymax=1,xmin=0,xmax=0.28]
\addplot+[ycomb,mark=*, cyan!50!white,very thick,mark options={solid,scale=2}]  file {sines_spectrum.dat};
\addplot+[ycomb,mark=x, red,mark options={fill=red,scale=2},very thick]  file {sdp_robust_clean.dat};
\addplot[black,forget plot] coordinates {(0,0) (0.3,0)};
\end{axis}
\end{tikzpicture} 
&
\begin{tikzpicture}[scale=0.78]
\begin{axis}[ticks=none,legend style={font=\small,at={(0,0)},anchor=south west,legend columns=-1,/tikz/every even column/.append style={column sep=0.25cm},draw=none},ymin=-1.2,ymax=1,xmin=0,xmax=0.28]
\addplot+[ycomb,mark=*, cyan!50!white,very thick,mark options={solid,scale=2}]  file {sines_spectrum.dat};
\addplot+[ycomb,mark=x, red,mark options={fill=red,scale=2},very thick]  file {sdp_robust.dat};
\addplot[black,forget plot] coordinates {(0,0) (0.3,0)};
\end{axis}
\end{tikzpicture} 
\end{tabular}
\caption{Estimate of the sparse spectrum of the multisinusoidal signal from Figure~\ref{fig:sines_spikes} when outliers are absent from the data (left column) and when they are present (right column). The estimates are shown in red; the true location of the spectra is shown in blue. Methods that do not account for outliers fail to recover all the spectral lines when impulsive noise corrupts the data, whereas an optimization-based estimator incorporating a sparse-noise model still achieves exact recovery.}
\label{fig:methods}
\end{figure}

Broadly speaking, there are three main approaches to spectral super-resolution: linear nonparametric methods~\cite{Stoica:2005wf}, techniques based on Prony's method~\cite{deProny:tg, Stoica:2005wf} and optimization-based methods~\cite{tang2014minimax,superres_new,Bhaskar:2012tq}. The first three rows of Figure~\ref{fig:methods} show the results of applying a representative of each approach to a spectral super-resolution problem when there are no outliers in the data (left column) and when there are (right column). 

In the absence of corruptions, the periodogram-- a linear nonparametric technique that uses windowing to reduce spectral aliasing~\cite{windows_harris}-- locates most of the relevant frequencies, albeit at a coarse resolution. In contrast, both the Prony-based approach-- represented by the Multiple Signal Classification (MUSIC) algorithm~\cite{music_bienvenu,music_schmidt}-- and the optimization-based method-- based on total-variation norm minimization~\cite{superres_noisy,tang2014minimax,Bhaskar:2012tq}-- recover the true spectrum of the signal perfectly. All of these techniques are designed to allow for small Gaussian-like perturbations to the data and hence their performance degrades gracefully when such noise is present (not shown in the figure). However, as we can see in the right column of Figure~\ref{fig:methods}, when outliers are present in the data their performance is severely affected: none of the methods detect the fourth spectral line of the signal and they all hallucinate two large spurious spectral lines to the right of the true spectrum. 

The subject of this paper is an optimization-based method that leverages sparsity-inducing norms to perform spectral super-resolution and simultaneously detect outliers in the data. The bottom row of Figure~\ref{fig:methods} shows that this approach is capable of super-resolving the spectrum of the multisinusoidal signal in Figure~\ref{fig:sines_spikes} exactly from the corrupted measurements, in contrast to techniques that do not account for the presence of outliers in the data. Below is a brief roadmap of the paper.

\begin{itemize}
\item 
Section~\ref{sec:main_results} describes our methods and main results. In Section~\ref{sec:math_model} we introduce a mathematical model of the spectral super-resolution problem. Section~\ref{sec:minimum_separation} justifies the need of a minimum-separation condition on the spectrum of the signal for spectral super-resolution to be well posed. In Section~\ref{sec:main_result} we present our optimization-based method and provide a theoretical characterization of its performance. Section~\ref{sec:reg_param} discusses the robustness of the technique to the choice of regularization parameter. Section~\ref{sec:dense_noise} explains how to adapt the method when the data are perturbed by dense noise. Section~\ref{sec:denoising} establishes a connection between our method and atomic-norm denoising. Finally, in Section~\ref{sec:related_work} we review the related literature.
\item Our main theoretical contribution-- Theorem~\ref{theorem:main}-- establishes that solving the convex program introduced in Section~\ref{sec:main_result} allows to super-resolve up to $k$ spectral lines exactly in the presence of $s$ outliers (i.e. when $s$ measurements are completely corrupted) with high probability from a number of data that is linear both in $k$ and $s$ up to logarithmic factors. Section~\ref{sec:proof} is dedicated to the proof of this result, which is non-asymptotic and holds under several assumptions that are described in Section~\ref{sec:main_result}.
\item Section~\ref{sec:algorithms} focuses on demixing algorithms. In Sections~\ref{sec:sdp} and \ref{sec:sdp_noise} we explain how to implement the methods discussed in Sections~\ref{sec:main_result} and \ref{sec:dense_noise} respectively by recasting the dual of the corresponding optimization problems as a tractable semidefinite program. In Section~\ref{sec:greedy} we propose a greedy demixing technique that achieves good empirical results when combined with a local nonconvex-optimization step. Section~\ref{sec:atomic_norm} describes the implementation of atomic-norm denoising in the presence of outliers using semidefinite programming. Matlab code of all the algorithms discussed in this section is available online\footnote{\url{http://www.cims.nyu.edu/~cfgranda/scripts/spectral_superres_outliers.zip}}.
\item Section~\ref{sec:numerical} reports numerical experiments illustrating the performance of the proposed approach. In Section~\ref{sec:phase_transitions} we investigate under what conditions our optimization-based method achieves exact demixing empirically. In Section~\ref{sec:experiments_denoising} we compare atomic-norm denoising to an alternative approach based on matrix completion. 
\item We conclude the paper outlining several future research directions in Section~\ref{sec:conclusion}.
\end{itemize}
\section{Robust spectral super-resolution via convex programming}
\label{sec:main_results}
\subsection{Mathematical model}
\label{sec:math_model}
We model the multisinusoidal signal of interest as a superposition of $k$ complex exponentials 
\begin{align}
\label{eq:multisinusoidal}
g\brac{t} & := 
\sum_{j =1 }^{k} \vct{x}_j \exp \brac{i 2 \pi f_j t},
\end{align}
where $\vct{x} \in \C^{k}$ is the vector of complex amplitudes and $\vct{x}_j$ is its $j$th entry. The spectrum of $g$ consists of spectral lines, modeled by Dirac measures that are supported on a subset $T:=\keys{f_1, \ldots, f_k}$ of the unit interval $\sqbr{0,1}$
\begin{align}
\label{eq:line_spectrum}
\mu & = \sum_{f_j \in T } \vct{x}_j \, \delta \brac{f - f_j},
\end{align}
where $\delta \brac{f - f_j}$ denotes a Dirac measure located at $f_j$. Sparse spectra of this form are often called \emph{line spectra} in the literature. Note that a simple change of variable allows to apply this model to signals with spectra restricted to any interval $\sqbr{f_{\min},f_{\max}}$.

By the Nyquist-Shannon sampling theorem we can recover $g$, and consequently $\mu$, from an infinite sequence of regularly spaced samples $\keys{g\brac{l},\; l \in \Z}$ by sinc interpolation. The aim of spectral super-resolution is to estimate the support of the line spectrum $T$ and the amplitude vector $\vct{x}$ from a \emph{finite} set of $n$ contiguous samples instead. Note that $\keys{g\brac{l}, \; l \in \Z}$ are the Fourier-series coefficients of $\mu$, so mathematically we seek to recover an atomic measure from a subset of its Fourier coefficients. As described in the introduction, we are interested in tackling this problem when a subset of the data is completely corrupted. These corruptions are modeled as additive impulsive noise, represented by a sparse vector $\vct{z} \in \C^{n}$ with $s$ nonzero entries. The data are consequently of the form 
\begin{align}
\label{eq:model_data}
\vct{y}_{l} & = g\brac{l} + \vct{z}_{l}, \quad 1 \leq l \leq n.
\end{align}
To represent the measurement model more compactly, we define an operator $\mathcal{F}_{n}$ that maps a measure to its first $n$ Fourier series coefficients,  
\begin{align}
\label{eq:model_matrix_form}
\vct{y} & = \mathcal{F}_{n} \, \mu + \vct{z}.
\end{align}
Intuitively, $\mathcal{F}_{n}$ maps the spectrum $\mu$ to $n$ regularly spaced samples of the signal $g$ in the time domain.

\subsection{Minimum-separation condition}
\label{sec:minimum_separation}

Even in the absence of any noise, the problem of recovering a signal from $n$ samples is vastly underdetermined: we can fill in the missing samples $g\brac{0}, g \brac{-1}, \ldots$ and $g\brac{n+1}, g \brac{n+2}, \ldots$ any way we like and then apply sinc interpolation to obtain an estimate 
that is consistent with the data. For the inverse problem to make sense we need to leverage additional assumptions about the structure of the signal. In spectral super-resolution the usual assumption is that the spectrum of the signal is sparse. This is reminiscent of compressed sensing~\cite{Candes:2006eq}, where signals are recovered robustly from randomized measurements by exploiting a sparsity prior. 

A crucial insight underlying compressed-sensing theory is that the randomized operator obeys the \emph{restricted-isometry property} (RIP), which ensures that the measurements preserve the energy of any sparse signal with high probability~\cite{Candes:2005cs}. Unfortunately, this is not the case for our measurement operator of interest. The reason is that signals consisting of clustered spectral lines may lie almost in the null space of the sampling operator, even if the number of spectral lines is small. Additional conditions beyond sparsity are necessary to ensure that the problem is well posed. To this end, we define the \emph{minimum separation} of the support of a signal, as introduced in~\cite{Candes:2012uf}.

\begin{definition}[Minimum separation] For a set of points $T \subset  \sqbr{0,1}$, the minimum separation (or minimum distance) is defined as the closest distance between any two elements from $T$,
  \begin{equation}
    \label{eq:min_distance}
    \Delta(T) = \inf_{(f_1, f_2) \in T \, : \, f_1 \neq f_2} \, \, |f_2 - f_1|. 
  \end{equation} 
To be clear, this is the wrap-around distance so that the distance between $f_1 = 0$ and $f_2 = 3/4$ is equal to $1/4$.
\end{definition}

If the minimum distance is too small with respect to the number of measurements then it may be impossible to resolve a signal even under very small levels of noise. A fundamental limit in this sense is $\Delta^{\ast} := \frac{2}{n-1}$, which is the width of the main lobe of the periodized sinc kernel that is convolved with the spectrum when we truncate the number of samples to $n$. This limit arises because for minimum separations just below $\Delta^{\ast} / 2$ there exist signals that are \emph{almost} suppressed by the sampling operator $\mathcal{F}_{n}$. If such a signal $d$ corresponds to the difference between two different signals $s_1$ and $s_2$ so that $s_1 - s_2 = d$, it will be very challenging to distinguish $s_1$ and $s_2$ from the available data\footnote{For a concrete example of two signals with a minimum separation of $0.9 \Delta^{\ast}$ that are almost indistinguishable from data consisting of $n = 2 \, 10^{3}$ samples see Figure 2 of~\cite{superres_new}}. This phenomenon can be characterized theoretically in an asymptotic setting using Slepian's prolate-spheroidal sequences~\cite{slepian} (see also Section 3.2 in \cite{Candes:2012uf}). More recently, Theorem 1.3 of~\cite{moitra_superres} provides a non-asymptotic analysis and other works have obtained lower bounds on the minimum separation necessary for convex-programming approaches to succeed~\cite{tang_resolution,peyreduval}.

\subsection{Robust spectral super-resolution via convex programming}
\label{sec:main_result}

Spectral super-resolution in the presence of outliers boils down to estimating $\mu$ and $\vct{z}$ in the mixture model~\eqref{eq:model_matrix_form}. Without additional constraints, this is not very ambitious: data consistency is trivially achieved, for instance, by setting the sines to zero and declaring every sample to be a spike. Our goal is to fit the two components \emph{in the simplest way possible}, i.e. so that the spectrum of the multisinusoidal signal-- the \emph{sines}-- is restricted to a small number of frequencies and the impulsive noise-- the \emph{spikes}-- only affects a small subset of the data. 

Many modern signal-processing methods rely on the design of cost functions that (1) encode prior knowledge about signal structure and (2) can be minimized efficiently. In particular, penalizing the $\ell_1$-norm is an efficient and robust method for obtaining sparse estimates in denoising~\cite{chen2001atomic}, regression~\cite{tibshirani1996regression} and inverse problems such as compressed sensing~\cite{donoho2006compressed,candes2006near}. In order to fit a mixture model where both the spikes and the spectrum of the sines are sparse, we propose minimizing a cost function that penalizes the $\ell_1$ norm of both components (or rather a continuous counterpart of the $\ell_1$ norm in the case of the spectrum, as we explain below). We would like to note that this approach was introduced by some of the authors of the present paper in~\cite{superres_new,tang2014robust}, but without any theoretical analysis, and applied to multiple-target tracking from radar measurements in~\cite{le_paper}. Similar ideas have been previously leveraged to separate low-rank and sparse matrices~\cite{chandrasekaran2011rank,candes2011robust}, perform compressed sensing from corrupted data~\cite{li2013compressed} and demix signals that are sparse in different bases~\cite{mccoy2014sharp}. 


Recall that the spectrum of the sinusoidal component in our mixture model is modeled as a measure that is supported on a continuous interval. Its $\ell_1$-norm is therefore not well defined. In order to promote sparsity in the estimate, we resort instead to a continuous version of the $\ell_1$ norm: the total-variation (TV) norm\footnote{\emph{Total variation} often also refers to the $\ell_1$ norm of the discontinuities of a piecewise-constant function, which is a popular regularizer in image processing and other applications~\cite{tv}.}. If we consider the space of measures supported on the unit interval, this norm is dual to the infinity norm, so that
\begin{align}
\normTV{ \nu } := \sup_{\normInf{h}\leq 1,h \in C\brac{\mathbb{T}}} \op{Re} \sqbr{  \int_{\mathbb{T}}\overline{h \brac{f}} \nu \brac{\text{d}f}}.
\end{align}
for any measure $\nu$ (for a different definition see Section A in the appendix of~\cite{Candes:2012uf}). In the case of a superposition of Dirac deltas as in~\eqref{eq:line_spectrum}, the total-variation norm is equal to the $\ell_1$ norm of the coefficients, i.e. $\normTV{ \mu }=\normOne{ \vct{x}}$. Spectral super-resolution via TV-norm minimization, introduced in~\cite{Candes:2012uf,supportPursuit} (see also~\cite{bredies2013inverse}), has been shown to achieve exact recovery under a minimum separation of $\frac{\mindist}{ n-1 }$ in~\cite{superres_new} and to be robust to missing data in~\cite{tang2012offgrid}. 

Our proposed method minimizes the sum of the $\ell_1$ norm of the spikes and the TV norm of the spectrum of the sines subject to a data-consistency constraint:
\begin{align}
  \min_{ \tilde{\mu}, \vct{ \tilde{z} } } &   \normTV{ \tilde{\mu}} + \lambda \normOne{ \vct{\tilde{z}} } \quad
  \text{subject to} \quad   \ml{F}_n \, \tilde{\mu} + \vct{\tilde{z}} = \vct{y} .  \label{eq:opt_problem}
\end{align}
$\lambda > 0$ is a regularization parameter that governs the weight of each penalty term. This optimization program is convex. Section~\ref{sec:sdp} explains how to solve it by reformulating its dual as a semidefinite program. Our main theoretical result is that solving~\eqref{eq:opt_problem} achieves perfect demixing with high probability under certain assumptions. 


\begin{theorem}[Proof in Section~\ref{sec:proof}]
\label{theorem:main} 
Suppose that we observe $n$ samples of the form
\begin{align}\label{eqn:signalmodel}
\vct{y} & = \mathcal{F}_{n} \, \mu + \vct{z},
\end{align}
where each entry in $\vct{z}$ is nonzero with probability $\frac{s}{n}$ (independently of each other) and the support $T:=\keys{f_1, \ldots, f_k}$ of
\begin{align}
\mu & := \sum_{f_j \in T} \vct{x}_j \delta \brac{f - f_j},
\end{align}
has a minimum separation lower bounded by 
\begin{align}
\label{condition:minimum_separation}
\Delta_{\min} := \frac{\mindist}{ n-1 }.
\end{align}
If the phases of the entries in $\vct{x} \in \C^{k}$ and the nonzero entries in $\vct{z} \in \C^{n}$ are iid random variables uniformly distributed in $\sqbr{0,2\pi}$, then the solution to Problem~\eqref{eq:opt_problem} with $\lambda = 1/\sqrt{n}$ is exactly equal to $\mu$ and $\vct{z}$ with probability $1-\epsilon$ for any $\epsilon>0$ as long as
\begin{align}
k & \leq C_k \brac{\log \frac{n}{\epsilon}}^{-2} n, \label{eq:cond_k}\\
s & \leq C_s \brac{\log \frac{n}{\epsilon}}^{-2} n, \label{eq:cond_s}
\end{align}
for fixed numerical constants $C_k$ and $C_s$ and $n \geq 2 \times 10^3$.
\end{theorem}
The theorem guarantees that our method is able to super-resolve a number of spectral lines that is proportional to the number of measurements, even if the data contain a \emph{constant fraction of outliers}, up to logarithmic factors. The proof is presented in Section~\ref{sec:proof}; it is based on the construction of a random trigonometric polynomial that certifies exact demixing. 
Our result is non-asymptotic and holds with high probability under several assumptions, which we now discuss in more detail. 

\begin{itemize}
\item The support of the sparse corruptions follows a Bernoulli model where each entry is nonzero with probability $s/n$ independently from each other. This model is essentially equivalent to choosing the support of the outliers uniformly at random from all possible subsets of cardinality $s$, as shown in Section 7.1 of~\cite{candes2011robust} (see also~\cite[Section 2.3]{Candes:2006eq} and~\cite[Section 8.1]{Candes:2010jb}). 
\item The phases of the amplitudes of the spectral lines are assumed to be iid uniform random variables (note however that the amplitudes can take any value). Modeling the phase of the spectral components of a multisinusoidal signal in this way is a common assumption in signal processing, see for example \cite[Chapter 4.1]{Stoica:2005wf}. 
\item The phases of the amplitudes of the additive corruptions are also assumed to be iid uniform random variables (the amplitudes can again take any value). If we constrain the corruptions to be real, the derandomization argument in~\cite[Section 2.2]{candes2011robust} allows to obtain guarantees for arbitrary sign patterns.
\item We have already discussed the minimum-separation condition on the spectrum of the multisinusoidal component in Section~\ref{sec:minimum_separation}.
\end{itemize}

Our assumptions model a non-adversarial situation where the outliers are not designed to cancel out the samples from the multisinusoidal signal. In the absence of any such assumption it is possible to concoct instances for which the demixing problem is ill posed, even if the number of spectral lines and outliers is small. We illustrate this with a simple example, based on the \emph{picket-fence} sequence used as an extremal function for signal-decomposition uncertainty principles in~\cite{donoho1989uncertainty,donoho2001uncertainty}. Consider $k'$ spectral lines with unit amplitudes with an equispaced support
\begin{align}
\mu' := \frac{1}{k'}\sum_{j=0}^{k'-1} \delta \brac{f - j /k' }.
\end{align}
The samples of the corresponding multisinusoidal signal $g'$ are zero except at multiples of $k'$
\begin{align}
g'\brac{l} & = 
\begin{cases} 
1 \quad & \text{if $ l/k' \in \Z$}, \\
0 \quad & \text{otherwise}.
\end{cases}
\end{align}
If we choose the corruptions $\vct{z'}$ to cancel out these nonzero samples
\begin{align}
\vct{z'}_{l} & = 
\begin{cases} 
-1 \quad & \text{if $ l/k' \in \Z$}, \\
0 \quad & \text{otherwise},
\end{cases}
\end{align}
then the corresponding measurements are all equal to zero! For these data the demixing problem is obviously impossible to solve by any method. Set $k':= \sqrt{n}$ so that the number of measurements $n$ equals $\brac{k'}^2$. Then the number of outliers is just $n/k' = \sqrt{n}$ and the minimum separation between the spikes is $1/\sqrt{n}$, which amply satisfies the minimum-separation condition~\ref{condition:minimum_separation}. This shows that additional assumptions beyond the minimum-separation condition are necessary for the inverse problem to make sense. A related phenomenon arises in compressed sensing, where random measurement schemes avoid similar adversarial situations (see~\cite[Section 1.3]{Candes:2006eq} and~\cite{tropp2008linear}). An interesting subject for future research is whether it is possible to establish the guarantees for exact demixing provided by Theorem~\ref{theorem:main} without random assumptions on the phase of the different components, or if these assumptions are necessary for the demixing problem to be well posed.

\subsection{Regularization parameter}
\label{sec:reg_param}
A question of practical importance is whether the performance of our demixing method is robust to the choice of the regularization parameter $\lambda$ in Problem~\eqref{eq:opt_problem}. Theorem~\ref{theorem:main} indicates that this is the case in the following sense. If we set $\lambda$ to a fixed value that is proportional to $1/\sqrt{n}$\footnote{To be precise, Theorem~\ref{theorem:main} assumes $\lambda:=1/\sqrt{n}$, but one can check that the whole proof goes through if we set $\lambda$ to $c/\sqrt{n}$ for any positive constant $c$. The only effect is a change in the constants $C_s$ and $C_k$ in~\eqref{eq:cond_k} and~\eqref{eq:cond_s}.}, then exact demixing occurs for a number of spectral lines $k$ and a number of outliers $s$ that range from zero to a certain maximum value proportional to $n$ (up to logarithmic factors). 

In this section we provide additional theoretical evidence for the robustness of our method to the choice of $\lambda$. If exact recovery occurs for a certain pair $\keys{\mu, \vct{z}}$ and a certain $\lambda$ then it will also succeed for \emph{any trimmed version} $\keys{\mu', \vct{z'}}$ (obtained by removing some elements of the support of $\mu$ or $\vct{z}$, or both) for \emph{the same value of $\lambda$}.

\begin{lemma}[Proof in Section~\ref{proof:trimmed_sol}]
\label{lemma:trimmed_sol}
Let $\vct{z}$ be a vector with support $\Omega$ and let $\mu$ be an arbitrary measure such that
\begin{align}
\vct{y} = \mathcal{F}_{n} \, \mu + \vct{z}.
\end{align}
Assume that the pair $\keys{\mu,\vct{z}}$ is the unique solution to Problem~\eqref{eq:opt_problem} and consider the data
\begin{align}
\vct{y'} = \mathcal{F}_{n} \, \mu' + \vct{\vct{z'}}.
\end{align}
$\mu'$ is a trimmed version of $\mu$: it is equal to $\mu$ on a subset of its support $T' \subseteq T$ and is zero everywhere else. Similarly, $\vct{z'}$ equals $\vct{z}$ on a subset of entries $\Omega' \subseteq \Omega$ and is zero otherwise. For any choice of $T'$ and $\Omega'$, $\keys{\mu,\vct{\vct{z'}}}$ is the unique solution to Problem~\eqref{eq:opt_problem} if we set the data vector to equal $\vct{y'}$ for the same value of $\lambda$.
\end{lemma}
This result and its proof are inspired by Theorem 2.2 in~\cite{candes2011robust}. As illustrated by Figures~\ref{fig:phase_transitions_mindist} and~\ref{fig:phase_transitions_lambda}, our numerical experiments corroborate the lemma: we consistently observe that if exact demixing occurs for most signals with a certain number of spectral lines and outliers, then it also occurs for most signals with less spectral lines and less corruptions (as long as the minimum separation is the same) for a fixed value of $\lambda$.

\subsection{Stability to dense perturbations}
\label{sec:dense_noise}
One of the advantages of our optimization-based framework is that we can account for additional assumptions on the problem structure by modifying either the cost function or the constraints of the optimization problem used to perform demixing. In most applications of spectral super-resolution, the data will deviate from the multisinusoidal model~\eqref{eq:multisinusoidal} because of measurement noise and other perturbations, even in the absence of outliers. We model such deviations as a dense additive perturbation $\vct{w}$, such that $\normTwo{ \vct{w}} \leq \sigma$ for a certain noise level $\sigma$,
\begin{align}
\label{eq:dense_noise}
\vct{y} = \mathcal{F}_{n} \, \mu + \vct{z}+ \vct{w}.
\end{align}
Problem~\eqref{eq:opt_problem} can be adapted to this measurement model by relaxing the equality constraint that enforces data consistency to an inequality which takes into account the noise level
\begin{align}
  \min_{ \tilde{\mu}, \vct{ \tilde{z} } } &   \normTV{ \tilde{\mu}} + \lambda \normOne{ \vct{\tilde{z}} } \quad
  \text{ subject to } \quad  \normTwo{ \vct{y} - \ml{F}_n \, \tilde{\mu} + \vct{\tilde{z}} } \leq \sigma.  \label{eq:opt_problem_noise}
\end{align}
Just like Problem~\eqref{eq:opt_problem}, this optimization problem can be solved by recasting its dual as a tractable semidefinite program, as we explain in detail in Section~\ref{sec:sdp_noise}.

\subsection{Atomic-norm denoising}
\label{sec:denoising}
Our demixing method is closely related to atomic-norm denoising of multisinusoidal samples. Consider the $n$-dimensional vector $ \vct{g} := \ml{F}_n \, \mu $ containing \emph{clean} samples from a signal $g$ defined by~\eqref{eq:multisinusoidal}. The assumption that the spectrum $\mu$ of $g$ consists of $k$ spectral lines is equivalent to $ \vct{g}$ having a sparse representation in an infinite dictionary of $n$-dimensional sinusoidal \emph{atoms} $\vct{a} \brac{f, \phi} \in \C^{n}$ parameterized by frequency $f \in [0, 1)$ and phase $\phi \in [0, 2\pi)$, 
\begin{align}
\label{eq:atoms}
  \vct{a} \brac{ f, \phi }_l & := \frac{1}{\sqrt{n}} e^{i\phi}e^{i2\pi l f},  \quad 1 \leq l \leq n.
\end{align}
Indeed, $ \vct{g}$ can be expressed as a linear combination of $k$ atoms
\begin{align}
\label{eq:g}
\vct{g} & = \sqrt{n}\sum_{j=1}^{k} \abs{  \vct{x}_j }   \vct{a} \brac{ f_j, \phi_j }, \quad \vct{x}_j := \abs{\vct{x}_j} e^{i 2 \pi \phi_j}.
\end{align}
This representation can be leveraged in an optimization framework using the atomic norm, an idea introduced in~\cite{chandrasekaran2012convex} and first applied to spectral super-resolution in~\cite{Bhaskar:2012tq}. The atomic norm induced by a set of atoms $\ml{A}$ is equal to the gauge of $\ml{A}$ defined by
\begin{align}
  \atomicnorm{ \vct{u} } & := \inf \left\{ t > 0: \vct{u} \in t \tmop{conv}  \brac{ \mathcal{A} } \right\},
\end{align}
which is a norm as long as $\mathcal{A}$ is centrally symmetric around the origin (as is the case for~\eqref{eq:atoms}). Geometrically, the unit ball of the atomic norm is the convex hull of the atoms in $\mathcal{A}$, just like the $\ell_1$-norm ball is the convex hull of unit-norm one-sparse vectors. As a result, signals consisting of a small number of atoms tend to have a smaller atomic norm (just like sparse vectors tend to have a smaller $\ell_1$-norm). 

Consider the problem of denoising the samples of $g$ from corrupted data of the form~\eqref{eq:model_matrix_form}, 
\begin{align}
\label{eq:data_denoising}
\vct{y} = \vct{g} + \vct{z}.
\end{align}
To be clear, the aim is now to separate $\vct{g}$ from the corruption vector $\vct{z}$ instead of directly estimating the spectrum of $\vct{g}$. In order to demix the two signals we penalize the atomic norm of the multisinusoidal component and the $\ell_1$ norm of the sparse component,
\begin{align}
  \min_{ \vct{\tilde{g}},\vct{\tilde{z}}} &  \;  \frac{1}{\sqrt{n}}\atomicnorm{ \vct{\tilde{g}} } + \lambda \normOne{ \vct{\tilde{z}}} \quad
  \text{subject to } \quad  \vct{\tilde{g}} + \vct{\tilde{z}} = \vct{y} ,  \label{eq:atomic_norm_min}
\end{align}
where $\lambda > 0$ is a regularization parameter. 

Problems~\ref{eq:opt_problem_noise} and~\ref{eq:atomic_norm_min} are closely related. Their convex cost functions are designed to exploit sparsity assumptions on the spectrum of $g$ and on the corruption vector $\vct{z}$ in ways that are essentially equivalent. More formally, both problems have the same dual, as implied by the following lemma and Lemma~\ref{lemma:dual_sdp}. 

\begin{lemma}[Proof in Section~\ref{proof:atomic_norm_dual}]
\label{lemma:atomic_norm_dual}
The dual of Problem~\eqref{eq:atomic_norm_min} is
\begin{align}
\label{eq:dual_atomic_norm}
\max_{ \vct{\eta} \in \C^{n}} \;   \PROD{\vct{y}}{\vct{\eta}}  \quad \text{subject to}
\quad & \normInf{\mathcal{F}_{n}^{\ast} \, \vct{\eta}}  \leq 1, \\
&  \normInf{\vct{\eta}}  \leq \lambda,
\end{align}
where the inner product is defined as $\PROD{ \vct{y}}{ \vct{\eta}} : = \op{Re}\brac{\vct{y}^{\ast}\vct{\eta}}$. 
\end{lemma}

The fact that the two optimization problems share the same dual has an important consequence established in Section~\ref{proof:atomic_denoising}: the same dual certificate can be used to prove that they achieve exact demixing. As a result, the proof of Theorem~\ref{theorem:main} immediately implies that solving Problem~\eqref{eq:atomic_norm_min} is successful in separating $\vct{g}$ and $\vct{z}$ under the conditions described in Section~\ref{sec:main_result}. 

\begin{corollary}[Proof in Section~\ref{proof:atomic_denoising}]
\label{cor:atomic_denoising}
Under the assumptions of Theorem~\ref{theorem:main}, $\vct{g} := \ml{F}_n \, \mu $ and $\vct{z}$ are the unique solutions to Problem~\eqref{eq:atomic_norm_min}.
\end{corollary}

Problem~\eqref{eq:atomic_norm_min} can be adapted to denoise data that is perturbed by both outliers and dense noise, which follows the measurement model~\eqref{eq:dense_noise}. Inspired by previous work on line-spectra denoising via atomic-norm minimization~\cite{Bhaskar:2012tq,tang2014minimax}, we remove the equality constraint and add a regularization term to ensure consistency with the data,
\begin{align}
  \min_{ \vct{\tilde{g}},\vct{\tilde{z}}} &   \;  \frac{1}{\sqrt{n}} \atomicnorm{ \vct{\tilde{g}} } + \lambda \normOne{ \vct{\tilde{z}}} + \frac{\gamma}{2} \normTwo{ \vct{y} - \vct{\tilde{g}} - \vct{\tilde{z}} }^2 ,  \label{eq:atomic_norm_noise}
\end{align}
where $\gamma > 0$ is a regularization parameter with a role analogous to $\sigma$ in Problem~\eqref{eq:opt_problem_noise}. 

In Section~\ref{sec:atomic_norm} we discuss how to implement atomic-norm denoising by reformulating Problems~\ref{eq:atomic_norm_min} and~\ref{eq:atomic_norm_noise} as semidefinite programs.

\subsection{Related work}
\label{sec:related_work}
Most previous works analyzing the problem of demixing sines and spikes make the assumption that the frequencies of the sinusoidal component lie on a grid with step size $1/n$, where $n$ is the number of samples. In that case, demixing reduces to a discrete sparse decomposition problem in a dictionary formed by the concatenation of an identity and a discrete-Fourier-transform matrix~\cite{donoho1989uncertainty}. Bounds on the coherence of this dictionary can be used to derive guarantees for basis pursuit~\cite{donoho2001uncertainty} and also techniques based on Prony's method~\cite{dragotti2014sparse}. Coherence-based bounds do not reflect the fact that most sparse subsets of the dictionary are well conditioned~\cite{tropp2008linear}, which can be exploited to obtain stronger guarantees for $\ell_1$-norm based methods under random assumptions~\cite{li2013compressed,su_corrupted_fourier}. In this paper we depart from this previous literature by considering a sinusoidal component whose spectrum may lie on arbitrary points of the unit interval. 

Our work draws from recent developments on the super-resolution of point sources and line spectra via convex optimization. In \cite{Candes:2012uf} (see also~\cite{supportPursuit}), the authors establish that TV minimization achieves exact recovery of measures satisfying a minimum separation of $\frac{4}{n-1}$, a result that is sharpened to $\frac{\mindist}{n-1}$ in~\cite{superres_new}. In~\cite{tang2012offgrid} the method is adapted to a compressed-sensing setting, where a large fraction of the measurements may be missing. The proof of Theorem~\ref{theorem:main} builds upon the techniques developed in~\cite{superres_new,Candes:2012uf,tang2012offgrid}. We would like to point out that stability guarantees for TV-norm-based approaches established in subsequent works~\cite{support_detection, azais2015spike,superres_noisy, tang2014minimax,peyreduval} hold only for small perturbations and do not apply when the data may be perturbed by sparse noise of arbitrary amplitude, as is the case in this paper. 

In~\cite{Chi2014spectral}, a spectral super-resolution approach based on robust low-rank matrix recovery is shown to be robust to outliers under some incoherence assumptions, which are empirically related to our minimum-separation condition (see Section A in~\cite{Chi2014spectral}). Ignoring logarithmic factors, the guarantees in~\cite{Chi2014spectral} allow for exact denoising of up to $\ml{O}\brac{\sqrt{n}}$ spectral lines in the presence of $\ml{O}\brac{n}$ outliers, where $n$ is the number of measurements. Corollary~\ref{cor:atomic_denoising}, which follows from our main result Theorem~\ref{theorem:main}, establishes that our approach succeeds in denoising up to  $\ml{O}\brac{n}$ spectral lines also in the presence of $\ml{O}\brac{n}$ outliers (again ignoring logarithmic factors). 
In Section~\ref{sec:experiments_denoising} we compare both techniques empirically. Finally, we would like to mention another method exploiting optimization and low-rank matrix structure~\cite{zeng2013_} and an alternative approach to gridless spectral super-resolution~\cite{stoica2011new}, which has been recently adapted to account for missing data and impulsive noise~\cite{yang2015gridless}. In both cases, no theoretical results guaranteeing exact recovery in the presence of outliers are provided.
\section{Proof of Theorem~\ref{theorem:main}} 
\label{sec:proof}

\subsection{Dual polynomial}
\label{sec:optimality}

We prove Theorem~\ref{theorem:main} by constructing a trigonometric polynomial whose existence certifies that solving Problem~\eqref{eq:opt_problem} achieves exact demixing. We refer to this object as a \emph{dual polynomial}, because its vector of coefficients is a solution to the dual of Problem~\eqref{eq:opt_problem}. This vector is known as a \emph{dual certificate} in the compressed-sensing literature~\cite{Candes:2006eq}, . 


\begin{proposition}[Proof in Section~\ref{proof:dual_polynomial}]
\label{proposition:dual_polynomial}
Let $T \subset \sqbr{0,1}$ be the nonzero support of $\mu$ and $\Omega \subset \keys{1,2,\ldots,n}$ the nonzero support of $\vct{z}$. If there exists a trigonometric polynomial of the form
\begin{align}
    Q \brac{ f } & =  \mathcal{F}_n^{\ast} \, \vct{q} \\
    & = \sum_{j =1}^{n} \vct{q}_j \, e^{-i 2 \pi j f}, \label{eqn:dualpoly}
\end{align}
which satisfies
  \begin{alignat}{2}
  &  Q \brac{ f_j }  =  \frac{\vct{x}_j}{\abs{\vct{x}_j}}, \qquad && \forall f_j \in T,  \label{eqn:condition:Q1}\\
  &  \abs{ Q \brac{ f } }  <  1, && \forall f \in T^c, 
    \label{eqn:condition:Q2}\\
  &  \vct{q}_l  = \lambda \frac{\vct{z}_l}{\abs{\vct{z}_l}}, && \forall l \in \Omega, 
     \label{eqn:condition:q1}\\
  &  \abs{ \vct{q}_l }  <  \lambda , && \forall l \in \Omega^c,  \label{eqn:condition:q2}
  \end{alignat}
then $\brac{\mu,\vct{z}}$ is the unique solution to Problem~\ref{eq:opt_problem} as long as $k+s \leq n$.
\end{proposition}
The dual polynomial can be interpreted as a subgradient of the TV norm at the measure $\mu$, in the sense that
\begin{align}
 \normTV{ \mu + \nu} & \geq \normTV{ \mu } +\PROD{Q}{\nu} , \quad \PROD{Q}{\nu} := \op{Re}\sqbr{ \int_{\sqbr{0,1}} \overline{Q \brac{f}} \, \diff{\nu \brac{f}}},
\end{align}
for any measure $\nu$ supported in the unit interval. In addition, weighting the coefficients of $Q$ by $1/\lambda$ yields a subgradient of the $\ell_1$ norm at the vector $\vct{z}$. This means that for any other feasible pair $\brac{ \mu', \vct{z}'}$ such that $\vct{y} = \mathcal{F}_n \,  \mu' + \vct{z}' $
\begin{align}
 \normTV{ \mu' } + \lambda \normOne{ \vct{z}' } & \geq \normTV{ \mu } + \PROD{Q}{\mu' - \mu } + \lambda \normOne{ \vct{z}} +\lambda \PROD{\frac{1}{\lambda} \vct{q}}{\vct{z}' -\vct{z}} \\
 & \geq \normTV{ \mu } + \PROD{\mathcal{F}_n^{\ast} \, \vct{q} }{\mu' - \mu } + \lambda \normOne{ \vct{z}} + \PROD{ \vct{q}}{\vct{z}' -\vct{z}} \\
 & = \normTV{ \mu } + \lambda \normOne{ \vct{z} } + \PROD{ \vct{q}}{\mathcal{F}_n  \brac{\mu' - \mu} + \vct{z}' -\vct{z}}\\
  & = \normTV{ \mu } + \lambda \normOne{ \vct{z} } \quad \text{since $\mathcal{F}_n \,  \mu' + \vct{z}' = \mathcal{F}_n \,  \mu + \vct{z}$}.
\end{align}
The existence of $Q$ thus implies that $\brac{\mu,\vct{z}}$ is a solution to Problem~\ref{eq:opt_problem}. In fact, as stated Proposition~\ref{proposition:dual_polynomial}, it implies that $\brac{\mu,\vct{z}}$ is the unique solution. 
The rest of this section is devoted to showing that a dual polynomial exists with high probability, as formalized by the following proposition. 
\begin{proposition}[Existence of dual polynomial]
\label{prop:dual_pol}
Under the assumptions of Theorem~\ref{theorem:main} there exists a dual polynomial associated to  $\mu$ and $\vct{z}$ with probability at least $1-\epsilon$.
\end{proposition}

In order to simplify notation in the sequel, we define the vectors $\signx \in \C^{k}$ and $\signz \in \C^{s}$ and an integer $m$ such that 
\begin{align}
\signx_j &:= \frac{\vct{x}_j}{\abs{\vct{x}_j}} \quad 1 \leq j \leq k, \\
\signz_l & := \frac{ \vct{z}_l}{\abs{ \vct{z}_l} } \quad l \in \Omega, \\
m & := \begin{cases}
\frac{n-1}{2} \quad & \text{if $n$ is odd,}\\
\frac{n}{2}-1 \quad & \text{if $n$ is even.}
\end{cases}
\end{align} 
Applying a simple change of variable, we express $Q$ as 
\begin{align}
Q\brac{f} & = \sum_{l = -m}^{m} \vct{q}_{l} \, e^{-i2\pi l f}.
\end{align}
In a nutshell, our goal is (1) to construct a polynomial of this form so that $Q$ interpolates $\signx$ on $T$ and $\vct{q}$ interpolates $\signz$ on $\Omega$, and (2) to verify that the magnitude of $Q$ is strictly bounded by one on $T^c$ and the magnitude of $\vct{q}$ is strictly bounded by $\lambda$ on $\Omega^c$.

\subsection{Construction via interpolation}
\label{sec:interpolation}
We now take a brief detour to introduce a basic technique for the construction of dual polynomials. Consider the spectral super-resolution problem when the data are of the form $\vct{\bar{y}} := \ml{F}_n \, \mu$, i.e. when there are no outliers. A simple corollary to Proposition~\ref{proposition:dual_polynomial} is that the existence of a dual polynomial of the form
\begin{align}
\label{eq:Qbar_qbar}
\bar{ Q }\brac{f} & = \sum_{l = -m}^{m} \vct{\bar{q}}_{l} \, e^{-i2\pi l f}
\end{align}
such that
 \begin{alignat}{2}
  &  \bar{ Q } \brac{ f_j }  =  \signx_j, \qquad && \forall f_j \in T,  \label{eqn:condition:Qbar1}\\
  &  \abs{ \bar{ Q } \brac{ f } }  <  1, && \forall f \in T^c, \label{eqn:condition:Qbar2}
   \end{alignat}    
implies that TV-norm minimization achieves exact recovery in the absence of noise. In this section we describe how to construct such a polynomial using interpolation. This technique was introduced in~\cite{Candes:2012uf} to obtain guarantees for super-resolution under a minimum-separation condition. 

The basic idea is to use a kernel $\bar{ K }$ and its derivative $\bar{ K}^{\brac{1}}$ to interpolate $\signx$ while forcing the derivative of the polynomial to equal zero on $T$. Setting the derivative to zero induces a local extremum which ensures that the magnitude of the polynomial stays bounded below one in the vicinity of $T$ (see Figure 11 in~\cite{superres_new} for an illustration). More formally,
\begin{align}
 \bar{ Q } \brac{ f } & :=  \sum_{j =1}^{k} \vct{\bar{\alpha}}_j \, \bar{ K } \brac{ f - f_j } + \kappa \sum_{j =1}^{k} \vct{\bar{\beta}}_j \, \bar{ K}^{\brac{1}}  \brac{ f - f_j },
\end{align}
where
\begin{align}
\label{eq:kappa}
\kappa := \frac{1}{\sqrt{ \abs{ \bar{K}^{\brac{2}} \brac{0}}}}
\end{align}
is the value of the second derivative of the kernel at the origin. This quantity will appear often in the proof to simplify notation. $\vct{\bar{\alpha}} \in \C^{k}$ and $\vct{\bar{\beta}} \in \C^{k}$ are coefficient vectors set so that 
\begin{alignat}{2}
\bar{ Q } \brac{ f_j } &= 
\signx_j, \quad   && f_j \in T,   \label{eq:interp1_Qbar}\\
\bar{ Q}_R^{\brac{1}}\brac{ f_j } + i \, \bar{ Q }_I^{\brac{1}}\brac{ f_j } & =  
0, \quad  && f_j \in T, \label{eq:interp2_Qbar}
\end{alignat}
where $\bar{ Q}_R^{\brac{1}}$ denotes the real part of $\bar{Q}^{\brac{1}}$ and $\bar{ Q}_I^{\brac{1}}$ the imaginary part. In matrix form, $\vct{\bar{\alpha}}$ and $\vct{\bar{\beta}} $ are the solution to the system
\begingroup
\renewcommand*{\arraystretch}{1.5}
\begin{align}
\MAT{ \bar{D}_0 & \bar{D}_1\\ -\bar{D}_1 & \bar{D}_2 } \MAT{ \vct{\bar{\alpha}} \\ \vct{\bar{\beta}} } & = \MAT{\signx \\ \vct{0} }
\end{align}
\endgroup
where
\begin{align}
\brac{\bar{D}_0}_{jl} =\bar{K} \brac{f_j - f_l}, \quad \brac{\bar{D}_1}_{jl} =
\kappa \, \bar{K} ^{\brac{1}}\brac{f_j - f_l}, \quad \brac{\bar{D}_2}_{jl} = -\kappa^2 \, \bar{K} ^{\brac{2}}\brac{f_j - f_l}.
\end{align}
In~\cite{Candes:2012uf} $\bar{Q}$ is shown to be a valid dual polynomial for a minimum separation equal to $\frac{4}{n-1}$ when the interpolation kernel is a squared Fej\'er kernel. The required minimum separation is sharpened to $\frac{\mindist}{n-1}$ in~\cite{superres_new} by using a different kernel, which will be our choice for $\bar{K}$ in this paper. Consider the Dirichlet kernel of order $\tilde{m} >0$
\begin{align}
\ml{D}_{\tilde{m}} \brac{f} := \frac{1}{2 \, \tilde{m}+1} \sum_{l= -\tilde{m}}^{\tilde{m}}  e^{i2\pi l f} 
  =
  \begin{cases}
   1 & \text{if } f=0\\
   \frac{\sin \brac{\brac{2 \, \tilde{m} +1} \pi f}}{\brac{2 \, \tilde{m}+1}\sin \brac{\pi f}} & \text{otherwise } 
  \text{.}
  \end{cases}
\label{eq:Dirichlet_kernel} 
\end{align}
Following~\cite{superres_new}, we define $\bar{K}$ as the product of three different Dirichlet kernels with different orders
\begin{align}
\bar{ K } \brac{f} & := \ml{D}_{\gammaOne \, m} \brac{ f} \ml{D}_{\gammaTwo \, m} \brac{ f} \ml{D}_{\gammaThree \, m} \brac{ f}\\
&  = \sum_{l = -m}^{m} \vct{c}_l \, e^{i2\pi l f} 
\label{eq:kbar} 
\end{align}
where $\vct{c} \in \C^{n}$ is the convolution of the Fourier coefficients of the three Dirichlet kernels. The choice of the width of the three kernels might seem rather arbitrary; it is chosen to optimize the bound on the minimum separation by achieving a good tradeoff between the  \emph{spikiness} of $\bar{K}$ in the vicinity of the origin and its asymptotic decay~\cite{superres_new}. For simplicity we assume that $\gammaOne \, m$, $\gammaTwo \, m$ and $\gammaThree \, m$ are all integers.\footnote{To avoid this assumption one can adapt the width of the three kernels so that the length of their convolution equals $2m$ and then recompute the bounds that we borrow from~\cite{superres_new}.} Figure~\ref{fig:kernel} shows $\bar{K}$ and its first derivative.

We end the section with two lemmas bounding $\kappa$ and the magnitude of the coefficients of $\vct{q}$, which will be useful at different points of the proof. 
\begin{lemma}
\label{lemma:bound_kappa}
If $m \geq 10^3$, the constant $\kappa$, defined by~\eqref{eq:kappa}, satisfies
\begin{align}
\frac{0.467}{ m} \leq \kappa \leq \frac{0.468}{ m}.
\end{align}
\end{lemma}
\begin{proof}
The bound follows from the fact that $\ml{D}_{\tilde{m}}^{\brac{2}} \brac{0} := -4 \pi^2 \tilde{m} \brac{1+\tilde{m}}/3$ and equation (C.19) in~\cite{superres_new} (see also Lemma 4.8 in~\cite{superres_new}).
\end{proof}
\begin{lemma}[Proof in Section~\ref{proof:c_amp}]
\label{lemma:c_amp}
The coefficients of $\bar{ K } $ satisfy
\begin{align}
\normInf{ \vct{c} } & \leq \frac{1.3}{m}.
\end{align}
\end{lemma}

\subsection{Interpolation with a random kernel}
\label{sec:random_interp}

The trigonometric polynomial $\bar{Q}$ defined in the previous section is not a valid certificate when outliers are present in the data; it does not satisfy~\eqref{eqn:condition:q1} and \eqref{eqn:condition:q2}. In order to adapt the construction so that it meets these conditions we draw upon techniques developed in~\cite{tang2012offgrid}, which studies spectral super-resolution in a compressed-sensing scenario where a subset $\ml{S}$ of the samples is missing. To prove that TV-norm minimization succeeds in such a scenario, the authors of~\cite{tang2012offgrid} construct a bounded polynomial with coefficients restricted to the complement of $\ml{S}$, which interpolates the sign pattern of the line spectra on their support. This is achieved by using an interpolation kernel with coefficients supported on $\ml{S}^c$.  

We denote our dual-polynomial candidate by $Q$. Let us begin by decomposing $Q$ into two components 
\begin{align}
\label{eq:Q_comps}
  Q \brac{ f } & :=  Q_{\op{aux}} \brac{ f } +  R\brac{f},
\end{align}
such that the coefficients of the first component are restricted to $\Omega^c$ ,
\begin{align}
\label{eq:Q_aux_def}
Q_{\op{aux}} \brac{ f } & :=   \sum_{l \in \Omega^c } \vct{q}_{l}\, e^{- i 2 \pi l f},
\end{align}
and the coefficients of the second component are restricted to $\Omega$ and \emph{fixed to equal} $\lambda \signz$ (recall that $\lambda = 1/\sqrt{n}$),
\begin{align}
R \brac{ f } 
& := \frac{1}{\sqrt{n}} \sum_{l \in \Omega}  \signz_{l} \, e^{- i 2 \pi l f}. \label{eqn:randompoly}
\end{align}
This immediately guarantees that $Q$ satisfies~\eqref{eqn:condition:q1}. Now our task is to construct $Q_{\op{aux}} $ so that $Q$ also meets the rest of conditions in Proposition~\ref{proposition:dual_polynomial}. 

Following the interpolation technique described in Section~\ref{sec:interpolation}, we constrain $Q$ to interpolate $\signx$ and have zero derivative in $T$,
\begin{alignat}{2}
Q\brac{ f_j } &= \signx_j, \quad   && f_j \in T,   \label{eq:interp1_Q}\\
 Q_R^{\brac{1}}\brac{ f_j } + i \, Q_I^{\brac{1}}\brac{ f_j } & = 0, \quad  && f_j \in T .  \label{eq:interp2_Q}
\end{alignat}
Given that $R$ is fixed, this is equivalent to
\begin{alignat}{2}
Q_{\op{aux}}\brac{ f_j } &= \signx_j - R\brac{f_j}, \quad   && f_j \in T,   \label{eq:interp1_aux}\\
\brac{Q_{\op{aux}}}_R^{\brac{1}}\brac{ f_j } + i \, \brac{Q_{\op{aux}}}_I^{\brac{1}}\brac{ f_j } & =  - R_R^{\brac{1}}\brac{ f_j } - i\,  R_I^{\brac{1}}\brac{ f_j } , \quad  && f_j \in T ,  \label{eq:interp2_aux}
\end{alignat}
where the subscript $R$ indicates the real part of a number and the subscript $I$ the imaginary part. This interpolation problem is very similar to the one that arises in compressed sensing off the grid~\cite{tang2012offgrid}: we must interpolate a certain vector with a polynomial whose coefficients are restricted to a certain subset, in our case $\Omega^c$. Following~\cite{tang2012offgrid} we employ an interpolation kernel $K$ obtained by selecting the coefficients of $\bar{K}$ in $\Omega^c$,
\begin{align}
 K \brac{f} & := \sum_{l \in \Omega^c} \vct{c}_l \, e^{i2\pi l f} \\
 & =  \sum_{l = -m}^{m} \delta_{\Omega^c} \brac{l} \vct{c}_l \, e^{i2\pi l f},
\label{eq:k} 
\end{align}
where $\delta_{\Omega^c}$ is an indicator random variable that is equal to one if $l \in \Omega^c$ and to zero otherwise. Under the assumptions of Theorem~\ref{theorem:main} these are independent Bernoulli random variables with parameter $\frac{n-s}{n}$, so that the mean of $K$ is equal to a scaled version of $\bar{K}$, 
\begin{align}
\E \brac{ K \brac{f} } & := \frac{n-s}{n}\sum_{l = -m}^{m}  \vct{c}_l \, e^{i2\pi l f} \\
& = \frac{n-s}{n} \bar{K} \brac{f}.
\end{align}
$K$ and its derivatives concentrate around $\bar{K}$ and its derivatives (scaled by $\frac{n-s}{n}$) near the origin, but they don't display the same asymptotic decay. This is illustrated in Figure~\ref{fig:kernel}. 

\begin{figure}
\begin{tabular}{ >{\centering\arraybackslash}m{0.14\linewidth} >{\centering\arraybackslash}m{0.4\linewidth}   >{\centering\arraybackslash}m{0.4\linewidth} }
&
\begin{tikzpicture}[scale=0.9]
\begin{axis}[xlabel={$m \, f $},ytick=\empty, legend pos=north east]
\addplot[very thick,blue] file {K_dat.dat};
\addplot[very thick,red,dashed] file {K_rand_real_dat.dat};
\addplot[very thick,green!60!black,dashed] file {K_rand_imag_dat.dat};
\addlegendentry{$\frac{n-s}{n}\bar{K}$}
\addlegendentry{$K$ (real)}
\addlegendentry{$K$ (imag.)}
\end{axis}
\end{tikzpicture}  
&
\vspace{0.2cm}
\begin{tikzpicture}[scale=0.9]
\begin{axis}[xlabel={$m \, f $}, ytick=\empty, legend pos=north east]
\addplot[very thick,blue] file {K_der_dat.dat};
\addplot[very thick,red,dashed] file {K_der_rand_real_dat.dat};
\addplot[very thick,green!60!black,dashed] file {K_der_rand_imag_dat.dat};
\addlegendentry{$\frac{n-s}{n}\bar{K}^{\brac{1}}$}
\addlegendentry{$K^{\brac{1}}$ (real)}
\addlegendentry{$K^{\brac{1}}$ (imag.)}
\end{axis}
\end{tikzpicture}
\\
&
\vspace{0.2cm}
 \includegraphics{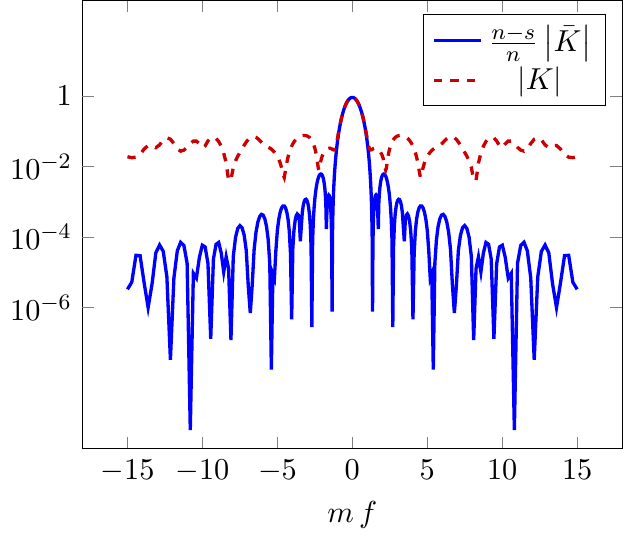}
&
\vspace{0.2cm} \hspace{-1cm}
\includegraphics{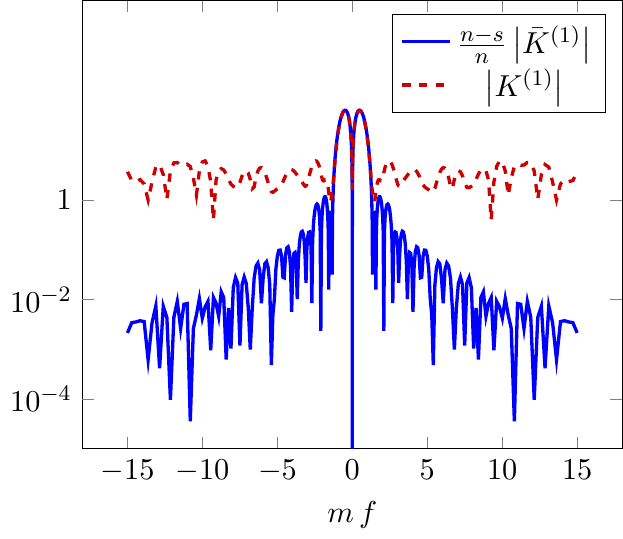}
\\
Frequency coefficients (magnitude)
&
\vspace{0.2cm}
\includegraphics{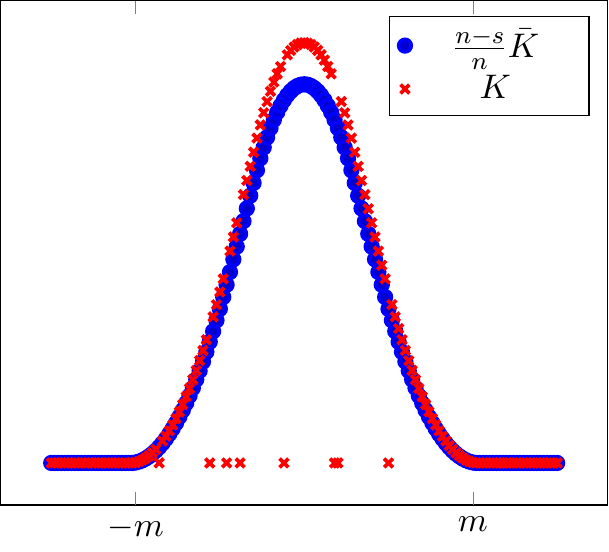}
&
\vspace{0.2cm}
\includegraphics{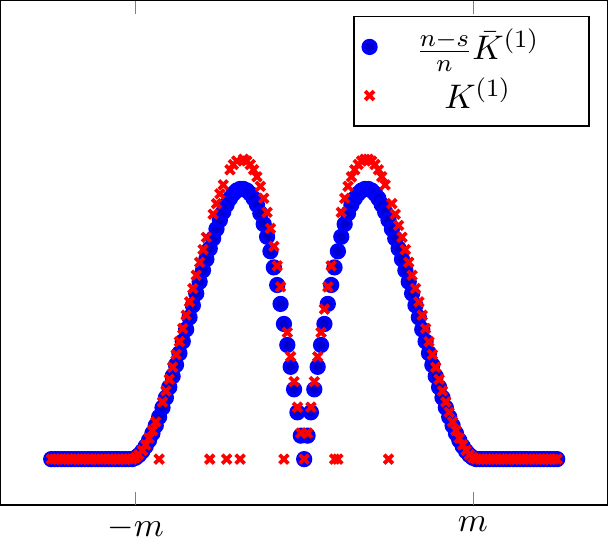}
\end{tabular}
 \caption{The top row shows the interpolating kernel $K$ and $K^{\brac{1}}$ compared to a scaled version of $\bar{K}$ and $\bar{K}^{\brac{1}}$. In the second row we see the asymptotic decay of the magnitudes of both kernels and their derivatives. The left image in the bottom row illustrates the construction of $K$: the Fourier coefficients $\vct{c}$ of $\bar{K}$ that lie in $\Omega$ are set to zero. On the right we can see the Fourier coefficients of $K^{\brac{1}}$ and a scaled version of $\bar{K}^{\brac{1}}$.}
\label{fig:kernel}
\end{figure}
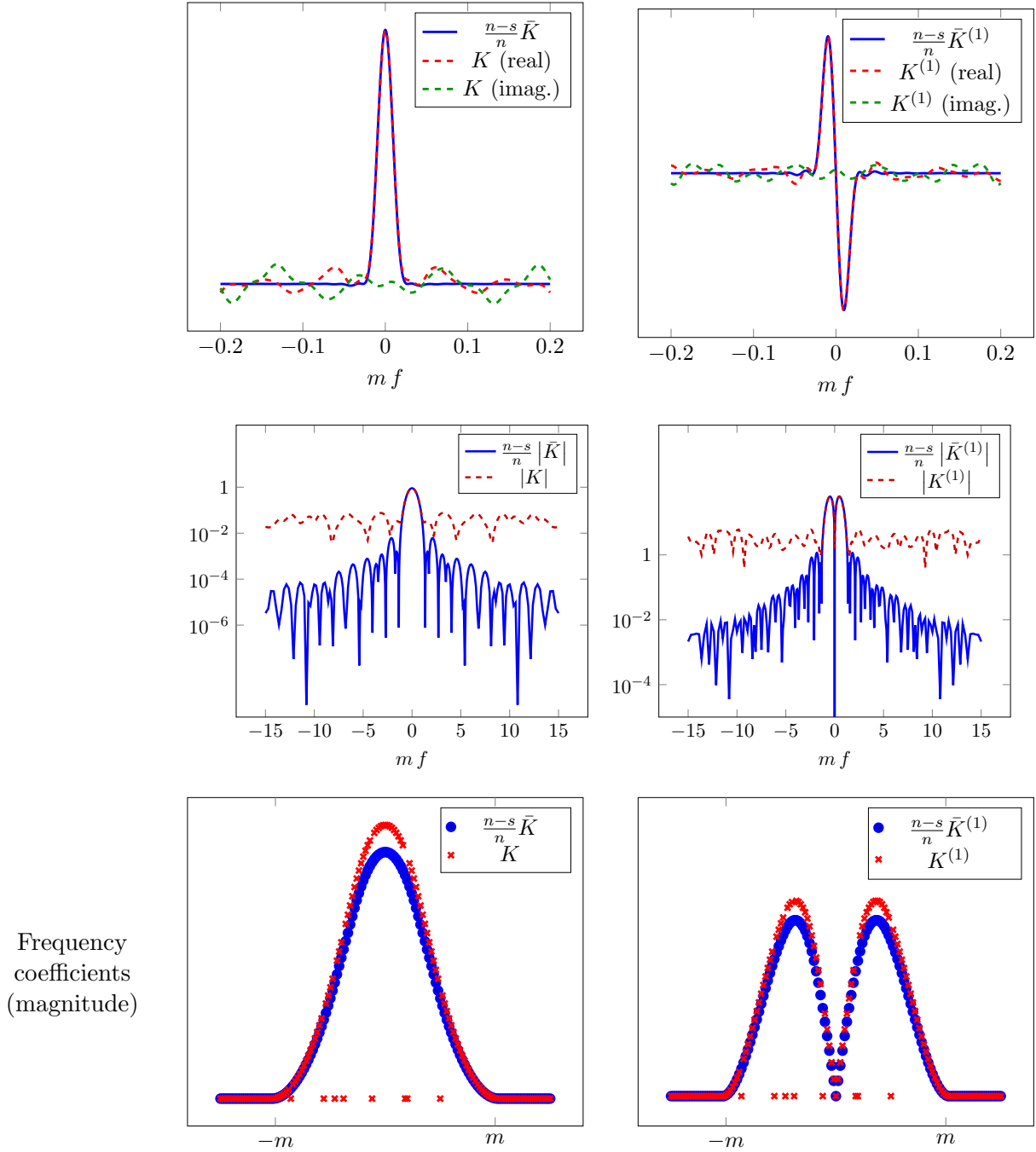

Using $K$ and its first derivative $K^{\brac{1}}$ to construct $Q_{\op{aux}}$ ensures that its nonzero coefficients are restricted to $\Omega^c$. In more detail, $Q_{\op{aux}}$ is a linear combination of shifted and scaled copies of $K$ and $K^{\brac{1}}$,
\begin{align}
\label{eq:Q_aux}
Q_{\op{aux}} \brac{f} & := \sum_{j =1}^{k} \vct{\alpha}_j \, K  \brac{ f-f_j } + \kappa \, \vct{\beta}_j \, K^{\brac{1}}  \brac{f- f_j },
\end{align}
where $\vct{\alpha} \in \C^{k}$ and $\vct{\beta} \in \C^{k}$ are chosen to satisfy~\eqref{eq:interp1_aux} and \eqref{eq:interp2_aux}. 
The corresponding system of equations~\eqref{eq:interp1_aux} and \eqref{eq:interp2_aux} can be recast in matrix form:
\begin{align}
\label{eq:alpha_beta}
\MAT{ D_0 & D_1\\ -D_1 & D_2 } \MAT{ \vct{\alpha} \\
  \vct{\beta} } =\MAT{  \signx  \\ 0 } - \frac{1}{\sqrt{n}} B_{\Omega} \, \signz,
\end{align}
where
\begin{align}
\brac{D_0}_{jl} =K\brac{f_j - f_l}, \quad \brac{D_1}_{jl} =
\kappa \, K^{\brac{1}}\brac{f_j - f_l}, \quad \brac{D_2}_{jl} = -\kappa^2 \, K^{\brac{2}}\brac{f_j - f_l}.
\end{align}
Note that we have expressed the values of $R$ and $R^{\brac{1}}$ in $T$ in terms of $\signz$,  
\begingroup
\renewcommand*{\arraystretch}{1.5}
\begin{align}
\frac{1}{\sqrt{n}}B_{\Omega} \, \signz = \MAT{ R \brac{ f_{1} } & R \brac{ f_{2} }  & \cdots & R \brac{ f_{k} } & -\kappa \,R^{\brac{1}} \brac{ f_{1} } & -\kappa \,R^{\brac{1}} \brac{ f_{2} }  & \cdots & -\kappa \,R^{\brac{1}} \brac{ f_{k} } }^T,
\end{align}
\endgroup
where   
\begin{align}
\vct{b} \brac{ l } & :=  
\MAT{ e^{-i 2 \pi l f_1} &  e^{-i 2 \pi l f_2} & \cdots &  e^{-i 2 \pi l f_{k}} &  i 2 \pi l \kappa \, e^{-i 2 \pi l f_1} & \cdots & i 2 \pi l \kappa \, e^{-i 2 \pi l f_{k}} }^T,  \label{eq:b}\\
B_{\Omega} &:= \MAT{ \vct{b}\brac{i_1} & \vct{b}\brac{i_2} & \cdots & \vct{b}\brac{i_s} }, \quad \Omega = \keys{i_1, i_2, \ldots i_s}.
\end{align}
Solving this system of equations yields $\vct{\alpha}$ and $\vct{\beta}$  and fixes the dual-polynomial candidate, 
\begin{align}
Q \brac{ f } & := \sum_{j =1}^{k} \vct{\alpha}_j \, K  \brac{ f - f_j } + \kappa \sum_{j =1}^{k} \vct{\beta}_j \, K^{\brac{1}}  \brac{ f - f_j } +  R \brac{ f } \\
  & = \vct{v_{0}} \brac{f}^TD^{-1} \brac{\MAT{  \signx  \\ 0 } - \frac{1}{\sqrt{n}} B_{\Omega} \, \signz }  + R\brac{ f },
  \label{eq:Q_vl_Dinv_ur}
    \end{align}
where we define
 \begin{align}
\vct{v_{\ell}} \brac{ f } & :=\kappa^\ell \MAT{
    K^{\brac{ \ell }} \brac{ f - f_1 } & \cdots & K^{\brac{ \ell }} \brac{ f - f_k } & \kappa \,  K^{(\ell+1)} \brac{ f - f_1 } &
  \cdots & \kappa \,  K^{(\ell+1)} \brac{ f - f_k } }^T \notag
\end{align}
 for $\ell=0,1,2, \ldots$ In the next section we establish that a polynomial of this form is guaranteed to be a valid certificate with high probability. Figure~\ref{fig:dual_polynomial} illustrates our construction for a specific example (note that for ease of visualization $\signx$ is real instead of complex). 

\begin{figure}
\begin{tabular}{ >{\centering\arraybackslash}m{0.14\linewidth} >{\centering\arraybackslash}m{0.4\linewidth}   >{\centering\arraybackslash}m{0.4\linewidth} }
  & & Spectrum (magnitude) \\ \\
$R\brac{f}$ & 
\includegraphics{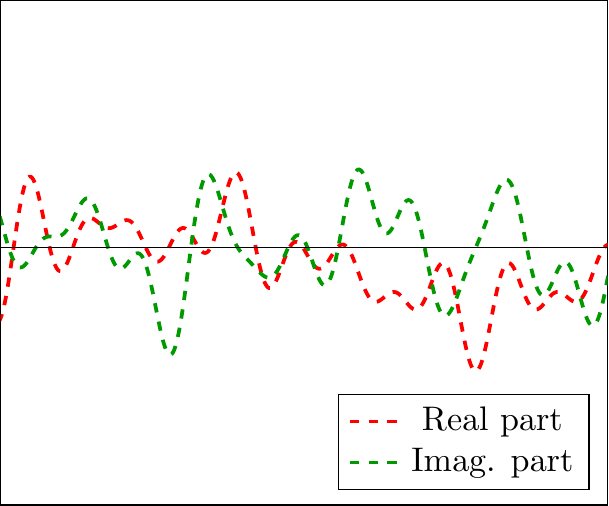}
& 
\includegraphics{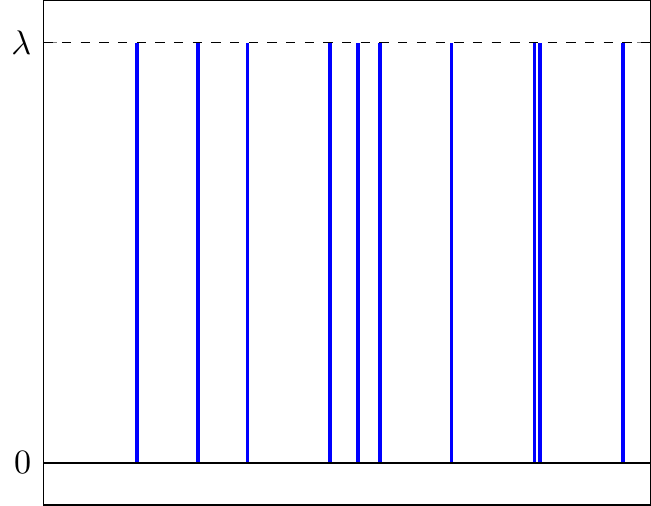}
\\ \\
$Q_{\op{aux}}\brac{f}$ & 
\includegraphics{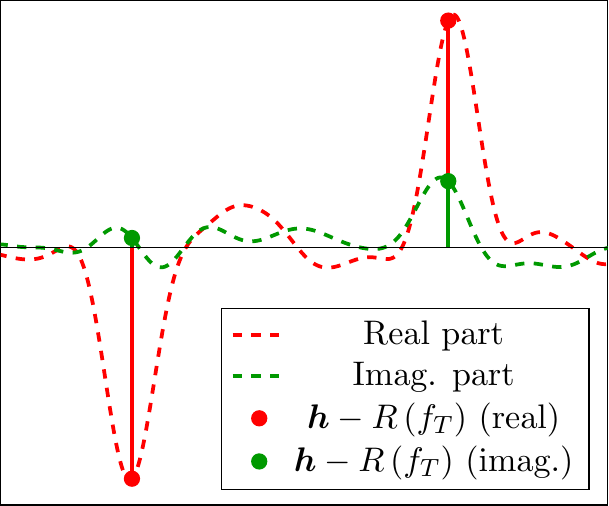}
& 
\includegraphics{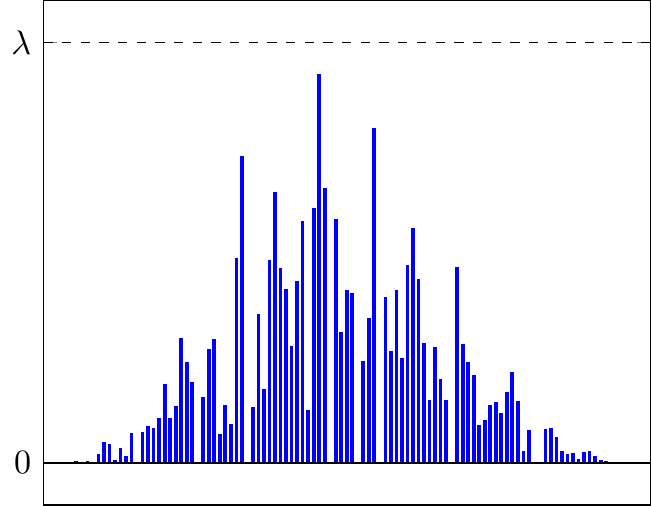}
\\ \\
$Q\brac{f}$ &
\includegraphics{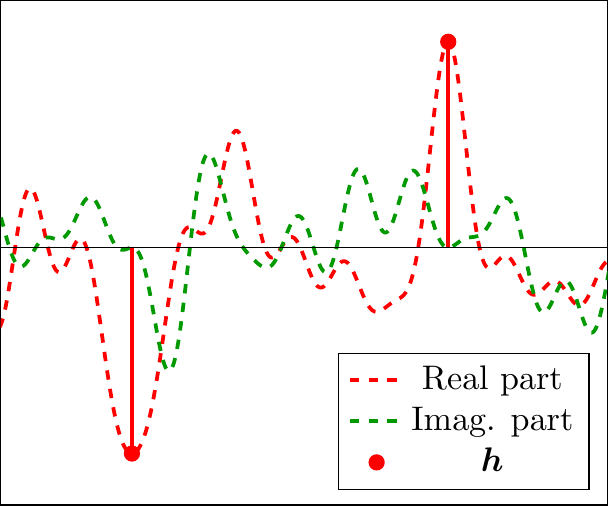}
&
\includegraphics{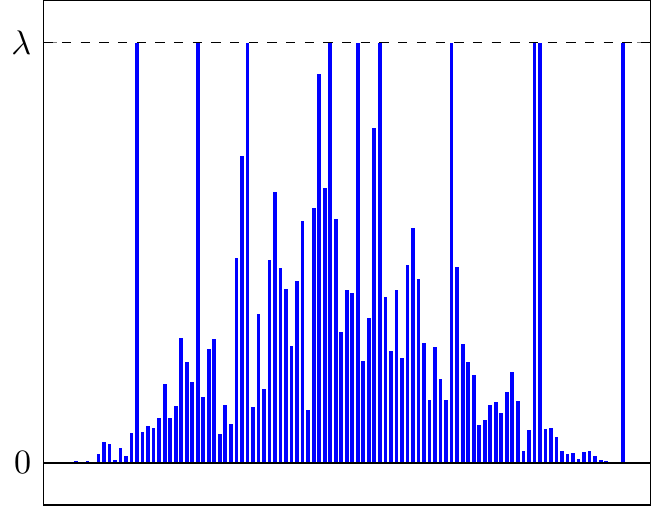}
\end{tabular}
\caption{Illustration of our construction of a dual-polynomial candidate $Q$. The first row shows $R$, the component that results from fixing the coefficients of $Q$ in $\Omega$ to equal $\signz$. The second row shows $Q_{\op{aux}}$, the component built to ensure that $Q$ interpolates $\signx$ by correcting for the presence of $R$. On the right image of the second row, we see that the coefficients of $Q_{\op{aux}}$ are indeed restricted to $\Omega^c$. Finally, the last row shows that $Q$ satisfies all of the conditions in Proposition~\ref{proposition:dual_polynomial}.}
\label{fig:dual_polynomial}
\end{figure}

Before ending this section, we record three useful lemmas concerning $\vct{b}$, $B_{\Omega}$ and $\vct{v_{\ell}}$. The first bounds the $\ell_2$ norm of $\vct{b}$.
\begin{lemma}
\label{lemma:bound_b}
If $m \geq 10^3$, for $ -m \leq l \leq m$
\begin{align}
\normTwo{\vct{b} \brac{ l }}^2 \leq 10 \, k .
\end{align}
\end{lemma}
\begin{proof}
\begin{align}
\normTwo{\vct{b} \brac{ l }}^2 \leq k \brac{1 + \max_{-m \leq l \leq m} \brac{2 \pi l \kappa}^2 } 
& \leq 9.65 \, k \quad \text{by Lemma~\ref{lemma:bound_kappa}.}
\end{align}
\end{proof}
The second yields a bound on the operator norm of $B_{\Omega}$ that holds with high probability.
\begin{lemma}[Proof in Section~\ref{proof:boundB}]
\label{lemma:boundB}
Under the assumptions of Theorem~\ref{theorem:main}, the event
  \begin{align}
  \ml{E}_{B} & := \keys{ \norm{ B_{\Omega} } > C_{B} \brac{\log \frac{n}{\epsilon} }^{-\frac{1}{2} } \sqrt{ n }},
  \end{align}
where $C_{B}$ is a numerical constant defined by~\eqref{eq:C_B_def}, occurs with probability at most $\epsilon / 5$. 
\end{lemma}
The third allows to control the behavior of $\vct{v_{\ell}}$, establishing that it does not deviate much from
  \begin{align}
  \vct{\bar{v}_{\ell}} \brac{ f } & := \kappa^\ell \MAT{
    \bar{K}^{\brac{ \ell }} \brac{ f - f_1 } & \cdots & \bar{K}^{\brac{ \ell }} \brac{ f - f_k } & \kappa \,  \bar{K}^{(\ell+1)} \brac{ f - f_1 } &
  \cdots & \kappa \,  \bar{K}^{(\ell+1)} \brac{ f - f_k } }^T \notag
  \end{align}
on a fine grid with high probability. 
\begin{lemma}[Proof in Section~\ref{proof:v_vbar}]
\label{lemma:v_vbar}
Let $\ml{G} \subseteq \sqbr{0,1}$ be an equispaced grid with cardinality $ 400 \, n^2$. Under the assumptions of Theorem~\ref{theorem:main}, the event
\begin{align}
\ml{E}_{v} := \keys{ \normTwo{  \vct{v_{\ell}} \brac{f} - \frac{n-s}{n} \vct{\bar{v}_{\ell}} \brac{f} }  > C_{\vct{v}} \brac{\log \frac{n}{\epsilon} }^{-\frac{1}{2} }, \quad \text{for all $f \in \ml{G}$ and $\ell \in \keys{0,1,2,3}$} },
\end{align}
where $C_{\vct{v}}$ is a numerical constant defined by~\eqref{eq:C_v_def}, has probability bounded by $\epsilon / 5$.
\end{lemma}

\subsection{Proof of Proposition~\ref{prop:dual_pol}}
\label{proof:dual_pol}
This section summarizes the remaining steps to establish that our proposed construction yields a valid certificate. A detailed description of each step is included in the appendix. First, we show that the system of equations~\eqref{eq:alpha_beta} has a unique solution with high probability, so that $Q$ is well defined. To alleviate notation, let
\begingroup
\renewcommand*{\arraystretch}{1.5}
\begin{align}
D & := \MAT{ D_0 & D_1\\ -D_1 & D_2 }, \qquad \bar{D} := \MAT{  \bar{D}_0 & \bar{D}_1\\ -\bar{D}_1 & \bar{D}_2 }.
\end{align}
\endgroup
The following result shows that $D$ concentrates around a scaled version of $\bar{D}$. As a result, it is invertible and we can bound the operator norm of its inverse leveraging results from~\cite{superres_new}. 
\begin{lemma}[Proof in Section~\ref{proof:D_bounds}]
\label{lemma:D_bounds}
Under the assumptions of Theorem~\ref{theorem:main}, the event
  \begin{align}
  \ml{E}_{D} & := \keys{ \norm{  D - \frac{n-s}{n} \bar{D} } \geq \frac{n-s}{4n} \min \keys{ 1, \frac{C_{D}}{4} \brac{\log \frac{n}{\epsilon} }^{-\frac{1}{2} }} }
  \end{align}
occurs with probability at most $\epsilon / 5$.

In addition, within the event $\ml{E}_{D}^c$, $D$ is invertible and
\begin{align}
\norm{D^{-1}} &  \leq 8, \\
\norm{ D^{-1} - \frac{n}{n-s} \bar{D}^{-1}} & \leq C_{D} \brac{\log \frac{n}{\epsilon} }^{-\frac{1}{2} }, 
\end{align}
where $C_{D}$ is a numerical constant defined by~\eqref{eq:C_D_def}. 
\end{lemma}
An immediate consequence of the lemma is that there exists a solution to the system~\eqref{eq:alpha_beta} and therefore~\eqref{eqn:condition:Q1} holds as long as $\ml{E}_{D}^c$ occurs.
\begin{corollary}
\label{cor:invertible}
In $\ml{E}_{D}^c$ $Q$ is well defined and $Q\brac{ f_j } = \signx_j$ for all $f_j \in T$.
\end{corollary}

All that remains is to establish that $Q$ meets conditions~\eqref{eqn:condition:Q2} and~\eqref{eqn:condition:q2}; recall that~\eqref{eqn:condition:q1} is satisfied by construction. 


To prove~\eqref{eqn:condition:Q2}, we apply a technique from~\cite{tang2012offgrid}. We first show that $Q$ and its derivatives concentrate
around $\bar{Q}$ and its derivatives respectively on a fine grid. Then we leverage Bernstein's inequality to demonstrate that both polynomials and their respective derivatives are close on the whole unit interval. Finally, we borrow some bounds on $\bar{Q}$ and its second derivative from~\cite{superres_new} to complete the proof. The details can be found in Section~\ref{proof:Qbound1} of the appendix.

\begin{proposition}[Proof in Section~\ref{proof:Qbound1}]
\label{proposition:Qbound1}
Conditioned on $\ml{E}_{B}^{c} \cap \ml{E}_{D}^{c} \cap \ml{E}_{v}^{c}$
\begin{align}
\label{eq:bound_Q_Tc}
\abs{ Q\brac{f} } < 1 \quad \text{for all } f \in T^c
\end{align}
with probability at least $1-\epsilon/5$ under the assumptions of Theorem~\ref{theorem:main}.
\end{proposition}


Finally, the following proposition establishes that the remaining condition~\eqref{eqn:condition:q2} holds in $\ml{E}_{B}^{c} \cap \ml{E}_{D}^{c} \cap \ml{E}_{v}^{c}$ with high probability. The proof uses Hoeffding's inequality combined with Lemmas~\ref{lemma:D_bounds} and~\ref{lemma:boundB} to control the magnitude of the coefficients of $\vct{q}$.
\begin{proposition}[Proof in Section~\ref{proof:qboundgamma}]
\label{proposition:qboundgamma}
Conditioned on $\ml{E}_{B}^{c} \cap \ml{E}_{D}^{c} \cap \ml{E}_{v}^{c}$
\begin{align}
\label{eq:bound_q_Omegac}
\abs{\vct{q}_{l}} & < \frac{1}{\sqrt{n}} , \quad \text{for all } l \in \Omega^c,
\end{align}
with probability at least $1-\epsilon/5$ under the assumptions of Theorem~\ref{theorem:main}.
\end{proposition}

Now, to complete the proof, let us define $\ml{E}_{Q}$ to be the event that~\eqref{eqn:condition:Q2} holds and $\ml{E}_{q}$ the event that~\eqref{eqn:condition:q2} holds. Applying De Morgan's laws, the union bound and the fact that for any pair of events $\ml{E}_A$ and $\ml{E}_B$ 
\begin{align} 
\Pr\brac{ \ml{E}_A } \leq \Pr\brac{ \ml{E}_A | \ml{E}_B^c} + \Pr\brac{ \ml{E}_B }.
\end{align}
we have
\begin{align}
\Pr\brac{ \brac{\ml{E}_{Q} \cap \ml{E}_{q}}^c} & = \Pr\brac{ \ml{E}_{Q}^c \cup \ml{E}_{q}^c } \\
& \leq \Pr\brac{ \ml{E}_{Q}^c \cup \ml{E}_{q}^c  | \ml{E}_{B}^c \cap \ml{E}_{D}^c \cap \ml{E}_{v}^c } + \Pr \brac{\ml{E}_{B} \cup \ml{E}_{D} \cup \ml{E}_{v}} \\
& \leq \Pr\brac{ \ml{E}_{Q}^c  | \ml{E}_{B}^c \cap \ml{E}_{D}^c \cap \ml{E}_{v}^c } + \Pr\brac{ \ml{E}_{q}^c  | \ml{E}_{B}^c \cap \ml{E}_{D}^c \cap \ml{E}_{v}^c }+ \Pr \brac{\ml{E}_{B} } + \Pr \brac{\ml{E}_{D} } + \Pr \brac{  \ml{E}_{v}} \\
& \leq \epsilon
\end{align}
by Lemmas~\ref{lemma:boundB}, \ref{lemma:v_vbar} and~\ref{lemma:D_bounds} and Propositions~\ref{proposition:Qbound1} and~\ref{proposition:qboundgamma}. We conclude that our construction yields a valid certificate with probability at least $1-\epsilon$.


\section{Algorithms}
\label{sec:algorithms}
In this section we discuss how to implement the techniques described in Section~\ref{sec:main_results}. In addition, we introduce a greedy demixing method which yields good empirical results. Matlab code implementing all the algorithms presented below is available online\footnote{\url{http://www.cims.nyu.edu/~cfgranda/scripts/spectral_superres_outliers.zip}}. The code allows to reproduce the figures in this section, which illustrate the performance of the different approaches through a running example. 
\subsection{Demixing via semidefinite programming}
\label{sec:sdp}
The main obstacle to solving Problem~\eqref{eq:opt_problem} is that the primal variable $\tilde{\mu}$ is infinite-dimensional. One could tackle this issue by discretizing the possible support of $\tilde{\mu}$ and replacing its TV norm by the $\ell_1$ norm of the corresponding vector \cite{tang2013discretize}. Here, we present an alternative approach, originally proposed in~\cite{superres_new}, that solves the infinite-dimensional optimization problem directly without resorting to discretization. The approach, inspired by a method for TV-norm minimization ~\cite{Candes:2012uf} (see also \cite{Bhaskar:2012tq}), relies on the fact that the dual of Problem~\eqref{eq:opt_problem} can be recast as a finite-dimensional semidefinite program (SDP).

To simplify notation we introduce the operator $\mathcal{T}$. For any vector $\vct{u}$ whose first entry $\vct{u}_1$ is positive and real, $\mathcal{T}\brac{\vct{u}}$ is a Hermitian Toeplitz matrix whose first row is equal to $\vct{u}^T$. The adjoint of $\mathcal{T}$ with respect to the usual matrix inner product $\PROD{M_1}{M_2}=\text{Tr}\brac{M_1^{\ast}M_2}$, extracts the sums of the diagonal and of the different off-diagonal elements of a matrix 
\begin{align}
\mathcal{T}^{\ast}\brac{M}_j = \sum_{i=1}^{n-j+1}M_{i,i+j-1}.
\end{align}
\begin{lemma}
\label{lemma:dual_sdp}
The dual of Problem~\eqref{eq:opt_problem} is
\begin{align}
\label{eq:dual_sinesspikes}
\max_{ \vct{\eta} \in \C^{n}} \;   \PROD{\vct{y}}{\vct{\eta}}  \quad \text{subject to}
\quad & \normInf{\mathcal{F}_{n}^{\ast} \, \vct{\eta}}  \leq 1, \\
&  \normInf{\vct{\eta}}  \leq \lambda,
\end{align}
where the inner product is defined as $\PROD{ \vct{y}}{ \vct{\eta}} : = \op{Re}\brac{\vct{y}^{\ast}\vct{\eta}}$. This problem is equivalent to the semidefinite program
\begin{align}
\label{eq:TVnormMin_sdp_sinesspikes}
\max_{\vct{\eta} \in \C^n, \, \Lambda \in \C^{n\times n}} \;   \PROD{ \vct{y}}{ \vct{\eta}} \quad
 \text{subject to} \quad & \MAT{\Lambda & \vct{\eta} \\ \vct{\eta}^{\ast} & 1} \succeq 0, \nonumber\\
 & \ml{T}^{\ast}\brac{\Lambda}= \MAT{1 \\ \vct{0}},  \nonumber \\
 & \normInf{\vct{\eta}}  \leq \lambda,
\end{align}
where $\vct{0} \in \C^{n-1}$ is a vector of zeros. 
\end{lemma}
Lemma~\ref{lemma:dual_sdp}, which follows from Lemma~\ref{lemma:dual_sdp_noise} below, shows that it is tractable to compute the $n$-dimensional solution to the dual of Problem~\eqref{eq:opt_problem}. However, our goal is to obtain the primal solution, which represents the estimate of the line spectrum and the sparse corruptions. The following lemma, which is a consequence of Lemma~\ref{lemma:primaldual_noise}, establishes that we can decode the support of the primal solution from the dual solution. 
\begin{lemma} 
\label{lemma:primaldual_sinesspikes}
Let 
\begin{align}
\hat{\mu} & =\sum_{f_j \in \widehat{T}} \vct{\hat{x}}_{j} \, \delta \brac{ f - f_j}, 
\end{align}
and $\vct{\hat{z}}$ be a solution to \eqref{eq:opt_problem}, such that $\widehat{T}$ and $\widehat{\Omega}$ are the nonzero supports of the line spectrum $\hat{\mu}$ and the spikes $\vct{\hat{z}}$ respectively. If $\vct{ \hat{\eta} } \in \C^n$ is a corresponding dual solution, then for any $f_j$ in $\widehat{T}$ 
\begin{align}
 \brac{\mathcal{F}_{n}^{\ast} \, \vct{ \hat{\eta} }} \brac{f_j} =\frac{\vct{\hat{x}}_j}{\abs{\vct{\hat{x}}_j}}
\end{align}
and for any $l$ in $\widehat{\Omega}$
\begin{align}
\vct{ \hat{\eta} }_l = \lambda \frac{\vct{\hat{z}}_l}{\abs{\vct{\hat{z}}_l}}.
\end{align}
\end{lemma}

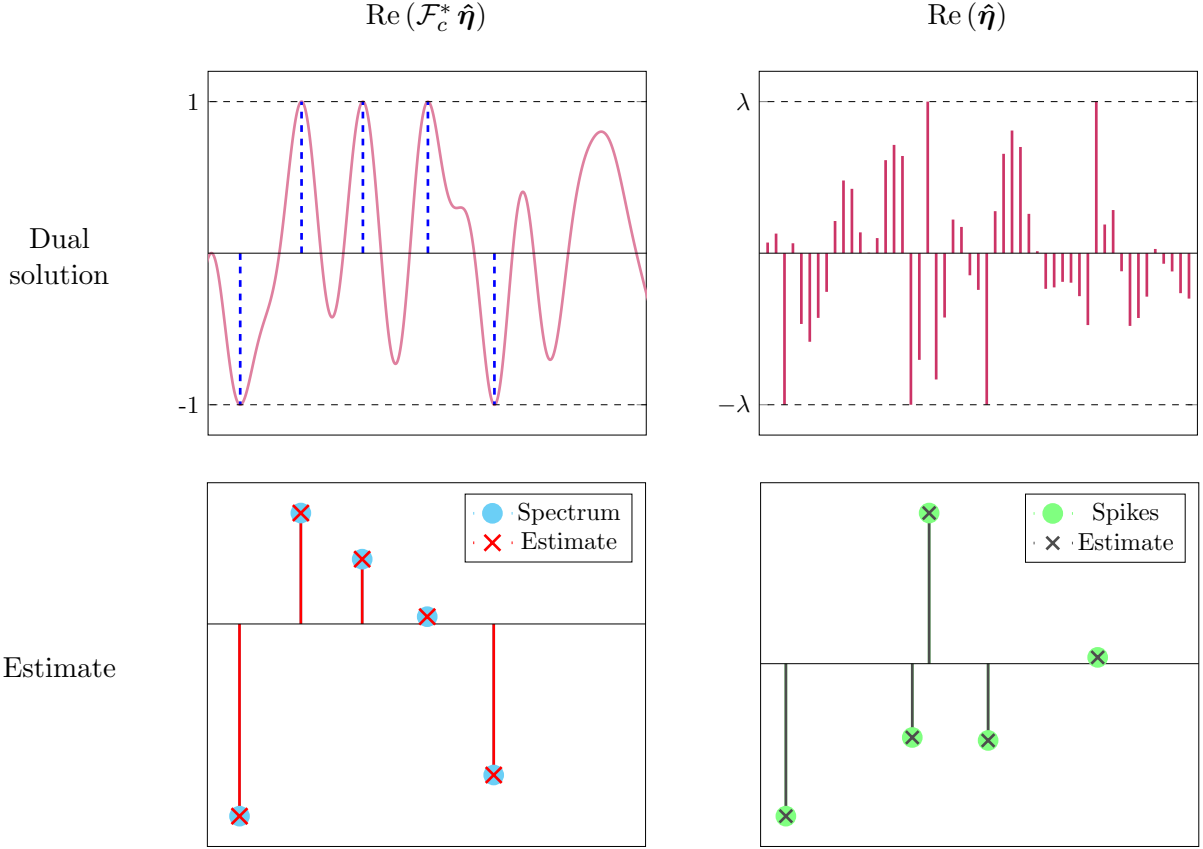
\begin{figure}[t]
\centering
\begin{tabular}{
>{\centering\arraybackslash}m{0.12\linewidth}>{\centering\arraybackslash}m{0.42\linewidth} >{\centering\arraybackslash}m{0.4\linewidth} }
 &  $\op{Re} \brac{ \mathcal{F}_c^{\ast} \, \vct{ \hat{\eta} } }$   & $\op{Re} \brac{ \vct{ \hat{\eta} } }$ \\ \\
Dual solution &
 \hspace{-0.6cm}
 \begin{tikzpicture}[scale=0.85]
\begin{axis}[ytick={-1,1}, yticklabels={-1,1},xmin=0.03,xmax=0.23, xtick=\empty]
\addplot[very thick,white!50!purple] file {sinesandspikes_dualpoly.dat};
\addplot+[ycomb,mark=none, dashed, blue,very thick]  file {sinesandspikes_sign.dat};
\addplot[black,forget plot] coordinates {(0,0) (0.27,0)};
\addplot[black,dashed,forget plot] coordinates {(0,1) (0.28,1)};
\addplot[black,dashed,forget plot] coordinates {(0,-1) (0.28,-1)};
\end{axis}
\end{tikzpicture}  
& 
\hspace{-0.5cm}
\begin{tikzpicture}[scale=0.85]
\begin{axis}[ytick={-0.0995, 0.0995}, yticklabels={$-\lambda$, $\lambda$},xmin=-1,xmax=51, xtick=\empty]
\addplot+[ycomb,mark=none,very thick,white!20!purple] file {sinesandspikes_dualpoly_spectrum.dat};
\addplot[black] coordinates {(-1,0) (51,0)};
\addplot[black,dashed,forget plot] coordinates {(0,0.0995) (51,0.0995)};
\addplot[black,dashed,forget plot] coordinates {(0,-0.0995) (51,-0.0995)};
\end{axis}
\end{tikzpicture} 
\\ \\
Estimate
&
\begin{tikzpicture}[scale=0.85]
\begin{axis}[ticks=none,xmin=0.03,xmax=0.23, legend pos=north east]
\addlegendentry{Spectrum}
 \addlegendentry{Estimate}
\addplot+[ycomb,mark=*, cyan!50!white,very thick,mark options={solid,scale=2}]  file {sines_spectrum.dat};
\addplot+[ycomb,mark=x, red,mark options={fill=red,scale=2.5},very thick]  file {sdp_robust.dat};
\addplot[black] coordinates {(0,0) (0.27,0)};
\end{axis}
\end{tikzpicture}
&
\hspace{0.2cm}
\begin{tikzpicture}[scale=0.85]
\begin{axis}[ticks=none,xmin=-1,xmax=51, legend pos=north east]
\addlegendentry{Spikes}
 \addlegendentry{Estimate}
\addplot+[ycomb,mark=*, green!50!white,very thick,mark options={solid,scale=2}]  file {spikes.dat};
\addplot+[ycomb,mark=x, black!70!white,mark options={fill=red,scale=2},very thick]  file {spikes_est.dat};
\addplot[black] coordinates {(-1,0) (51,0)};
\end{axis}
\end{tikzpicture}
\end{tabular}
\caption{Demixing of the signal in Figure~\ref{fig:sines_spikes} by semidefinite programming. Top left: the polynomial $\mathcal{F}_c^{\ast} \, \vct{ \hat{\eta} }$ (light red), where $\vct{ \hat{\eta} }$ is a solution of Problem~\eqref{eq:TVnormMin_sdp_sinesspikes}, interpolates the sign of the line spectrum of the sines (dashed red) on their support. Top right: $\lambda^{-1} \vct{ \hat{\eta} }$ interpolates the sign pattern of the spikes on their support. Bottom: locating the support of $\mu$ and $\vct{z}$ allows to demix very accurately (the circular markers represent the original spectrum of the sines and the original spikes and the crosses the corresponding estimates). The parameter $\lambda$ is set to $1/\sqrt{n}$.}
\label{fig:sines_spikes_sdp}
\end{figure}

In words, the weighted dual solution $\lambda^{-1} \vct{ \hat{\eta} }$ and the corresponding polynomial $\mathcal{F}_{n}^{\ast} \, \vct{ \hat{\eta} }$ interpolate the sign patterns of the primal-solution components $\vct{\hat{z}}$ and $\hat{\mu}$ on their respective supports, as illustrated in the top row of Figure~\ref{fig:sines_spikes_sdp}.  This suggests estimating the support of the line spectrum and the outliers in the following way. 
\begin{enumerate}
\item Solve~\eqref{eq:TVnormMin_sdp_sinesspikes} to obtain a dual solution $\vct{ \hat{\eta} }$ and compute $\mathcal{F}_c^{\ast} \, \vct{ \hat{\eta} }$. 
\item Set the estimated support of the spikes $\widehat{\Omega}$ to the set of points where $\abs{\vct{ \hat{\eta} }}$ equals $\lambda$.
\item Set the estimated support of the line spectrum $\widehat{T}$ to the set of points where $\abs{ \mathcal{F}_c^{\ast} \, \vct{ \hat{\eta} } }$ equals one.
\item Estimate the amplitudes of $\hat{\mu}$ and $\vct{\hat{\eta}}$ on $\hat{T}$ and $\hat{\Omega}$ respectively by solving a system of linear equations $\vct{y} = \mathcal{F}_n \hat{\mu} + \hat{\vct{\eta}}$.
\end{enumerate}

Figure~\ref{fig:sines_spikes_sdp} shows the results obtained by this method on the data described in Figure~\ref{fig:sines_spikes}: both components are recovered very accurately. However, we caution the reader that while the primal solution $(\hat{\mu}, \hat{\vct{z}})$ is generally unique, the dual solutions are non-unique and some of the dual solutions might produce spurious frequencies and spikes in steps 2 and 3. In fact, the dual solutions form a convex set and only those in the interior of this convex set give exact supports $\hat{\Omega}$ and $\hat{T}$, while those on the boundary generate spurious estimates. When the semidefinite program \eqref{eq:TVnormMin_sdp_sinesspikes} is solved using interior point algorithms as the case in CVX, a dual solution in the interior is returned, generating correct supports as shown in Figure \ref{fig:sines_spikes_sdp}. Refer to \cite{tang2012offgrid} for a rigorous treatment of this topic for the related missing-data case. Such technical complication will not seriously affect our estimates of the supports since the amplitudes inferred in step 4 will be zero for the extra frequencies and spikes, providing a means to eliminate them.


\subsection{Demixing in the presence of dense perturbations}
\label{sec:sdp_noise}
As described in Section~\ref{sec:dense_noise} our demixing method can be adapted to the presence of dense noise in the data by relaxing the equality constraint in Problem~\ref{eq:opt_problem} to an inequality constraint. The only effect on the dual of the optimization problem, which can still be reformulated as an SDP, is an extra term in the cost function.
\begin{lemma}[Proof in Section~\ref{proof:dual_sdp_noise}]
\label{lemma:dual_sdp_noise}
The dual of Problem~\eqref{eq:opt_problem_noise} is
\begin{alignat}{2}
\label{eq:dual_noise}
& \max_{ \vct{\eta} \in \C^{n}} \;   \PROD{\vct{y}}{\vct{\eta}} && - \sigma \normTwo{\vct{\eta}}    \\
& \text{subject to} \quad && \normInf{\mathcal{F}_{n}^{\ast} \, \vct{\eta}}  \leq 1, \label{eq:cond_one_primaldual}\\
& && \normInf{\vct{\eta}}  \leq \lambda. \label{eq:cond_lambda_primaldual}
\end{alignat}
This problem is equivalent to the semidefinite program
\begin{align}
\label{eq:TVnormMin_sdp_sinesspikes_noise}
\max_{\vct{\eta} \in \C^n, \, \Lambda \in \C^{n\times n}} \;   \PROD{ \vct{y}}{ \vct{\eta}} - \sigma \normTwo{\vct{\eta}} \quad
 \text{subject to} \quad & \MAT{\Lambda & \vct{\eta} \\ \vct{\eta}^{\ast} & 1} \succeq 0, \\
 & \ml{T}^{\ast}\brac{\Lambda}= \MAT{1 \\ \vct{0}},  \\
 & \normInf{\vct{\eta}}  \leq \lambda,
\end{align}
where $\vct{0} \in \C^{n-1}$ is a vector of zeros. 
\end{lemma}

\begin{figure}[tp]
\begin{tabular}{
>{\centering\arraybackslash}m{0.1\linewidth}>{\centering\arraybackslash}m{0.4 \linewidth} >{\centering\arraybackslash}m{0.36\linewidth} }
& $\abs{\vct{ \hat{\eta} }}$  &  $\abs{\mathcal{F}_c^{\ast} \, \vct{ \hat{\eta} }}$   \\ \\
No dense noise 
& 
\begin{tikzpicture}[scale=0.91]
\begin{axis}[ytick={0, 0.0995}, yticklabels={0, $\lambda$},xmin=6,xmax=51,  xtick=\empty]
\addplot+[ycomb,mark=none,very thick,purple] file {sinesandspikes_dualpoly_spectrum_abs.dat};
\addplot[only marks,mark=*, blue,mark options={scale=1},very thick] file {spikes_sign_abs.dat};
\addplot[black] coordinates {(-1,0) (51,0)};
\addplot[black,dashed,forget plot] coordinates {(0,0.0995) (51,0.0995)};
\end{axis}
\end{tikzpicture}
&
 \begin{tikzpicture}[scale=0.91]
\begin{axis}[ytick={0,1}, yticklabels={0,1},xmin=0.03,xmax=0.23, xtick=\empty,ymin=-0.1]
\addplot[very thick,white!50!purple] file {sinesandspikes_dualpoly_abs.dat};
\addplot+[ycomb,mark=none, dashed, blue,very thick]  file {sinesandspikes_sign_abs.dat};
\addplot[black,forget plot] coordinates {(0,0) (0.27,0)};
\addplot[black,dashed,forget plot] coordinates {(0,1) (0.28,1)};
\end{axis}
\end{tikzpicture}   \\
\\
30 dB 
& 
\begin{tikzpicture}[scale=0.91]
\begin{axis}[ytick={0, 0.0995}, yticklabels={0, $\lambda$},xmin=6,xmax=51, xtick=\empty]
\addplot+[ycomb,mark=none,very thick,white!20!purple] file {noise_sdp_dualpoly_spectrum_abs_1.dat};
\addplot[only marks,mark=*, blue,mark options={scale=1},very thick] file {lambda_sign_abs.dat};
\addplot[black] coordinates {(-1,0) (51,0)};
\addplot[black,dashed,forget plot] coordinates {(0,0.0995) (51,0.0995)};
\end{axis}
\end{tikzpicture}
&
 \begin{tikzpicture}[scale=0.91]
\begin{axis}[ytick={0,1}, yticklabels={0,1},xmin=0.03,xmax=0.23, xtick=\empty,ymin=-0.1]
\addplot[very thick,white!50!purple] file {noise_sdp_dualpoly_abs_1.dat};
\addplot+[ycomb,mark=none, dashed, blue,very thick]  file {sinesandspikes_sign_abs.dat};
\addplot[black] coordinates {(0,0) (0.27,0)};
\addplot[black,dashed,forget plot] coordinates {(0,1) (0.28,1)};
\end{axis}
\end{tikzpicture} 
\\
\\
15 dB &
\begin{tikzpicture}[scale=0.91]
\begin{axis}[ytick={0, 0.0995}, yticklabels={0, $\lambda$},xmin=6,xmax=51, xtick=\empty]
\addplot+[ycomb,mark=none,very thick,white!20!purple] file {noise_sdp_dualpoly_spectrum_abs_2.dat};
\addplot[only marks,mark=*, blue,mark options={scale=1},very thick] file {lambda_sign_abs.dat};
\addplot[black] coordinates {(-1,0) (51,0)};
\addplot[black,dashed,forget plot] coordinates {(0,0.0995) (51,0.0995)};
\end{axis}
\end{tikzpicture}
&
 \begin{tikzpicture}[scale=0.91]
\begin{axis}[ ytick={0,1}, yticklabels={0,1},xmin=0.03,xmax=0.23, xtick=\empty,ymin=-0.1]
\addplot[very thick,white!50!purple] file {noise_sdp_dualpoly_abs_2.dat};
 \addplot+[ycomb,mark=none, dashed, blue,very thick]  file {sinesandspikes_sign_abs.dat};
\addplot[black] coordinates {(0.0,0) (0.27,0)};
\addplot[black,dashed,forget plot] coordinates {(0,1) (0.28,1)};
\end{axis}
\end{tikzpicture}  
\end{tabular}
\caption{The left column shows the magnitude of the solution to Problem~\eqref{eq:dual_sinesspikes} (top row) and to Problem~\ref{eq:dual_noise} for different noise levels (second and third rows). $\abs{\vct{ \hat{\eta} }}$ is represented by red lines. Additionally, the support of the sparse perturbation $\vct{z}$ is marked in blue. The right column shows the trigonometric polynomial corresponding to the dual solutions in red, as well as the support of the spectrum of the multisinusoidal components in blue. The data are the same as in Figure~\ref{fig:sines_spikes} (except for the added noise, which is iid Gaussian). The parameters $\lambda$ and $\sigma$ are set to $1/\sqrt{n}$ and $1.5 \, \normTwo{\vct{w}}$ respectively. Note that in practice, the value of the noise level would have to be estimated, for example by cross validation.}
\label{fig:dual_poly_dense_noise}
\end{figure}
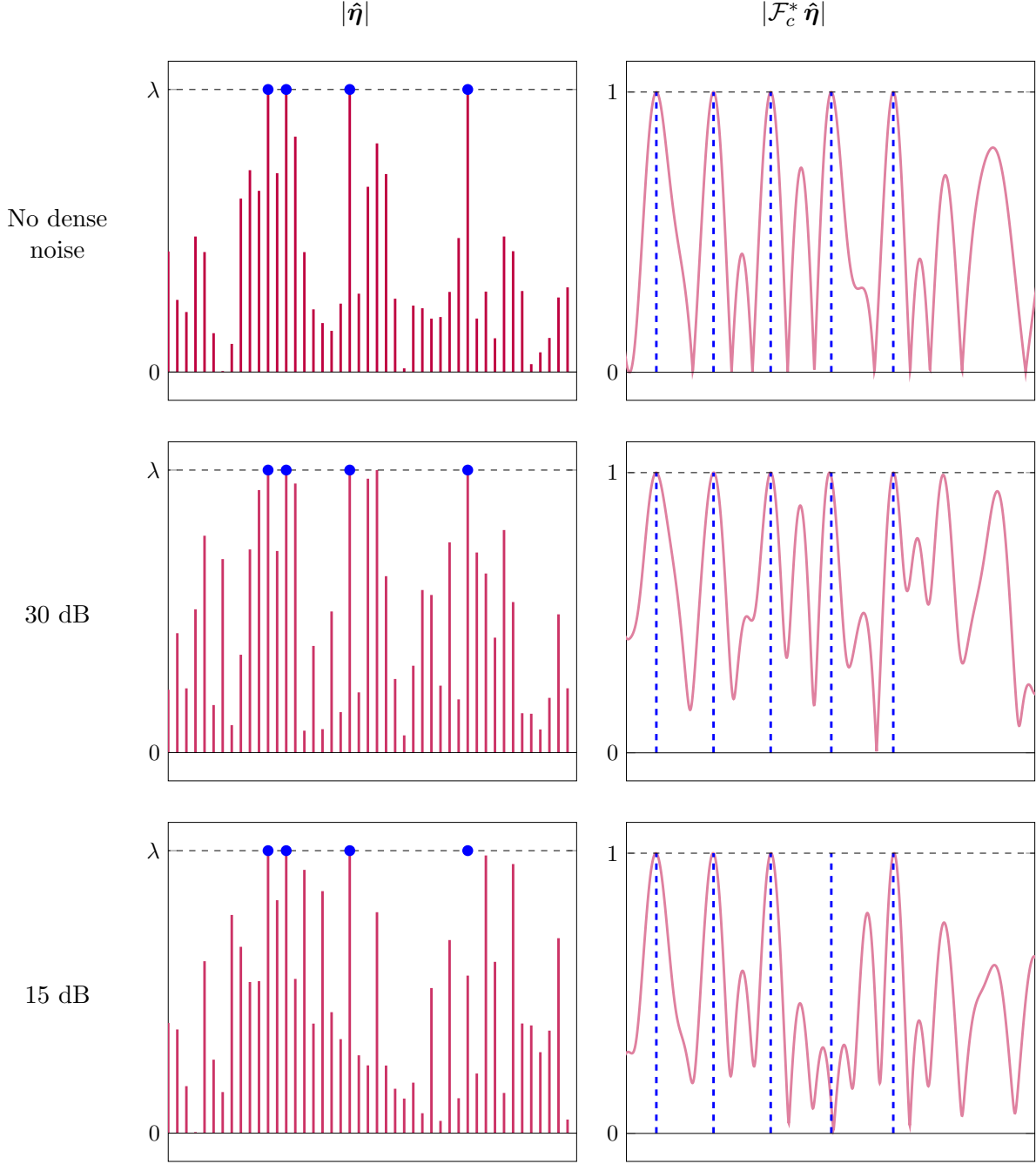

As in the case without dense noise, the support of the primal solution of Problem~\eqref{eq:opt_problem_noise} can be decoded from the dual solution. This is justified by the following lemma, which establishes that the weighted dual solution $\lambda^{-1} \vct{ \hat{\eta} }$ and the corresponding polynomial $\mathcal{F}_{n}^{\ast} \, \vct{ \hat{\eta} }$ interpolate the sign patterns of the primal-solution components $\vct{\hat{z}}$ and $\hat{\mu}$ on their respective supports. 

\begin{lemma}[Proof in Section~\ref{proof:primaldual_noise}]
\label{lemma:primaldual_noise}
Let 
\begin{align}
\hat{\mu} & =\sum_{f_j \in \widehat{T}} \vct{\hat{x}}_{j} \, \delta \brac{ f - f_j}, 
\end{align}
and $\vct{\hat{z}}$ be a solution to \eqref{eq:opt_problem_noise}, such that $\widehat{T}$ and $\widehat{\Omega}$ are the nonzero supports of the line spectrum $\hat{\mu}$ and the spikes $\vct{\hat{z}}$ respectively. If $\vct{ \hat{\eta} } \in \C^n$ is a corresponding dual solution, then for any $f_j$ in $\widehat{T}$ 
\begin{align}
\label{eq:eta_mu_noise}
 \brac{\mathcal{F}_{n}^{\ast} \, \vct{ \hat{\eta} }} \brac{f_j} =\frac{\vct{\hat{x}}_j}{\abs{\vct{\hat{x}}_j}}
\end{align}
and for any $l$ in $\widehat{\Omega}$
\begin{align}
\label{eq:eta_z_noise}
\vct{ \hat{\eta} }_l = \lambda \frac{\vct{\hat{z}}_l}{\abs{\vct{\hat{z}}_l}}.
\end{align}
\end{lemma}

\begin{figure}[tp]
\centering
\begin{tabular}{  >{\centering\arraybackslash}m{0.15\linewidth} >{\centering\arraybackslash}m{0.36\linewidth} >{\centering\arraybackslash}m{0.35\linewidth}  }
SNR \hspace{0.5cm} (dense noise) &  30 dB  & 15 dB \\ 
Spectrum estimate
 &
\begin{tikzpicture}[scale=0.75]
\begin{axis}[ticks=none,legend pos = north east,xmin=0.01,xmax=0.2]
\addlegendentry{Spectrum}
\addlegendentry{Estimate}
\addplot+[ycomb,mark=*, cyan!50!white,very thick,mark options={solid,scale=2}]  file {noise_sines_spectrum_1.dat};
\addplot+[ycomb,mark=x, red,mark options={fill=red,scale=2.5},very thick]  file {noise_sdp_robust_1.dat};
\addplot[black,forget plot] coordinates {(-0.2,0) (0.45,0)};
\end{axis}
\end{tikzpicture} 
&
\begin{tikzpicture}[scale=0.75]
\begin{axis}[ticks=none,legend style={font=\small,at={(0,0)},anchor=south west,legend columns=-1,/tikz/every even column/.append style={column sep=0.25cm},draw=none},xmin=0.01,xmax=0.2]
\addplot+[ycomb,mark=*, cyan!50!white,very thick,mark options={solid,scale=2}]  file {noise_sines_spectrum_2.dat};
\addplot+[ycomb,mark=x, red,mark options={fill=red,scale=2.5},very thick]  file {noise_sdp_robust_2.dat};
\addplot[black,forget plot] coordinates {(-0.2,0) (0.45,0)};
\end{axis}
\end{tikzpicture} 
\\
Data + spike estimate &
\begin{tikzpicture}[scale=0.75]
\begin{axis}[xmin=-1,xmax=51, ticks=none]
\addlegendentry{Data}
 \addlegendentry{Outliers}
 \addlegendentry{Detected}
 \addlegendimage{blue,thick}
\addplot+[ycomb,mark=none,very thick,white!50!blue,forget plot] file {noise_y_abs_1.dat};
\addplot[only marks,mark=o, orange,mark options={fill=red,scale=2},very thick] file {noise_sparse_noise_true_1.dat};
\addplot[only marks,mark=x, red,mark options={fill=red,scale=2.5},very thick] file {noise_sparse_noise_est_1.dat};
\addplot[black] coordinates {(-2,0) (51,0)};
\end{axis}
\end{tikzpicture} 
& 
\begin{tikzpicture}[scale=0.75]
\begin{axis}[xmin=-1,xmax=51, ticks=none]
\addplot+[ycomb,mark=none,very thick,white!50!blue] file {noise_y_abs_2.dat};
\addplot[only marks,mark=o, orange,mark options={fill=red,scale=2},very thick] file {noise_sparse_noise_true_2.dat};
\addplot[only marks,mark=x, red,mark options={fill=red,scale=2.5},very thick] file {noise_sparse_noise_est_2.dat};
\addplot[black] coordinates {(-2,0) (51,0)};
\end{axis}
\end{tikzpicture} 
\\
Noise &
\begin{tikzpicture}[scale=0.75]
\begin{axis}[xmin=-1,xmax=51, ticks=none]
\addlegendentry{Sparse}
 \addlegendentry{Dense}
 \addlegendimage{orange,thick}
 \addlegendimage{black!50!green,thick}
\addplot+[ycomb,mark=none,very thick,black!50!green] file {noise_dense_noise_abs_1.dat};
\addplot+[ycomb,mark=none,very thick,orange] file {noise_sparse_noise_abs_1.dat};
\addplot[black] coordinates {(-2,0) (51,0)};
\end{axis}
\end{tikzpicture}
& 
\begin{tikzpicture}[scale=0.75]
\begin{axis}[xmin=-1,xmax=51, ticks=none]
\addplot+[ycomb,mark=none,very thick,black!50!green] file {noise_dense_noise_abs_2.dat};
\addplot+[ycomb,mark=none,very thick,orange] file {noise_sparse_noise_abs_2.dat};
\addplot[black] coordinates {(-1,0) (51,0)};
\end{axis}
\end{tikzpicture}
\end{tabular}
\caption{The top row shows the results of applying SDP-based spectral super-resolution in the presence of both dense noise and outliers (bottom row) for two different dense-noise levels (left and right columns). The second row shows the magnitude of the data, the location of the outliers and the outlier estimate produced by the method. In the bottom row we can see the magnitude of the sparse and dense noise (note that when the SNR is 15 dB the smallest sparse-noise components is below the dense-noise level). The signal is the same as in Figure~\ref{fig:sines_spikes} and the data are the same as in Figure~\ref{fig:dual_poly_dense_noise}. The parameter $\sigma$ is set to $1.5 \, \normTwo{\vct{w}}$ and $\lambda$ is set to $1/\sqrt{n}$.}
\label{fig:sdp_noisy}
\end{figure}
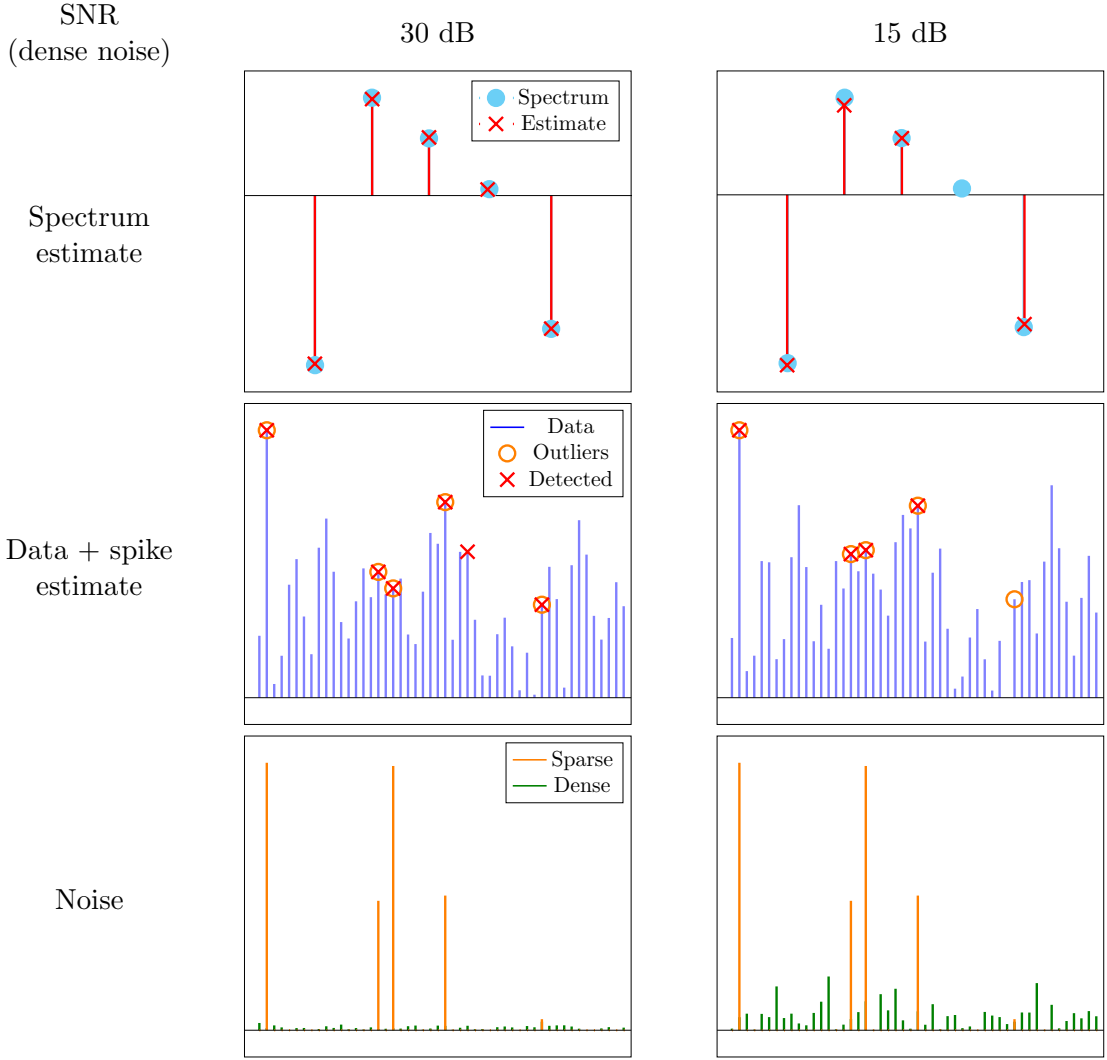
Figure~\ref{fig:dual_poly_dense_noise} shows the magnitude of the dual solutions for different values of additive noise. Motivated by the lemma, we propose to estimate the support of the outliers using $\vct{ \hat{\eta} }$ and the support of the spectral lines using $\abs{\mathcal{F}_c^{\ast} \, \vct{ \hat{\eta} }}$. Our method to perform spectral super-resolution in the presence of outliers and dense noise consequently consists of the following steps:
\begin{enumerate}
\item Solve~\eqref{eq:TVnormMin_sdp_sinesspikes_noise} to obtain a dual solution $\vct{ \hat{\eta} }$ and compute $\mathcal{F}_c^{\ast} \, \vct{ \hat{\eta} }$. 
\item Set the estimated support of the spikes $\widehat{\Omega}$ to the set of points where $\abs{\vct{ \hat{\eta} }}$ equals $\lambda$.
\item Set the estimated support of the spectrum $\widehat{T}$ to the set of points where $\abs{ \mathcal{F}_c^{\ast} \, \vct{ \hat{\eta} } }$ equals one.
\item Estimate the amplitudes of $\hat{\mu}$ by solving a least-squares problem using only the data that do not lie in the estimated support of the spikes $\widehat{\Omega}$.
\end{enumerate}

Figure~\ref{fig:sdp_noisy} shows the result of applying our method to data that includes additive iid Gaussian noise with a signal-to-noise ratio (SNR) of 30 and 15 dB. 
Despite the presence of the dense noise, our method is able to detect all spectral lines at 30 dB and all but one at 15 dB. Additionally, it is capable of detecting most of the spikes correctly: at 30 dB it detects a spurious spike and at 15 dB it misses one. Note that the spike that is not detected when the SNR is 15 dB has a magnitude small enough for it to be considered part of the dense noise.

\subsection{Greedy demixing enhanced by local nonconvex optimization}
\label{sec:greedy}
In this section we propose an alternative method for spectral super-resolution in the presence of outliers, which is significantly faster than the SDP-based approach described in the previous sections. In the spirit of matching-pursuit methods~\cite{mp,omp}, the algorithm selects the spectral lines of the signal and the locations of the outliers in a greedy fashion. This is equivalent to choosing atoms from a dictionary of the form
\begin{align}
\ml{D} := \keys{\vct{a} \brac{ f, 0 }, \, f \in \sqbr{0,1} } \cup \keys{\vct{e}\brac{l}, \, 1 \leq l \leq n}.
\end{align}
The dictionary includes the multisinusoidal atoms $\vct{a} \brac{ f, 0 }$ defined in~\eqref{eq:atoms} and $n$ \emph{spiky} atoms $\vct{e}\brac{l} \in \R^{n}$, which are equal to the one-sparse standard-basis vectors. By~\eqref{eq:data_denoising}, if the data $\vct{y}$ are of the form~\eqref{eq:model_data} then they have a $\brac{k+s}$-sparse representation in terms of the atoms in $\ml{D}$. Greedy demixing aims to find this sparse representation iteratively. 

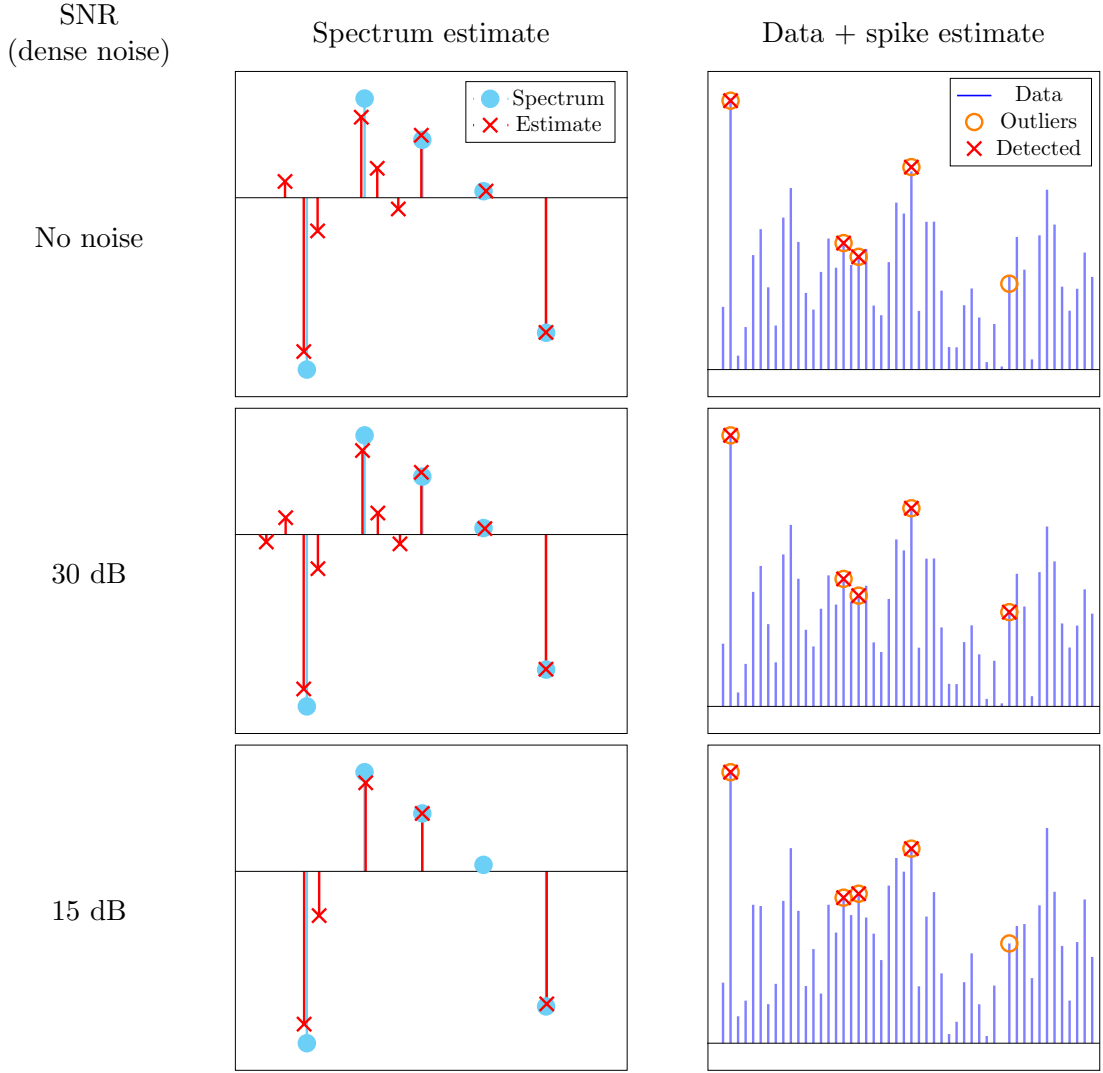
\begin{figure}[tp]
\begin{tabular}{
>{\centering\arraybackslash}m{0.15\linewidth}>{\centering\arraybackslash}m{0.35 \linewidth} >{\centering\arraybackslash}m{0.36\linewidth} }
SNR \hspace{0.5cm} (dense noise) & Spectrum estimate & Data + spike estimate\\
No noise   
  &
 \begin{tikzpicture}[scale=0.76]
\begin{axis}[ticks=none,legend pos=north east,xmin=0.01,xmax=0.2]
\addlegendentry{Spectrum}
 \addlegendentry{Estimate}
 \addplot+[ycomb,mark=*, cyan!50!white,very thick,mark options={solid,scale=2}]  file {noise_sines_spectrum_1.dat};
\addplot+[ycomb,mark=x, red,mark options={fill=red,scale=2.5},very thick]  file {greedy_no_ld.dat};
\addplot[black,forget plot] coordinates {(-0.2,0) (0.45,0)};
\end{axis}
\end{tikzpicture} 
  &
  \begin{tikzpicture}[scale=0.76]
\begin{axis}[xmin=-1,xmax=51, ticks=none]
\addlegendentry{Data}
 \addlegendentry{Outliers}
 \addlegendentry{Detected}
 \addlegendimage{blue,thick}
\addplot+[ycomb,mark=none,very thick,white!50!blue,forget plot] file {noise_y_abs_1.dat};
\addplot[only marks,mark=o, orange,mark options={fill=red,scale=2},very thick] file {sparse_noise_true.dat};
\addplot[only marks,mark=x, red,mark options={fill=red,scale=2.5},very thick] file {greedy_no_ld_spikes.dat};
\addplot[black] coordinates {(-2,0) (51,0)};
\end{axis}
\end{tikzpicture} 
\\
30 dB   &
 \begin{tikzpicture}[scale=0.76]
\begin{axis}[ticks=none,legend pos=north east,xmin=0.01,xmax=0.2]
 \addplot+[ycomb,mark=*, cyan!50!white,very thick,mark options={solid,scale=2}]  file {noise_sines_spectrum_1.dat};
\addplot+[ycomb,mark=x, red,mark options={fill=red,scale=2.5},very thick]  file {noise_greedy_no_ld_1.dat};
\addplot[black,forget plot] coordinates {(-0.2,0) (0.45,0)};
\end{axis}
\end{tikzpicture} 
&
\begin{tikzpicture}[scale=0.76]
\begin{axis}[xmin=-1,xmax=51, ticks=none]
 \addlegendimage{blue,thick}
\addplot+[ycomb,mark=none,very thick,white!50!blue,forget plot] file {noise_y_abs_1.dat};
\addplot[only marks,mark=o, orange,mark options={fill=red,scale=2},very thick] file {noise_sparse_noise_true_1.dat};
\addplot[only marks,mark=x, red,mark options={fill=red,scale=2.5},very thick] file {noise_greedy_no_ld_spikes_1.dat};
\addplot[black] coordinates {(-2,0) (51,0)};
\end{axis}
\end{tikzpicture} 
\\
15 dB
&
\begin{tikzpicture}[scale=0.76]
\begin{axis}[ticks=none,legend style={font=\small,at={(0,0)},anchor=south west,legend columns=-1,/tikz/every even column/.append style={column sep=0.25cm},draw=none},xmin=0.01,xmax=0.2]
\addplot+[ycomb,mark=*, cyan!50!white,very thick,mark options={solid,scale=2}]  file {noise_sines_spectrum_2.dat};
\addplot+[ycomb,mark=x, red,mark options={fill=red,scale=2.5},very thick]  file {noise_greedy_no_ld_2.dat};
\addplot[black,forget plot] coordinates {(-0.2,0) (0.45,0)};
\end{axis}
\end{tikzpicture} 
& 
\begin{tikzpicture}[scale=0.76]
\begin{axis}[xmin=-1,xmax=51, ticks=none]
\addplot+[ycomb,mark=none,very thick,white!50!blue] file {noise_y_abs_2.dat};
\addplot[only marks,mark=o, orange,mark options={fill=red,scale=2},very thick] file {noise_sparse_noise_true_2.dat};
\addplot[only marks,mark=x, red,mark options={fill=red,scale=2.5},very thick] file {noise_greedy_no_ld_spikes_2.dat};
\addplot[black] coordinates {(-2,0) (51,0)};
\end{axis}
\end{tikzpicture} 
\end{tabular}
\caption{Demixing via greedy demixing without a local optimization step. The signal is the same as in Figure~\ref{fig:sines_spikes} and the noisy data are the same as in Figures~\ref{fig:dual_poly_dense_noise} and~\ref{fig:sdp_noisy}. The thresholding parameter $\tau$ is set depending on the noise level: at 30 dB and in the absence of dense noise it is set small enough not to eliminate the spectral line with the smallest coefficient in the pruning step, whereas at 15 dB it is set so as not to discard the spectral line with the second smallest coefficient.
}
\label{fig:greedy}
\end{figure}

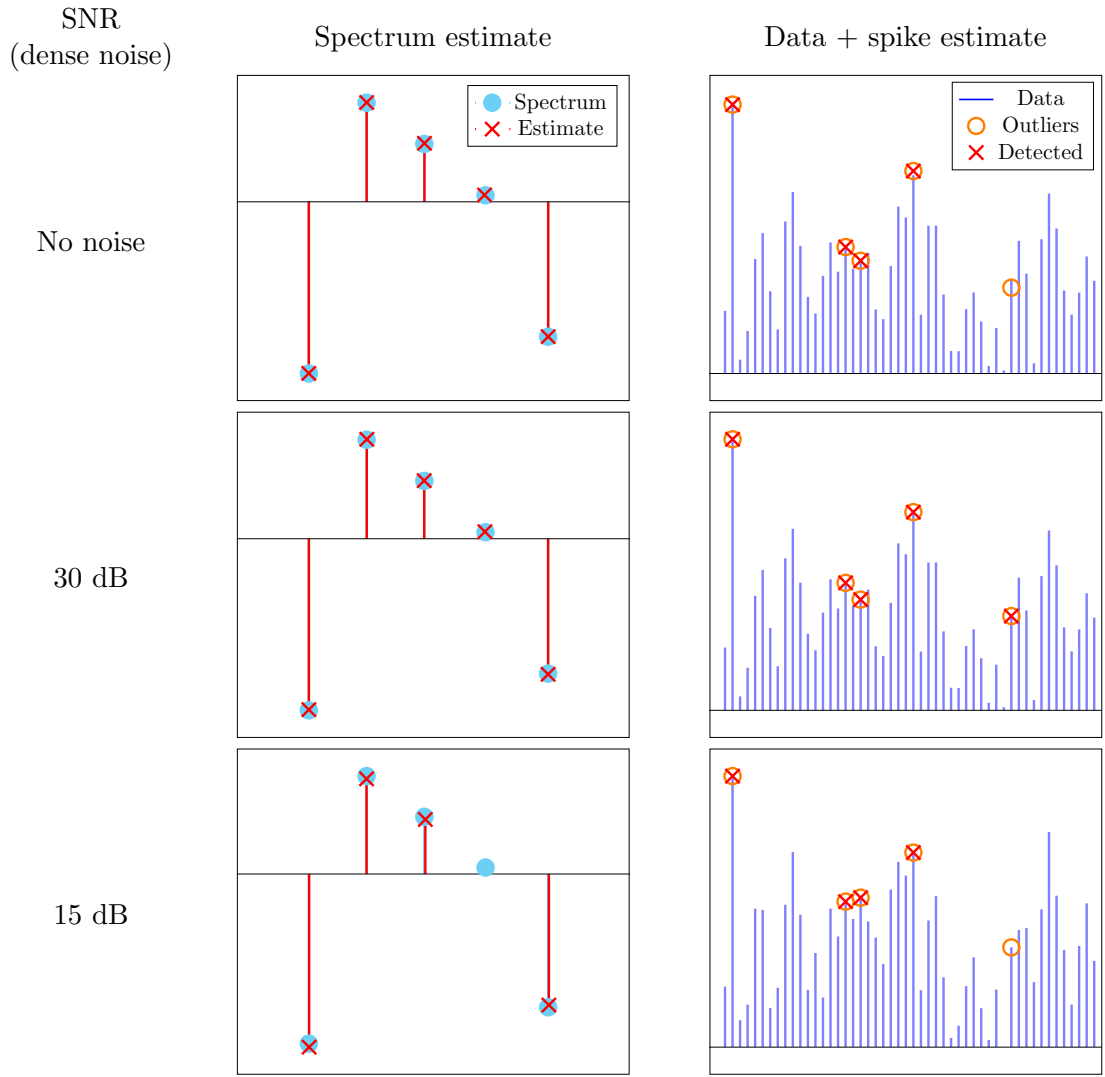
\begin{figure}[tp]
\begin{tabular}{
>{\centering\arraybackslash}m{0.15\linewidth}>{\centering\arraybackslash}m{0.35 \linewidth} >{\centering\arraybackslash}m{0.36\linewidth} }
SNR \hspace{0.5cm} (dense noise) & Spectrum estimate & Data + spike estimate\\
No noise   
  &
 \begin{tikzpicture}[scale=0.76]
\begin{axis}[ticks=none,legend pos=north east,xmin=0.01,xmax=0.2]
\addlegendentry{Spectrum}
 \addlegendentry{Estimate}
 \addplot+[ycomb,mark=*, cyan!50!white,very thick,mark options={solid,scale=2}]  file {noise_sines_spectrum_1.dat};
\addplot+[ycomb,mark=x, red,mark options={fill=red,scale=2.5},very thick]  file {greedy.dat};
\addplot[black,forget plot] coordinates {(-0.2,0) (0.45,0)};
\end{axis}
\end{tikzpicture} 
  &
  \begin{tikzpicture}[scale=0.76]
\begin{axis}[xmin=-1,xmax=51, ticks=none]
\addlegendentry{Data}
 \addlegendentry{Outliers}
 \addlegendentry{Detected}
 \addlegendimage{blue,thick}
\addplot+[ycomb,mark=none,very thick,white!50!blue,forget plot] file {noise_y_abs_1.dat};
\addplot[only marks,mark=o, orange,mark options={fill=red,scale=2},very thick] file {sparse_noise_true.dat};
\addplot[only marks,mark=x, red,mark options={fill=red,scale=2.5},very thick] file {greedy_spikes.dat};
\addplot[black] coordinates {(-2,0) (51,0)};
\end{axis}
\end{tikzpicture} 
\\
30 dB   &
 \begin{tikzpicture}[scale=0.76]
\begin{axis}[ticks=none,legend pos=north east,xmin=0.01,xmax=0.2]
 \addplot+[ycomb,mark=*, cyan!50!white,very thick,mark options={solid,scale=2}]  file {noise_sines_spectrum_1.dat};
\addplot+[ycomb,mark=x, red,mark options={fill=red,scale=2.5},very thick]  file {noise_greedy_1.dat};
\addplot[black,forget plot] coordinates {(-0.2,0) (0.45,0)};
\end{axis}
\end{tikzpicture} 
&
\begin{tikzpicture}[scale=0.76]
\begin{axis}[xmin=-1,xmax=51, ticks=none]
 \addlegendimage{blue,thick}
\addplot+[ycomb,mark=none,very thick,white!50!blue,forget plot] file {noise_y_abs_1.dat};
\addplot[only marks,mark=o, orange,mark options={fill=red,scale=2},very thick] file {noise_sparse_noise_true_1.dat};
\addplot[only marks,mark=x, red,mark options={fill=red,scale=2.5},very thick] file {noise_greedy_spikes_1.dat};
\addplot[black] coordinates {(-2,0) (51,0)};
\end{axis}
\end{tikzpicture} 
\\
15 dB
&
\begin{tikzpicture}[scale=0.76]
\begin{axis}[ticks=none,legend style={font=\small,at={(0,0)},anchor=south west,legend columns=-1,/tikz/every even column/.append style={column sep=0.25cm},draw=none},xmin=0.01,xmax=0.2]
\addplot+[ycomb,mark=*, cyan!50!white,very thick,mark options={solid,scale=2}]  file {noise_sines_spectrum_2.dat};
\addplot+[ycomb,mark=x, red,mark options={fill=red,scale=2.5},very thick]  file {noise_greedy_2.dat};
\addplot[black,forget plot] coordinates {(-0.2,0) (0.45,0)};
\end{axis}
\end{tikzpicture} 
& 
\begin{tikzpicture}[scale=0.76]
\begin{axis}[xmin=-1,xmax=51, ticks=none]
\addplot+[ycomb,mark=none,very thick,white!50!blue] file {noise_y_abs_2.dat};
\addplot[only marks,mark=o, orange,mark options={fill=red,scale=2},very thick] file {noise_sparse_noise_true_2.dat};
\addplot[only marks,mark=x, red,mark options={fill=red,scale=2.5},very thick] file {noise_greedy_spikes_2.dat};
\addplot[black] coordinates {(-2,0) (51,0)};
\end{axis}
\end{tikzpicture} 
\end{tabular}
\caption{Demixing via greedy demixing with a local optimization step. The signal is the same as in Figure~\ref{fig:sines_spikes} and the noisy data are the same as in Figures~\ref{fig:dual_poly_dense_noise}, \ref{fig:sdp_noisy} and~\ref{fig:greedy}. The thresholding parameter $\tau$ is set as described in the caption of Figure~\ref{fig:greedy}.}
\label{fig:greedy_local_opt}
\end{figure}

Inspired by recent work on atomic-norm minimization based on the conditional-gradient method~\cite{rao2015forward,rao2014forward,sparse_inverse_ben}, our greedy-demixing procedure includes selection, pruning and local-optimization steps (see also~\cite{stoicaNLSmultimodal,eftekhari2015greed,fannjiang2012coherence} for spectral super-resolution algorithms that leverage a local-optimization step similar to ours).
\begin{enumerate}
\item \textbf{Initialization:} The residual $\vct{r} \in \C^{n}$ is initialized to equal the data vector $\vct{y}$. The sets of estimated spectral lines $\widehat{T}$ and spikes $\widehat{\Omega}$ are initialized to equal the empty set.  
\item \textbf{Selection:} At each iteration we compute the atom in $\ml{D}$ that has the highest correlation with the current residual $\vct{r}$ and update either $\widehat{T}$ or $\widehat{\Omega}$. For the \emph{spiky} atoms the correlation is just equal to $\normInf{\vct{r}}$. For the sinusoidal atoms, we compute the highest correlation by first determining the location $f_{\op{grid}}$ of the maximum of the function $\op{\emph{corr}}\brac{f}:= \abs{\PROD{\vct{a}\brac{f,0}}{\vct{r}}}$ on a fine grid, which can be done efficiently by computing an oversampled fast Fourier transform, and then finding a local minimum of the function $\op{\emph{corr}}\brac{f}$ using a local search method initialized at $f_{\op{grid}}$.
\item \textbf{Pruning:} After adding a new atom to $\widehat{T}$ or $\widehat{\Omega}$, we compute the coefficients corresponding to the selected atoms using a least-squares fit. We then remove any atoms whose corresponding coefficients are smaller than a threshold $\tau > 0$.
\item \textbf{Local optimization:} We fix the number of selected sinusoidal atoms $\hat{k}:=|\widehat{T}|$ and optimize their locations to update $\widehat{T}$ by finding a local minimum of the least-squares cost function
\begin{align}
\op{\emph{ls}}\brac{f_1, \ldots, f_{\hat{k}}} := \min_{\vct{\hat{x}} \in \C^{\hat{k}}, \vct{\hat{z}} \in \C^{\abs{\widehat{\Omega}}}} \normTwo{\vct{y} -\sqrt{n}\sum_{j=1}^{\hat{k}} \vct{\hat{x}}_j \, \vct{a}\brac{f_j,0} -\sum_{l \in \widehat{\Omega}} \vct{\hat{z}}_l \, \vct{e}\brac{l} },
\end{align}
using a local search method\footnote{We use the Matlab function \emph{fminsearch} based on the simplex search method~\cite{lagarias1998convergence}.} initialized at the current estimate $\widehat{T}$. Alternatively, one can use other methods such as gradient descent to find a local minimum of the nonconvex function.
\item The residual is updated by computing the coefficients corresponding to the currently selected atoms using least-squares and subtracting the resulting approximation from $\vct{y}$.
\end{enumerate}
This algorithm can be applied without any modification to data that are perturbed by dense noise. In Figures~\ref{fig:greedy} and~\ref{fig:greedy_local_opt} we illustrate the performance of the method on the same data used in Figures~\ref{fig:sines_spikes_sdp} and~\ref{fig:sdp_noisy}. Figure~\ref{fig:greedy} shows what happens if we omit the local-optimization step: the algorithm does not yield exact demixing even in the absence of dense noise. In contrast, in Figure~\ref{fig:greedy_local_opt} we see that greedy demixing combined with local optimization recovers the two mixed components exactly when no additional noise perturbs the data. In addition, the procedure is robust to the presence of dense noise, as shown in the last two rows of Figure~\ref{fig:greedy_local_opt}. 

Intuitively, the greedy method is not able to achieve exact recovery, because it optimizes the position of each spectral line one by one, eventually not being able to make further progress. The local-optimization step refines the fit by optimizing over the positions of the spectral lines simultaneously. This succeeds when the initialization is close enough to a \emph{good} local minimum of the cost function. Our experiments seem to indicate that the greedy scheme provides such an initialization.

\begin{figure}[tp]
\begin{center}
\includegraphics[scale=.65]{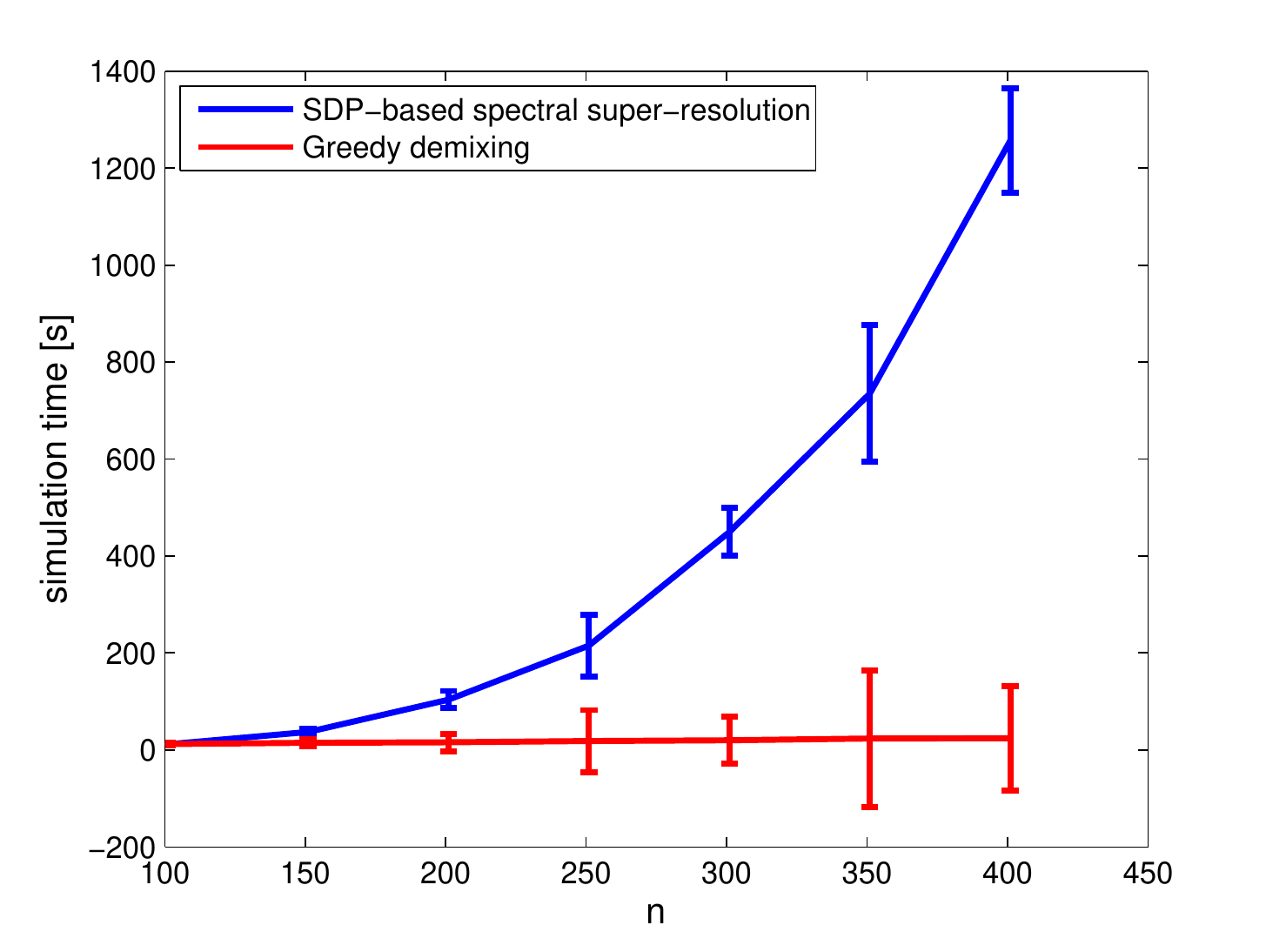}
\end{center}
\caption{Comparison of average running times for the SDP-based demixing approach described in Section~\ref{sec:sdp} and greedy demixing with a local optimization step over 10 tries (the error bars show 95\% confidence intervals). The number of spectral lines and of outliers equal $10$. The amplitudes of both components are iid Gaussian. The minimum separation of the spectral lines is $2.8/(n+1)$. Both algorithms achieve exact recovery in all instances. The experiments were carried out on a laptop with an Intel Core i5-5300 CPU 2.3GHz and 12G RAM.}
\label{fig:greedy_sim_time}
\end{figure}

As illustrated in Figure~\ref{fig:greedy_sim_time}, the greedy scheme is significantly faster than the SDP-based approach described earlier. These preliminary empirical results show the potential of coupling greedy approaches with local nonconvex optimization. Establishing guarantees for such demixing procedures is an interesting research direction. 


\subsection{Atomic-norm denoising}
\label{sec:atomic_norm}
In this section, we discuss how to implement the atomic-norm based denoising procedure described in Section~\ref{sec:denoising}. Our method relies on the fact that the atomic norm has a semidefinite characterization when the dictionary contains sinusoidal atoms of the form~\eqref{eq:atoms} . This is established in the following proposition, which we borrow from~\cite{Bhaskar:2012tq, tang2012offgrid}.
\begin{proposition}[Proposition 2.1 \cite{tang2012offgrid}, \cite{Bhaskar:2012tq}]\label{pro:sdp-char}
For $\vct{g} \in \C^{n}$
\begin{align} 
\atomicnorm{ \vct{g} }  = \inf_{ t \in \R, \, \vct{u} \in \C^n} \keys{ \frac{n \, \vct{u}_1 + t}{2}  ~:~  \MAT{ \ml{T}\brac{ \vct{u} } & \vct{g} \\ \vct{g}^{\ast} & t}  \succeq 0 
},
\end{align}
where the operator $\ml{T}$ is defined in Section~\ref{sec:sdp}.
\end{proposition}
This result allows us to rewrite \eqref{eq:atomic_norm_min} as the semidefinite program
\begin{align} 
\label{eq:denoising_sdp}
 \min_{\substack{ t \in \R, \, \vct{u} \in \C^n, \\ \vct{\tilde{g}} \in \C^n , \, \vct{\tilde{z}} \in \C^n}} \;     \frac{n \, \vct{u}_1 + t}{2 \sqrt{n}}  + \lambda \normOne{ \vct{\tilde{z}} }    \quad
 \text{subject to} \quad & \MAT{ \ml{T}\brac{ \vct{u} } & \vct{\tilde{g}} \\ \vct{\tilde{g}}^{\ast} & t}  \succeq 0, \\
& \; \;   \vct{\tilde{g}} + \vct{\tilde{z}} = \vct{y},
\end{align}
which is precisely the dual program of \eqref{eq:TVnormMin_sdp_sinesspikes}.

Similarly, Problem~\eqref{eq:atomic_norm_noise} can be reformulated as the semidefinite program,
\begin{align} 
\label{eq:denoising_sdp_noise}
 \min_{\substack{ t \in \R, \, \vct{u} \in \C^n, \\ \vct{\tilde{g}} \in \C^n , \, \vct{\tilde{z}} \in \C^n}} \;     \frac{n \, \vct{u}_1 + t}{2 \sqrt{n}} + \lambda\normOne{ \vct{\tilde{z}} }  + \frac{\gamma}{2}  \normTwo{ \vct{y} - \vct{\tilde{g}} - \vct{\tilde{z}} }^2  \quad
 \text{subject to} \quad & \MAT{ \ml{T}\brac{ \vct{u} } & \vct{\tilde{g}} \\ \vct{\tilde{g}}^{\ast} & t}  \succeq 0
 \end{align}
This problem can be solved efficiently using the alternating direction method of multipliers~\cite{boyd2011distributed} (see also~\cite{Bhaskar:2012tq} for a similar implementation of SDP-based atomic-norm denoising for the case without outliers), as described in detail in Section~\ref{sec:admm} of the appendix. Figure~\ref{fig:sdp_denoising} shows the results of applying this method to denoise the data used in Figures~\ref{fig:sdp_noisy}, ~\ref{fig:greedy} and~\ref{fig:greedy_local_opt}. In the absence of dense noise, the approach yields perfect denoising (not shown in the figure). When dense noise perturbs the data, the method is still able to perform effective denoising, correcting for the presence of the outliers. 
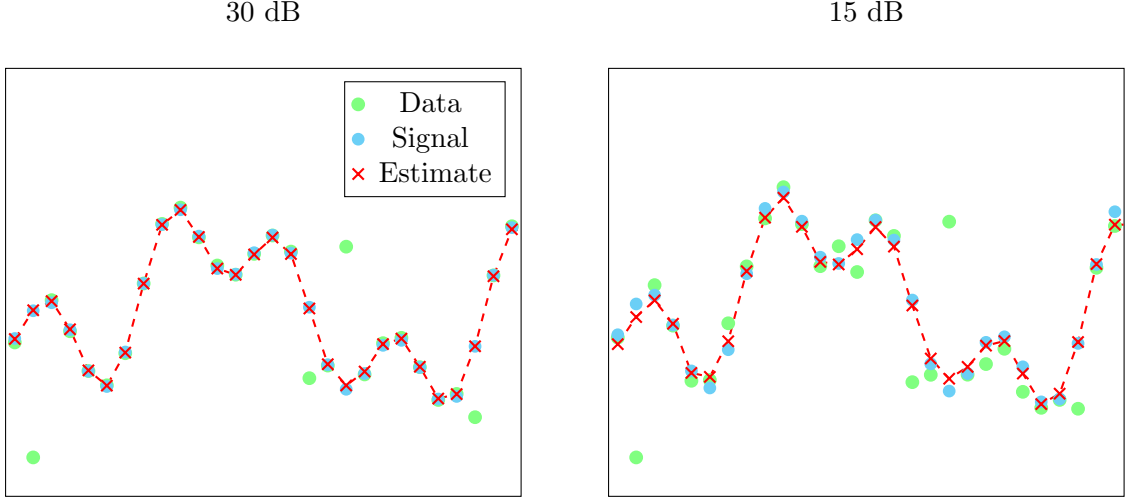
\begin{figure}[tp]
\centering
\begin{tabular}{ >{\centering\arraybackslash}m{0.5\linewidth} >{\centering\arraybackslash}m{0.42\linewidth}  }
  30 dB  & 15 dB \\ \\
\begin{tikzpicture}[scale=1]
\begin{axis}[ticks=none,xmin=0.5,xmax=28.5, legend pos=north east,ymax=8]
\addlegendentry{Data}
 \addlegendentry{Signal}
 \addlegendentry{$ $ Estimate}
\addplot[mark=*, green!50!white,mark options={solid,fill=green!50!white,scale=1},very thick, only marks]  file {admm_y_1.dat};
\addplot[only marks,mark=*, cyan!50!white, thick,mark options={solid,scale=1}]  file {admm_true_1.dat};
\addplot[mark=x, dashed,red,mark options={solid,fill=red,scale=1.5},thick,forget plot]  file {admm_est_1.dat};
\addlegendimage{mark=x, red,mark options={ fill=red,scale=1.5},thick, only marks}
\end{axis}
\end{tikzpicture} 
&
\begin{tikzpicture}[scale=1]
\begin{axis}[ticks=none,legend style={font=\small,at={(0,0)},anchor=south west,legend columns=-1,/tikz/every even column/.append style={column sep=0.25cm},draw=none},xmin=0.5,xmax=28.5,ymax=7]
\addplot[mark=*, green!50!white,mark options={solid,fill=green!50!white,scale=1},very thick, only marks]  file {admm_y_2.dat};
\addplot[only marks,mark=*, cyan!50!white, thick,mark options={solid,scale=1}]  file {admm_true_2.dat};
\addplot[mark=x, dashed,red,mark options={solid,fill=red,scale=1.5},thick]  file {admm_est_2.dat};
\end{axis}
\end{tikzpicture} 
\end{tabular}
\caption{Denoising via atomic-norm minimization in the presence of both outliers and dense noise. The signal is the same as in Figure~\ref{fig:sines_spikes} and the data is the same as in Figures~\ref{fig:dual_poly_dense_noise} and~\ref{fig:sdp_noisy}. The parameter $\lambda$ is set to $1/\sqrt{n}$, whereas $\gamma$ is set to $1/\normTwo{w}$ (in practice, we would have to estimate the noise level or set the parameter via cross validation).
}
\label{fig:sdp_denoising}
\end{figure}
\section{Numerical Experiments}
\label{sec:numerical}

\subsection{Demixing via semidefinite programming}
\label{sec:phase_transitions}

\begin{figure}[tp]
\centering
\begin{tabular}{   >{\centering\arraybackslash}m{0.09\linewidth} >{\centering\arraybackslash}m{0.3\linewidth} >{\centering\arraybackslash}m{0.34\linewidth} >{\centering\arraybackslash}m{0.1\linewidth} }
& $s = 10$ &  $k = 15$\\
 $\Delta \brac{ n- 1}$ & 
 \includegraphics{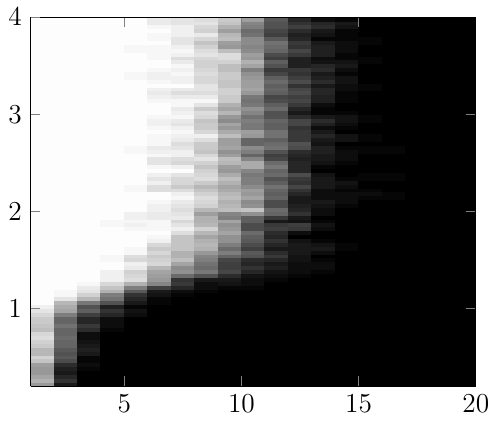}
 &  
 \includegraphics{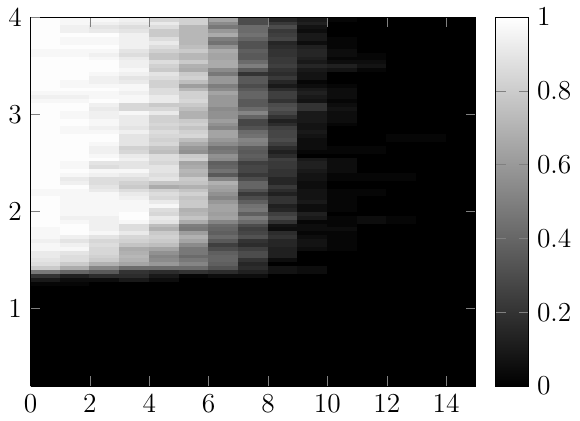}
& $n=61$ 
 \\
 $\Delta \brac{ n- 1}$ & 
 \includegraphics{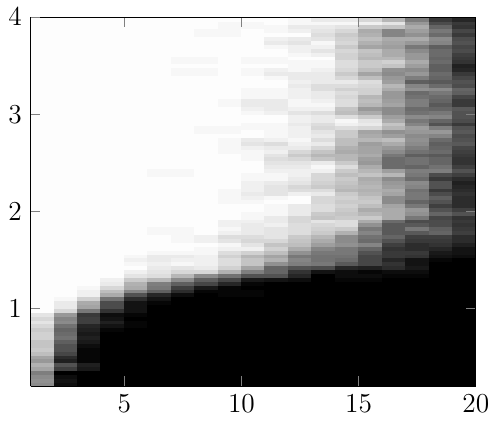}
& 
 \includegraphics{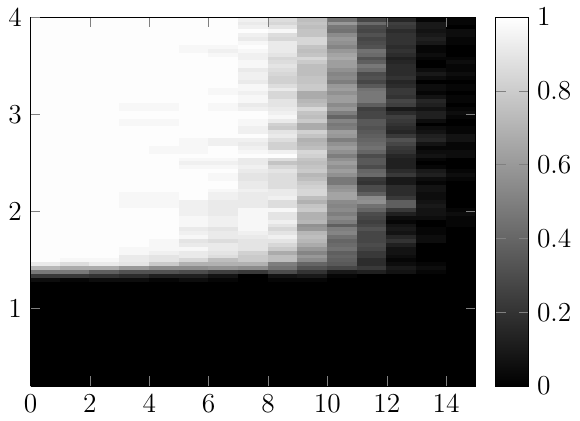}
 & $n=81$ 
\\
 $\Delta \brac{ n- 1}$ & 
 \includegraphics{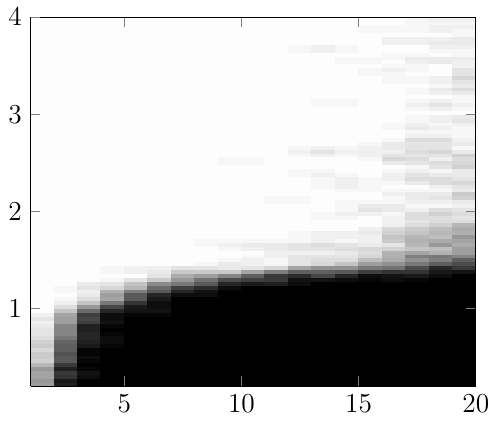}
& 
 \includegraphics{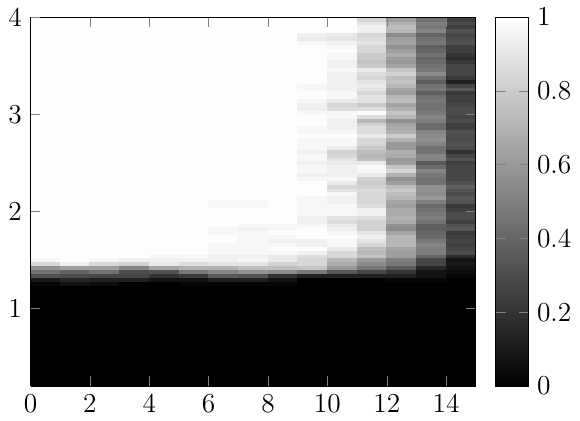}
& $n = 101$  
\\
& $k$ &  $s$ 
\end{tabular}
\caption{Graphs showing the fraction of times Problem~\eqref{eq:opt_problem} achieves exact demixing over 10 trials with random signs and supports for different numbers of spectral lines $k$ (left column) and outliers $s$ (right column), as well as different values of the minimum separation of the spectral lines. Each row shows results for a different number of measurements. The value of the regularization parameter $\lambda$ is 0.1 for the left column and 0.15 for the second column. The simulations are carried out using CVX~\cite{cvx}.}
\label{fig:phase_transitions_mindist}
\end{figure}

\begin{figure}[tp]
\hspace{-0.4cm}
\begin{tabular}{  >{\centering\arraybackslash}m{0.005\linewidth} >{\centering\arraybackslash}m{0.26\linewidth} >{\centering\arraybackslash}m{0.26\linewidth}  >{\centering\arraybackslash}m{0.32\linewidth} >{\centering\arraybackslash}m{0.08\linewidth} }
 &$\lambda=0.1$ &  $\lambda=0.15$ & $\lambda = 0.2$ \\ 
 $s$ & 
 \includegraphics{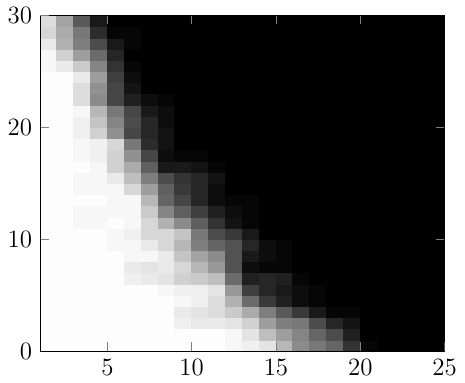}
 &  
 \includegraphics{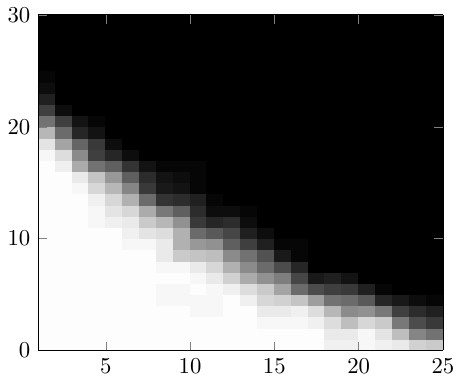}
& 
 \includegraphics{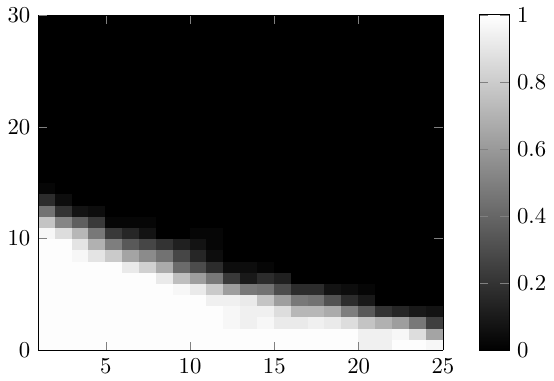}
  & $n=61$ \\
\\
 $s$ &  
  \includegraphics{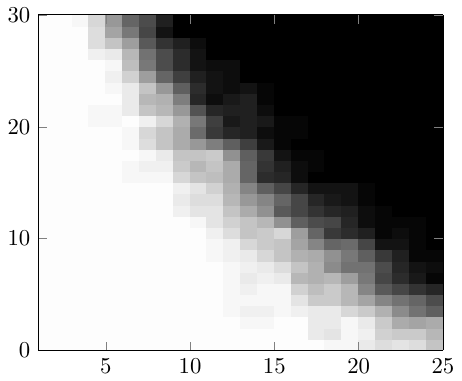}
 &  
 \includegraphics{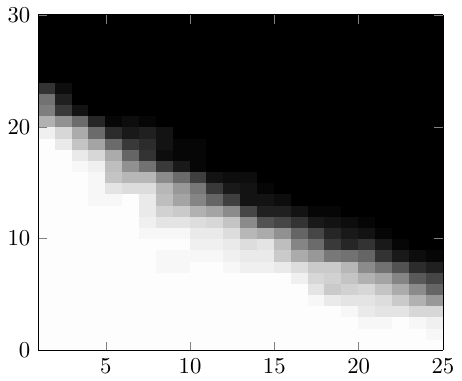}
& 
 \includegraphics{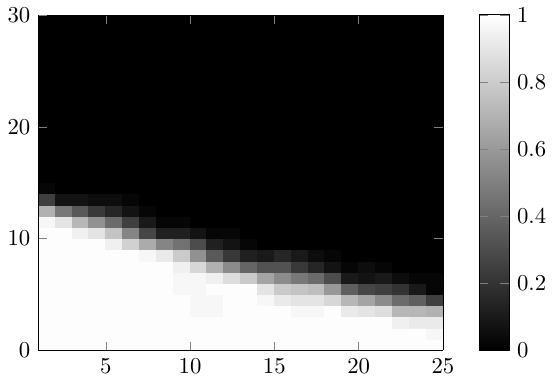}
& $n=81$\\
  &    $k$ &  $k$ &  $k$ 
\end{tabular}
\caption{Graphs showing the fraction of times Problem~\eqref{eq:opt_problem} achieves exact demixing over 10 trials with random signs and supports for different numbers of spectral lines $k$ and outliers $s$. The minimum separation of the spectral lines is $ 2 / (n-1)$. Each column shows results for a different value of the regularization parameter $\lambda$. Each row shows results for a different number of measurements $n$. The simulations are carried out using CVX~\cite{cvx}.}
\label{fig:phase_transitions_lambda}
\end{figure}

In this section we investigate the performance of the method described in Section~\ref{sec:main_results}. To do this, we apply the SDP-based approach described in Section~\ref{sec:sdp} to data of the form~\eqref{eq:model_data} varying the different parameters of interest. Fixing either the number of spectral lines $k$ or the number of outliers $s$ allows us to visualize the performance of the method for a range of values of the line spectrum's minimum separation $\Delta$ (defined by~\eqref{eq:min_distance}). The results are shown in Figure~\ref{fig:phase_transitions_mindist}. We observe that in every instance there is a rapid phase transition between the values at which the method always achieves exact demixing and the values at which it fails. The minimum separation at which this phase transition takes place is between $1/\brac{n-1}$ and $2/\brac{n-1}$, which is smaller than the minimum-separation required by Theorem~\ref{theorem:main}. We conjecture that if we allow for arbitrary sign patterns, the phase transition would occur near $2/\brac{n-1}$. In fact, if we constrain the amplitudes of the spectral lines to be real instead of complex, the phase transition occurs at a higher minimum separation, as shown in~\cite[Figure 7]{superres_new}. 

In order to investigate the effect of the regularization parameter on the performance of the algorithm, we fix $\Delta$ and perform demixing for different values of $k$ and $s$. The results are shown in Figure~\ref{fig:phase_transitions_lambda}. As suggested by Lemma~\ref{lemma:trimmed_sol}, for fixed $s$ the method succeeds for all values of $k$ below a certain limit, and vice versa when we vary $s$. Since $\lambda$ weights the effect of the terms that promote sparsity of the two different components in our mixture model, it is no surprise that varying it affects the tradeoff between the number of spectral lines and of spikes that we can demix. For smaller $\lambda$ the sparsity-inducing term affecting the multisinusoidal component is stronger, so the method succeeds for mixtures with smaller $k$ and larger $s$. Analogously, for larger $\lambda$ the sparsity-inducing term affecting the outlier component is stronger, so the method succeeds for mixtures with larger $k$ and smaller $s$.   

\subsection{Comparison with matrix-completion based denoising}
\label{sec:experiments_denoising}
In this section, we compare the SDP-based atomic-norm denoising method described in Section~\ref{sec:atomic_norm} to the matrix-completion based denoising method from~\cite{Chi2014spectral}. Both algorithms are implemented using CVX~\cite{cvx} and applied to data following model~\eqref{eq:data_denoising}. In general we observe that both methods either \emph{succeed}, achieving extremely small errors (the relative MSE\footnote{The relative MSE is defined as the ratio between the $\ell_2$-norm of the difference between the \emph{clean} samples $\vct{g}$ and the estimate divided by $\normTwo{\vct{g}}$.} is smaller than $10^{-8}$), or \emph{fail}, producing very large errors. We compare the performance by recording whether the methods succeed or fail in denoising randomly generated signals for different number of spectral lines $k$ and outliers $s$. To provide a more complete picture, we repeat the simulations for different values of the regularization parameters ($\lambda$ for atomic-norm denoising and $\theta$ for matrix-completion denoising) that govern the sparsity-inducing terms of the corresponding optimization problems. The values of $\lambda$ and $\theta$ are chosen separately to yield the best possible performance. 

\begin{figure}[tp]
\hspace{-0.4cm}
\begin{tabular}{>{\centering\arraybackslash}m{0.02\linewidth} >{\centering\arraybackslash}m{0.28\linewidth} >{\centering\arraybackslash}m{0.28\linewidth}  >{\centering\arraybackslash}m{0.3\linewidth} }
 &&Atomic norm \\ \\
  & $\lambda=0.1$ &  $\lambda=0.15$ & $\lambda = 0.2$  \\ \\
  $s$ & 
 \includegraphics{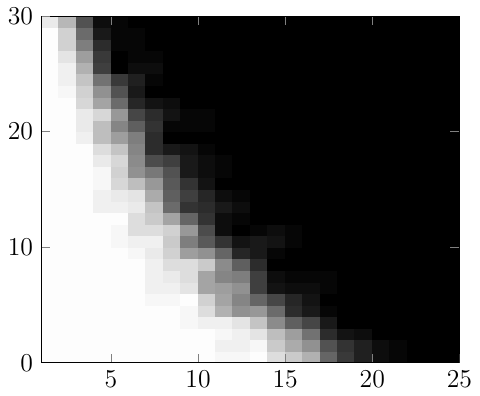}
& 
 \includegraphics{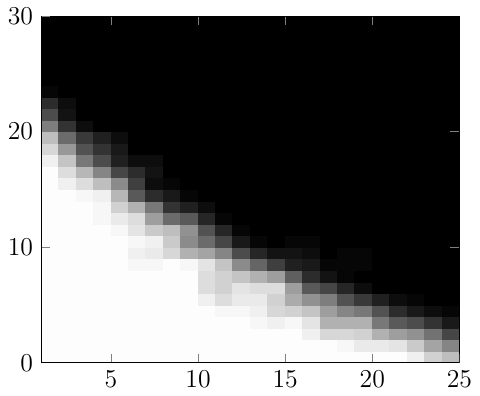}
& 
 \includegraphics{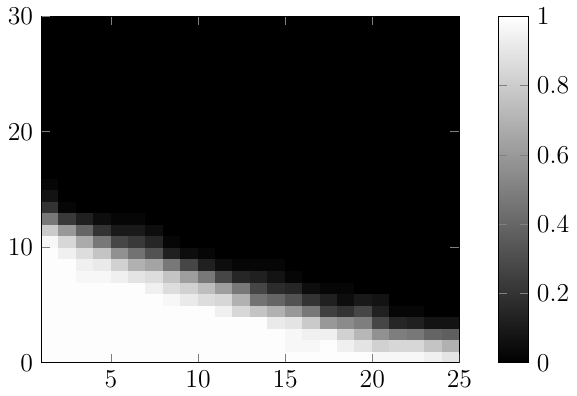}
  \\
\\
$s$ &  
 \includegraphics{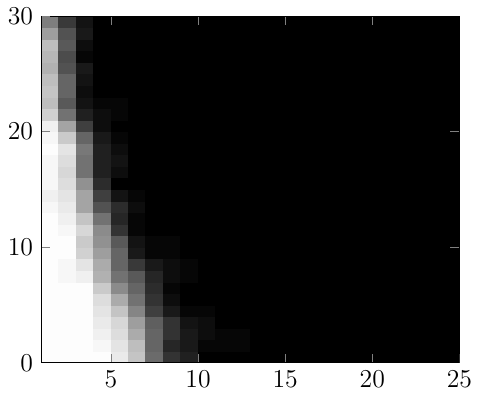}
& 
 \includegraphics{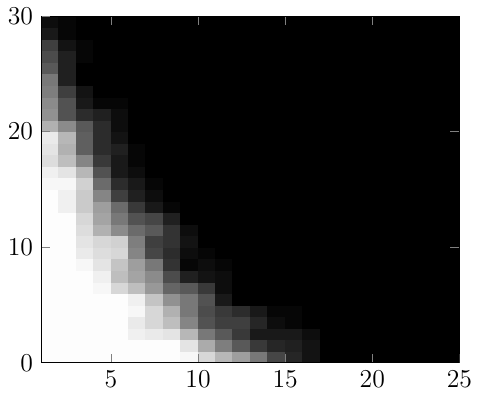}
& 
 \includegraphics{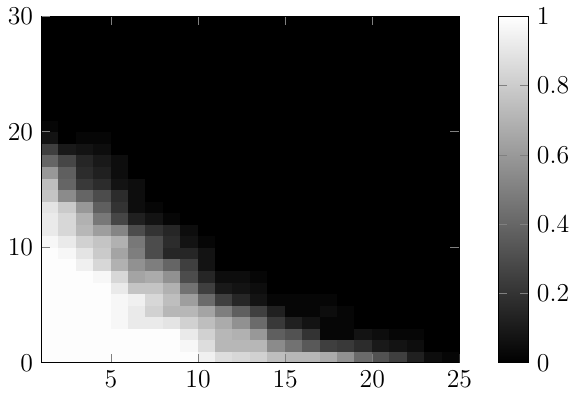}
\\
    &  $k$ &  $k$ &  $k$ \\ \\
   & $\theta=0.15$ &  $\theta=0.2$ & $\theta = 0.25$ \\ \\
& & Matrix completion \\ $ $
\end{tabular}
\caption{Graphs showing the fraction of times Problem~\eqref{eq:denoising_sdp} (top row) and the matrix-completion approach from~\cite{Chi2014spectral} (bottom row) achieve exact denoising for different values of their respective regularization parameters over 10 trials with random signs and supports. The minimum separation of the spectral lines is $ 2 / (n-1)$ and the number of data is $n=61$. The simulations are carried out using CVX~\cite{cvx}.}
\label{fig:comp_at_mc}
\end{figure}

Figure~\ref{fig:comp_at_mc} shows the results. We observe that atomic-norm denoising consistently outperforms matrix-completion denoising across regimes in which the methods achieve different tradeoffs between the values of $k$ and $s$. In addition, atomic-norm denoising is faster: the average running time for each trial is 3.25 seconds with a standard deviation of 0.30 s, whereas the average running time for the matrix-completion approach is of 11.1 s with a standard deviation of 1.32 s. The experiments were carried out on an Intel Xeon desktop computer with a 3.5 GHz CPU and 24 GB of RAM.
\section{Conclusion and future research directions}
\label{sec:conclusion}

In this work we propose an optimization-based method for spectral super-resolution in the presence of outliers and characterize its performance theoretically. In addition, we describe how to implement the approach using semidefinite programming, discuss its connection to atomic-norm denoising and present a greedy demixing algorithm with a promising empirical performance. Our results suggest the following directions for future research.

\begin{itemize}
\item Proving a result similar to Theorem~\ref{theorem:main} without the assumption that the phases of the different components are random. This would require showing that the dual-polynomial construction in Section~\ref{sec:random_interp} is valid, without leveraging the concentration bounds that we use for our proof. It is unclear whether this is possible because the interpolation kernel $K$ does not display a good asymptotic decay, as shown in Figure~\ref{fig:kernel}. Note that if the amplitudes of the sparse noise $\vct{z}$ are constrained to be real, then a derandomization argument similar to the one in~\cite[Theorem 2.1]{candes2011robust} allows to establish the same guarantees as Theorem~\ref{theorem:main} for a sparse perturbation that has an arbitrary deterministic sign pattern.
\item Deriving guarantees for spectral super-resolution via the approach described in Section~\ref{sec:dense_noise} in the presence of dense and sparse noise. To achieve this, one could combine our dual polynomial construction with the techniques developed in~\cite{support_detection,superres_noisy, tang2014minimax}. In addition, it would be interesting to investigate the application of the method when the level of dense noise is unknown, as in~\cite{superres_unknown_noise}.
\item Developing fast algorithms to solve the semidefinite programs in Sections~\ref{sec:sdp} and \ref{sec:sdp_noise}. We have found that ADMM is effective for denoising, but the dual variable converges too slowly for it to be effective in super-resolving the line spectrum.
\item Investigating whether greedy demixing techniques, like the one in Section~\ref{sec:greedy}, can achieve the same performance as our convex-programming approach both empirically and theoretically. 
\item Considering other structured noise models, beyond sparse perturbations, which could be learnt from data by leveraging techniques such as dictionary learning~\cite{olshausen_field,mairal_bach_tutorial}. For instance, this could allow to deal with recurring interferences in radar applications.
\end{itemize}

\subsection*{Acknowledgements}
C.F. is generously supported by NSF award DMS-1616340. G.T. is generously supported by NSF award CCF-1464205.

\begin{small}
\bibliographystyle{abbrv}
\bibliography{rlse}

\begin{thebibliography}{10}

\bibitem{azais2015spike}
J.-M. Azais, Y.~De~Castro, and F.~Gamboa.
\newblock Spike detection from inaccurate samplings.
\newblock {\em Applied and Computational Harmonic Analysis}, 38(2):177--195,
  2015.

\bibitem{beatty1978use}
L.~G. Beatty, J.~D. George, and A.~Z. Robinson.
\newblock Use of the complex exponential expansion as a signal representation
  for underwater acoustic calibration.
\newblock {\em The Journal of the Acoustical Society of America},
  63(6):1782--1794, 1978.

\bibitem{berni1975target}
A.~J. Berni.
\newblock Target identification by natural resonance estimation.
\newblock {\em IEEE Transactions on Aerospace and Electronic systems},
  (2):147--154, 1975.

\bibitem{Bhaskar:2012tq}
B.~Bhaskar, G.~Tang, and B.~Recht.
\newblock Atomic norm denoising with applications to line spectral estimation.
\newblock {\em Signal Processing, IEEE Transactions on}, 61(23):5987--5999, Dec
  2013.

\bibitem{music_bienvenu}
G.~Bienvenu.
\newblock Influence of the spatial coherence of the background noise on high
  resolution passive methods.
\newblock In {\em Proceedings of the International Conference on Acoustics,
  Speech and Signal Processing}, volume~4, pages 306 -- 309, 1979.

\bibitem{borcea2002imaging}
L.~Borcea, G.~Papanicolaou, C.~Tsogka, and J.~Berryman.
\newblock Imaging and time reversal in random media.
\newblock {\em Inverse Problems}, 18(5):1247, 2002.

\bibitem{sparse_inverse_ben}
N.~Boyd, G.~Schiebinger, and B.~Recht.
\newblock The alternating descent conditional gradient method for sparse
  inverse problems.
\newblock Preprint.

\bibitem{boyd2011distributed}
S.~Boyd, N.~Parikh, E.~Chu, B.~Peleato, and J.~Eckstein.
\newblock Distributed optimization and statistical learning via the alternating
  direction method of multipliers.
\newblock {\em Foundations and Trends{\textregistered} in Machine Learning},
  3(1):1--122, 2011.

\bibitem{Boyd:2004uz}
S.~P. Boyd and L.~Vandenberghe.
\newblock {\em {Convex Optimization}}.
\newblock Cambridge Univ Pr, Mar. 2004.

\bibitem{superres_unknown_noise}
C.~Boyer, Y.~{De Castro}, and J.~Salmon.
\newblock Adapting to unknown noise level in sparse deconvolution.
\newblock Preprint.

\bibitem{bredies2013inverse}
K.~Bredies and H.~K. Pikkarainen.
\newblock Inverse problems in spaces of measures.
\newblock {\em ESAIM: Control, Optimisation and Calculus of Variations},
  19(1):190--218, 2013.

\bibitem{Candes:2012uf}
E.~J. Cand\`{e}s and C.~Fernandez-Granda.
\newblock Towards a mathematical theory of super-resolution.
\newblock {\em Communications on Pure and Applied Mathematics}, 67(6):906--956,
  Mar.

\bibitem{superres_noisy}
E.~J. {Cand\`{e}s} and C.~Fernandez-Granda.
\newblock Super-resolution from noisy data.
\newblock {\em Journal of Fourier Analysis and Applications}, 19(6):1229--1254,
  2013.

\bibitem{candes2011robust}
E.~J. Cand{\`e}s, X.~Li, Y.~Ma, and J.~Wright.
\newblock Robust principal component analysis?
\newblock {\em Journal of the ACM}, 58(3):11, 2011.

\bibitem{candes2011probabilistic}
E.~J. Candes and Y.~Plan.
\newblock A probabilistic and ripless theory of compressed sensing.
\newblock {\em Information Theory, IEEE Transactions on}, 57(11):7235--7254,
  2011.

\bibitem{Candes:2007es}
E.~J. Cand\`{e}s and J.~Romberg.
\newblock {Sparsity and incoherence in compressive sampling}.
\newblock {\em Inverse Problems}, 23(3):969--985, Apr. 2007.

\bibitem{Candes:2006eq}
E.~J. Cand\`{e}s, J.~Romberg, and T.~Tao.
\newblock {Robust uncertainty principles: exact signal reconstruction from
  highly incomplete frequency information}.
\newblock {\em IEEE Trans. Inf. Thy.}, 52(2):489--509, Feb. 2006.

\bibitem{Candes:2005cs}
E.~J. Cand\`{e}s and T.~Tao.
\newblock {Decoding by linear programming}.
\newblock {\em IEEE Trans. Inf. Thy.}, 51(12):4203--4215, 2005.

\bibitem{candes2006near}
E.~J. Candes and T.~Tao.
\newblock Near-optimal signal recovery from random projections: Universal
  encoding strategies?
\newblock {\em IEEE transactions on information theory}, 52(12):5406--5425,
  2006.

\bibitem{Candes:2010jb}
E.~J. Cand\`{e}s and T.~Tao.
\newblock {The Power of Convex Relaxation: Near-Optimal Matrix Completion}.
\newblock {\em IEEE Trans. Inf. Thy.}, 56(5):2053--2080, 2010.

\bibitem{carriere1992high}
R.~Carriere and R.~L. Moses.
\newblock High resolution radar target modeling using a modified {P}rony
  estimator.
\newblock {\em IEEE Transactions on Antennas and Propagation}, 40(1):13--18,
  1992.

\bibitem{chandrasekaran2012convex}
V.~Chandrasekaran, B.~Recht, P.~A. Parrilo, and A.~S. Willsky.
\newblock The convex geometry of linear inverse problems.
\newblock {\em Foundations of Computational Mathematics}, 12(6):805--849, 2012.

\bibitem{chandrasekaran2011rank}
V.~Chandrasekaran, S.~Sanghavi, P.~A. Parrilo, and A.~S. Willsky.
\newblock Rank-sparsity incoherence for matrix decomposition.
\newblock {\em SIAM Journal on Optimization}, 21(2):572--596, 2011.

\bibitem{chen2001atomic}
S.~S. Chen, D.~L. Donoho, and M.~A. Saunders.
\newblock Atomic decomposition by basis pursuit.
\newblock {\em SIAM review}, 43(1):129--159, 2001.

\bibitem{Chi2014spectral}
Y.~Chen and Y.~Chi.
\newblock Robust spectral compressed sensing via structured matrix completion.
\newblock {\em Information Theory, IEEE Transactions on}, 60(10):6576--6601,
  Oct 2014.

\bibitem{supportPursuit}
Y.~De~Castro and F.~Gamboa.
\newblock Exact reconstruction using {B}eurling minimal extrapolation.
\newblock {\em Journal of Mathematical Analysis and Applications},
  395(1):336--354.

\bibitem{deProny:tg}
B.~G.~R. de~Prony.
\newblock {Essai {\'e}xperimental et analytique: sur les lois de la
  dilatabilit{\'e} de fluides {\'e}lastique et sur celles de la force expansive
  de la vapeur de l'alkool, {\`a} diff{\'e}rentes temp{\'e}ratures}.
\newblock {\em Journal de l'{\'e}cole Polytechnique}, 1(22):24--76, 1795.

\bibitem{donoho2006compressed}
D.~L. Donoho.
\newblock Compressed sensing.
\newblock {\em IEEE Transactions on information theory}, 52(4):1289--1306,
  2006.

\bibitem{donoho2001uncertainty}
D.~L. Donoho and X.~Huo.
\newblock Uncertainty principles and ideal atomic decomposition.
\newblock {\em Information Theory, IEEE Transactions on}, 47(7):2845--2862,
  2001.

\bibitem{donoho1989uncertainty}
D.~L. Donoho and P.~B. Stark.
\newblock Uncertainty principles and signal recovery.
\newblock {\em SIAM Journal on Applied Mathematics}, 49(3):906--931, 1989.

\bibitem{dragotti2014sparse}
P.~L. Dragotti and Y.~M. Lu.
\newblock On sparse representation in {F}ourier and local bases.
\newblock {\em IEEE Transactions on Information Theory}, 60(12):7888--7899,
  2014.

\bibitem{Dumitrescu:2007vw}
B.~Dumitrescu.
\newblock {\em {Positive Trigonometric Polynomials and Signal Processing
  Applications}}.
\newblock Springer Verlag, Feb. 2007.

\bibitem{peyreduval}
V.~Duval and G.~Peyr{\'e}.
\newblock Exact support recovery for sparse spikes deconvolution.
\newblock {\em Foundations of Computational Mathematics}, pages 1--41, 2015.

\bibitem{eftekhari2015greed}
A.~Eftekhari and M.~B. Wakin.
\newblock Greed is super: A fast algorithm for super-resolution.
\newblock Preprint.

\bibitem{fannjiang2012coherence}
A.~Fannjiang and W.~Liao.
\newblock Coherence pattern-guided compressive sensing with unresolved grids.
\newblock {\em SIAM Journal on Imaging Sciences}, 5(1):179--202, 2012.

\bibitem{faxin2001effective}
Y.~Faxin, S.~Yiying, and L.~Yongtan.
\newblock An effective method of anti-impulsive-disturbance for ship-target
  detection in hf radar.
\newblock In {\em Radar, 2001 CIE International Conference on, Proceedings},
  pages 372--375. IEEE, 2001.

\bibitem{support_detection}
C.~Fernandez-Granda.
\newblock Support detection in super-resolution.
\newblock In {\em Proceedings of the 10th International Conference on Sampling
  Theory and Applications}, pages 145--148, 2013.

\bibitem{superres_new}
C.~Fernandez-Granda.
\newblock Super-resolution of point sources via convex programming.
\newblock {\em Information and Inference}, 2016.

\bibitem{cvx}
M.~Grant and S.~Boyd.
\newblock {CVX}: Matlab software for disciplined convex programming, version
  1.21.
\newblock \url{../../cvx}, Apr. 2011.

\bibitem{Gross:2009id}
D.~Gross.
\newblock Recovering low-rank matrices from few coefficients in any basis.
\newblock {\em IEEE Trans. Inf. Thy.}, 57(3):1548--1566, Mar. 2009.

\bibitem{windows_harris}
F.~Harris.
\newblock On the use of windows for harmonic analysis with the discrete
  {F}ourier transform.
\newblock {\em Proceedings of the IEEE}, 66(1):51 -- 83, 1978.

\bibitem{lagarias1998convergence}
J.~C. Lagarias, J.~A. Reeds, M.~H. Wright, and P.~E. Wright.
\newblock Convergence properties of the nelder--mead simplex method in low
  dimensions.
\newblock {\em SIAM Journal on optimization}, 9(1):112--147, 1998.

\bibitem{leonowicz2003advanced}
Z.~Leonowicz, T.~Lobos, and J.~Rezmer.
\newblock Advanced spectrum estimation methods for signal analysis in power
  electronics.
\newblock {\em IEEE Transactions on Industrial Electronics}, 50(3):514--519,
  2003.

\bibitem{li2013compressed}
X.~Li.
\newblock Compressed sensing and matrix completion with constant proportion of
  corruptions.
\newblock {\em Constructive Approximation}, 37(1):73--99, 2013.

\bibitem{lu2010impulsive}
X.~Lu, J.~Wang, A.~Ponsford, and R.~Kirlin.
\newblock Impulsive noise excision and performance analysis.
\newblock In {\em 2010 IEEE Radar Conference}, pages 1295--1300. IEEE, 2010.

\bibitem{mairal_bach_tutorial}
J.~Mairal, F.~Bach, and J.~Ponce.
\newblock Sparse modeling for image and vision processing.
\newblock {\em arXiv preprint arXiv:1411.3230}, 2014.

\bibitem{mp}
S.~G. Mallat and Z.~Zhang.
\newblock Matching pursuits with time-frequency dictionaries.
\newblock {\em IEEE Transactions on Signal Processing}, 41(12):3397--3415,
  1993.

\bibitem{mccoy2014sharp}
M.~B. McCoy and J.~A. Tropp.
\newblock Sharp recovery bounds for convex demixing, with applications.
\newblock {\em Foundations of Computational Mathematics}, 14(3):503--567, 2014.

\bibitem{moitra_superres}
A.~Moitra.
\newblock Super-resolution, extremal functions and the condition number of
  {V}andermonde matrices.
\newblock In {\em Proceedings of the 47th Annual ACM Symposium on Theory of
  Computing (STOC)}, 2015.

\bibitem{olshausen_field}
B.~A. Olshausen and D.~Field.
\newblock Emergence of simple-cell receptive field properties by learning a
  sparse code for natural images.
\newblock {\em Nature}, 381(6583):607--609, 1996.

\bibitem{omp}
Y.~C. Pati, R.~Rezaiifar, and P.~Krishnaprasad.
\newblock Orthogonal matching pursuit: Recursive function approximation with
  applications to wavelet decomposition.
\newblock In {\em 27th Asilomar Conference on Signals, Systems and Computers},
  pages 40--44. IEEE, 1993.

\bibitem{rao2014forward}
N.~Rao, P.~Shah, and S.~Wright.
\newblock Forward?backward greedy algorithms for signal demixing.
\newblock In {\em 2014 48th Asilomar Conference on Signals, Systems and
  Computers}, pages 437--441. IEEE, 2014.

\bibitem{rao2015forward}
N.~Rao, P.~Shah, and S.~Wright.
\newblock Forward--backward greedy algorithms for atomic norm regularization.
\newblock {\em IEEE Transactions on Signal Processing}, 63(21):5798--5811,
  2015.

\bibitem{rockafellar1974conjugate}
R.~Rockafellar.
\newblock {\em Conjugate Duality and Optimization}.
\newblock Regional conference series in applied mathematics. Society for
  Industrial and Applied Mathematics, 1974.

\bibitem{tv}
L.~I. Rudin, S.~Osher, and E.~Fatemi.
\newblock Nonlinear total variation based noise removal algorithms.
\newblock {\em Physica D: Nonlinear Phenomena}, 60(1):259--268, 1992.

\bibitem{Schaeffer:1941wm}
A.~Schaeffer.
\newblock {Inequalities of A. Markoff and S. Bernstein for polynomials and
  related functions}.
\newblock {\em Bull. Amer. Math. Soc.}, 47, Nov. 1941.

\bibitem{music_schmidt}
R.~Schmidt.
\newblock {Multiple emitter location and signal parameter estimation}.
\newblock {\em Antennas and Propagation, IEEE Transactions on}, 34(3):276--280,
  1986.

\bibitem{slepian}
D.~{Slepian}.
\newblock {Prolate spheroidal wave functions, {F}ourier analysis, and
  uncertainty. V - The discrete case}.
\newblock {\em Bell System Technical Journal}, 57:1371--1430, 1978.

\bibitem{smith2008introduction}
J.~O. Smith.
\newblock {\em Introduction to digital filters: with audio applications},
  volume~2.
\newblock Julius Smith, 2008.

\bibitem{stoica2011new}
P.~Stoica, P.~Babu, and J.~Li.
\newblock New method of sparse parameter estimation in separable models and its
  use for spectral analysis of irregularly sampled data.
\newblock {\em IEEE Transactions on Signal Processing}, 59(1):35--47, 2011.

\bibitem{stoicaNLSmultimodal}
P.~Stoica, R.~Moses, B.~Friedlander, and T.~Soderstrom.
\newblock Maximum likelihood estimation of the parameters of multiple sinusoids
  from noisy measurements.
\newblock {\em IEEE Transactions on Acoustics, Speech and Signal Processing},
  37(3):378--392, 1989.

\bibitem{Stoica:2005wf}
P.~Stoica and R.~L. Moses.
\newblock {\em {Spectral analysis of signals}}.
\newblock Prentice Hall, Upper Saddle River, New Jersey, 1 edition, 2005.

\bibitem{su_corrupted_fourier}
D.~Su.
\newblock Compressed sensing with partially corrupted {F}ourier measurements.
\newblock Preprint.

\bibitem{tang_resolution}
G.~Tang.
\newblock Resolution limits for atomic decompositions via {M}arkov-{B}ernstein
  type inequalities.
\newblock In {\em Proceedings of the 10th International Conference on Sampling
  Theory and Applications}, pages 548--552, 2015.

\bibitem{tang2014minimax}
G.~Tang, B.~Bhaskar, and B.~Recht.
\newblock Near minimax line spectral estimation.
\newblock {\em Information Theory, IEEE Transactions on}, 61(1):499--512, Jan
  2015.

\bibitem{tang2012offgrid}
G.~Tang, B.~Bhaskar, P.~Shah, and B.~Recht.
\newblock Compressed sensing off the grid.
\newblock {\em Information Theory, IEEE Transactions on}, 59(11):7465--7490,
  Nov 2013.

\bibitem{tang2013discretize}
G.~Tang, B.~N. Bhaskar, and B.~Recht.
\newblock Sparse recovery over continuous dictionaries-just discretize.
\newblock In {\em 2013 Asilomar Conference on Signals, Systems and Computers},
  pages 1043--1047, Nov 2013.

\bibitem{tang2014robust}
G.~Tang, P.~Shah, B.~N. Bhaskar, and B.~Recht.
\newblock Robust line spectral estimation.
\newblock In {\em Signals, Systems and Computers, 2014 48th Asilomar Conference
  on}, pages 301--305. IEEE, 2014.

\bibitem{tibshirani1996regression}
R.~Tibshirani.
\newblock Regression shrinkage and selection via the lasso.
\newblock {\em Journal of the Royal Statistical Society. Series B
  (Methodological)}, pages 267--288, 1996.

\bibitem{tropp2008linear}
J.~A. Tropp.
\newblock On the linear independence of spikes and sines.
\newblock {\em Journal of Fourier Analysis and Applications}, 14(5-6):838--858,
  2008.

\bibitem{Tropp:2011vm}
J.~A. Tropp.
\newblock User-friendly tail bounds for sums of random matrices.
\newblock {\em Found. Comput. Math.}, 12(4):389--434, Aug. 2011.

\bibitem{viti1997prony}
V.~Viti, C.~Petrucci, and P.~Barone.
\newblock Prony methods in {NMR} spectroscopy.
\newblock {\em International Journal of Imaging Systems and Technology},
  8(6):565--571, 1997.

\bibitem{yang2015gridless}
Z.~Yang and L.~Xie.
\newblock On gridless sparse methods for line spectral estimation from complete
  and incomplete data.
\newblock {\em IEEE Transactions on Signal Processing}, 63(12):3139--3153.

\bibitem{zeng2013_}
W.-J. Zeng, H.~So, and L.~Huang.
\newblock $\ell_p$-music: Robust direction-of-arrival estimator for impulsive
  noise environments.
\newblock {\em IEEE Transactions on Signal Processing}, 61:4296--4308, 2013.

\bibitem{le_paper}
L.~Zheng and X.~Wang.
\newblock Improved {NN-JPDAF} for joint multiple target tracking and feature
  extraction.
\newblock Preprint.

\end{thebibliography}
\end{small}

\appendix

\section{Proof of Lemma~\ref{lemma:trimmed_sol}}
\label{proof:trimmed_sol}
For any vector $\vct{u}$ and any atomic measure $\nu$, we denote by $\vct{u}_{\ml{S}}$ and $\nu_{\ml{S}}$ the restriction of $\vct{u}$ and $\nu$ to the subset of their support indexed by a set $\ml{S}$. Let $\keys{\hat{\mu},\vct{ \hat{z} }}$ be any solution to Problem~\eqref{eq:opt_problem} applied to $\vct{y'}$. The pair $\keys{\hat{\mu}+\mu_{T/T'},\vct{ \hat{z} }+\vct{ z }_{\Omega/\Omega'}}$ is feasible for Problem~\eqref{eq:opt_problem} applied to $\vct{y}$ since
\begin{align}
\mathcal{F}_{n} \, \hat{\mu} +\mathcal{F}_{n} \, \mu_{T/T'}+ \vct{ \hat{z} }+\vct{ z }_{\Omega/\Omega'} & = \vct{y'} +\mathcal{F}_{n} \, \mu_{T/T'} + \vct{ z }_{\Omega/\Omega'}\\
& = \mathcal{F}_{n} \, \mu' +\mathcal{F}_{n} \, \mu_{T/T'} + \vct{z'} +\vct{ z }_{\Omega/\Omega'} \\
& = \mathcal{F}_{n} \, \mu + \vct{z}\\
& = \vct{y}. 
\end{align}
By the triangle inequality and the assumption that $\keys{\mu,\vct{z}}$ is the unique solution to Problem~\eqref{eq:opt_problem} applied to $\vct{y'}$, this implies
\begin{align}
\normTV{\mu} + \lambda \normOne{\vct{z}} & <  \normTV{\hat{\mu}+\mu_{T/T'}} + \lambda \normOne{\vct{ \hat{z} }+\vct{ z }_{\Omega/\Omega'}} \\
& \leq \normTV{\hat{\mu}} +\normTV{\hat{\mu}_{T/T'}} + \lambda \normOne{\vct{ \hat{z} }} + \lambda \normOne{\vct{ z }_{\Omega/\Omega'}}
\end{align}
unless $ \hat{\mu}+\mu_{T/T'} = \mu$ and $\vct{ \hat{z} }+\vct{z}_{\Omega/\Omega'} = \vct{z}$, so that
\begin{align}
\normTV{\mu'} + \lambda \normOne{\vct{z'}} & = \normTV{\mu} - \normTV{\mu_{T/T'}} + \lambda \normOne{\vct{z}} -   \lambda \normOne{\vct{ z }_{\Omega/\Omega'}} \\
& <  \normTV{\hat{\mu}} + \lambda \normOne{\vct{ \hat{z} }},
\end{align}
unless $ \hat{\mu} = \mu$ and $\vct{ \hat{z} } = \vct{z'}$. We conclude that $\keys{\mu',\vct{z'}}$ must be the unique solution to Problem~\eqref{eq:opt_problem} applied to $\vct{y'}$.

\section{Atomic-norm denoising}
\subsection{Proof of Lemma~\ref{lemma:atomic_norm_dual}}
\label{proof:atomic_norm_dual}
We define a scaled dual norm $ \|\cdot\|_{\A'} := \|\cdot\|_\A / \sqrt{n}$. The dual norm of $\|\cdot\|_{\A'}$ is
\begin{align}\label{eqn:dualatomicnorm}
  \|\vct{ \eta }  \|_{\A'} ^{\ast} & = \sup_{ \atomicnorm{ \vct{ \tilde{g}}} \leq \sqrt{n}} \PROD{ \vct{\eta}}{ \vct{\tilde{g}}} \\
  & = \sup_{\phi \in [0, 2\pi), f \in [0, 1]} \PROD{ \vct{\eta}}{  \sqrt{n} e^{i\phi} \vct{a}\brac{f,0} } \\
 & =  \sup_{f \in \sqbr{0, 1} } \abs{ \PROD{ \vct{\eta} }{ \sqrt{n} \vct{a}\brac{f,0} } } \\
 & = \normInf{ \mathcal{F}_{n}^{\ast} \, \vct{\eta}}.
 \end{align}
The result now follows from the fact that the dual of~\ref{eq:atomic_norm_min} is 
\begin{align}
\label{eq:dual_sinesspikes}
\max_{ \vct{\eta} \in \C^{n}} \;   \PROD{\vct{y}}{\vct{\eta}}  \quad \text{subject to}
\quad &   \|\vct{ \eta }  \|_{\A'} ^{\ast} \leq 1, \\
&  \normInf{\vct{\eta}}  \leq \lambda,
\end{align}
by a standard argument~\cite[Section 2.1]{chandrasekaran2012convex}.

\subsection{Proof of Corollary \ref{cor:atomic_denoising}}
\label{proof:atomic_denoising}
The corollary is a direct consequence of the following lemma, which establishes that the dual polynomial whose existence we establish in Proposition~\ref{prop:dual_pol} also guarantees that solving Problem~\eqref{eq:atomic_norm_min} achieves exact demixing. 
\begin{lemma}
If there exists a trigonometric polynomial $Q$ satisfying the conditions listed in Proposition~\ref{proposition:dual_polynomial}, then $\vct{g}$ and $\vct{z}$ are the unique solutions to Problem~\eqref{eq:atomic_norm_min}.
\end{lemma}
\begin{proof}
In the case of the atoms defined by~\eqref{eq:atoms}, the atomic norm is given by
\begin{align}
\atomicnorm{ \vct{u} } & = \inf_{\substack{\keys{\vct{\tilde{x}}_j \geq 0}, \; \keys{\phi_j \in [0, 2\pi)} \\ \keys{f_j \in [0, 1]}}} \keys{ \sum_j \vct{\tilde{x}}_j: \vct{u} = \sum_j \vct{\tilde{x}}_j \vct{a}\brac{f_j, \phi_j}  },
\end{align}
so that 
\begin{align}
\atomicnorm{ \vct{g} } & \leq \normOne{\vct{x}} \quad \text{due to~\eqref{eq:g}} \\
& = \normTV{\mu} . \label{eq:ineq_g_mu}
\end{align}
By construction,
\begin{align}
\PROD{\vct{q}}{\vct{y}} & = \PROD{\vct{q}}{ \vct{g} + \vct{z}} \\
& = \PROD{\ml{F}_n^{\ast}\vct{q}}{ \mu } + \PROD{\vct{q}}{ \vct{z}} \\
& = \int_{\sqbr{0,1}} \overline{Q \brac{f}} \, \diff{\mu \brac{f}} + \lambda \sum_{l=1}^{s} \abs{\vct{z}_l} \\
& = \normTV{\mu} + \lambda \normOne{ \vct{ z }}. \label{eq:eq_q_y_mu_z}
\end{align}
Consider an arbitrary feasible pair $\keys{ \vct{g'}, \vct{z'}}$ different from $\keys{ \vct{g}, \vct{z}}$, such that $\vct{z'}$ has nonzero support $\Omega'$ and
\begin{align}
\vct{g'}  = \sqrt{n}  \sum_{f_j \in T'} \vct{x'}_j \vct{a}\brac{f_j,0}, \quad \atomicnorm{ \vct{g'} } := \sum_{f_j \in T'} \abs{\vct{x'}_j }
\end{align} 
for a sequence of complex coefficients $\vct{x'}$ and a set of frequency locations $T' \subseteq \sqbr{0,1}$. 

Note that as long as $k + s \leq n$ (recall that $k := \abs{T}$ and $s:=\abs{\Omega}$) then either $T \neq T'$ or $\Omega \neq \Omega'$. The reason is that under that condition any set formed by $k$ atoms of the form $\vct{a}\brac{f_j,0}$ and $s$ vectors with cardinality one is linearly independent (this is equivalent to the matrix $\MAT{F_T & \Id_{\Omega}}$ in Section~\ref{proof:F_T_I} being full rank), so that if both $T = T'$ and $\Omega = \Omega'$ then $\vct{g} + \vct{z}= \vct{g'} + \vct{z}$ would imply that $\vct{g}= \vct{g'}$ and $ \vct{z}= \vct{z}$ (and we are assuming this is not the case).

By conditions~\eqref{eqn:condition:Q1} and \eqref{eqn:condition:Q2}
\begin{alignat}{2}
\sqrt{n} \PROD{\vct{q}}{ \vct{a}\brac{f_j,0}} & = Q\brac{f_j} \\
& = \frac{\vct{x}_j}{\abs{\vct{x}_j}}, \qquad && \forall f_j \in T, \label{eq:Q_a_1}\\
\sqrt{n} \PROD{\vct{q}}{ \vct{a}\brac{f_j,0}} & = \abs{ Q \brac{ f } } \\ 
&  <  1, && \forall f \in T^c. \label{eq:Q_a_2}
\end{alignat}
We have
\begin{align}
\atomicnorm{ \vct{g} } + \lambda \normOne{ \vct{ z }} & \leq \PROD{\vct{q}}{\vct{y}} \quad \text{by~\eqref{eq:ineq_g_mu} and~\eqref{eq:eq_q_y_mu_z}}\\
& =  \PROD{\vct{q}}{\vct{g'}} + \PROD{\vct{q}}{\vct{z'}} \\
& = \sqrt{n} \sum_{f_j \in T'} \vct{x'}_j \PROD{\vct{q}}{\vct{a}\brac{f,0}} + \PROD{\vct{q}_{\Omega'}}{\vct{z'}} \\
& < \sqrt{n} \sum_{f_j \in T'} \abs{\vct{x'}_j} + \lambda \sum_{l \in \Omega'} \abs{\vct{z'}_j} \label{eq:ineq_xprime_zprime}\\
& = \atomicnorm{ \vct{g'} } + \lambda \normOne{ \vct{ z' }} 
\end{align}
where \eqref{eq:ineq_xprime_zprime} follows from conditions~\eqref{eqn:condition:q1} and \eqref{eqn:condition:q2}, \eqref{eq:Q_a_1}, \eqref{eq:Q_a_2} and the fact that either $T \neq T'$ or $\Omega \neq \Omega'$. We conclude that $\keys{ \vct{g}, \vct{z}}$ must be the unique solution to Problem~\eqref{eq:atomic_norm_min}.
 \end{proof}


\section{Proof of Proposition~\ref{proposition:dual_polynomial}}
\label{proof:dual_polynomial}
For any vector $\vct{u}$ and any atomic measure $\nu$, we denote by $\vct{u}_{\ml{S}}$ and $\nu_{\ml{S}}$ the restriction of $\vct{u}$ and $\nu$ to the subset of their support indexed by a set $\ml{S}$ ($\vct{u}_{\ml{S}}$ has the same dimension as $\vct{u}$ and $\nu_{\ml{S}}$ is still a measure in the unit interval). Let us consider an arbitrary feasible pair $\mu'$ and $\vct{z'}$, such that $\mu'\neq \mu$ or $\vct{z'}\neq \vct{z}$. Due to the constraints of the optimization problem, $\mu'$ and $\vct{z'}$ satisfy 
\begin{align}
\vct{y}=\mathcal{F}_n \, \mu+\vct{z}=\mathcal{F}_n \, \mu'+\vct{z'}. \label{eq:constraint}
\end{align}
The following lemma establishes that $\mu_{T^c}'$ and $\vct{z}_{\Omega^c}'$ cannot both equal zero.
\begin{lemma}[Proof in Section~\ref{proof:F_T_I}]
\label{lemma:F_T_I}
If $\keys{\mu',\vct{z'}}$ is feasible and $\mu_{T^c}'$ and $\vct{z}_{\Omega^c}'$ both equal zero, then $\mu=\mu'$ and $\vct{z}=\vct{z'}$.
\end{lemma}
This lemma and the existence of $Q$ imply that the cost function evaluated at $\keys{\mu',\vct{z'}}$ is larger than at $\keys{\mu,\vct{z}}$:
\begin{align}
  \normTV{\mu'} +  \lambda \normOne{\vct{z'}} & = \normTV{\mu_{T}'} +\normTV{\mu_{T^c}'} +  \lambda \normOne{\vct{z}_{\Omega}'} +  \lambda \normOne{\vct{z}_{\Omega^c}'}\\
& > \normTV{\mu_{T}'} +\PROD{Q}{\mu_{T^c}'} +  \lambda \normOne{\vct{z}_{\Omega}'} +  \PROD{\vct{q}}{\vct{z}_{\Omega^c}'} \qquad \text{by Lemma~\ref{proof:F_T_I}, \eqref{eqn:condition:Q2} and \eqref{eqn:condition:q2}} \notag \\
& \geq \PROD{Q}{\mu'} +  \PROD{\vct{q}}{\vct{z'}} \qquad \text{by \eqref{eqn:condition:Q1} and \eqref{eqn:condition:q1}}\\
&=\PROD{\mathcal{F}_n^{\ast}\vct{q}}{\mu'} +  \PROD{\vct{q}}{\vct{z'}}\\
&=\PROD{\vct{q}}{\mathcal{F}_n \, \mu'+\vct{z'}}\\
&=\PROD{\vct{q}}{\mathcal{F}_n \, \mu+\vct{z}} \qquad \text{by \eqref{eq:constraint}} \\
&=\PROD{\mathcal{F}_n^{\ast}\vct{q}}{\mu} +  \PROD{\vct{q}}{\vct{z}}\\
&=\PROD{Q}{\mu} +  \PROD{\vct{q}}{\vct{z}}\\
& =   \normTV{\mu} +  \lambda \normOne{\vct{z}} \qquad \text{by \eqref{eqn:condition:Q1} and \eqref{eqn:condition:q1}}.
\end{align}
We conclude that $\keys{\mu,\vct{z}}$ must be the unique solution.

\subsection{Proof of Lemma~\ref{lemma:F_T_I}}
\label{proof:F_T_I}
If $\mu_{T^c}'$ and $\vct{z}_{\Omega^c}'$ both equal zero, then
\begin{align}
\mathcal{F}_n \, \mu+\vct{z} - \mathcal{F}_n \, \mu_{T}'-\vct{z}_{\Omega}'  & = \mathcal{F}_n \, \mu'+\vct{z'} - \mathcal{F}_n \, \mu_{T}'-\vct{z}_{\Omega}' \qquad \text{by \eqref{eq:constraint}} \\
& = \mathcal{F}_n \, \mu_{T^c}'+\vct{z}_{\Omega^c}'   \\
& = 0.  \label{eq:zero} 
\end{align}
We index the entries of $\Omega := \keys{i_1,i_2, \ldots,i_s}$ and define the matrix $\MAT{F_T & \Id_{\Omega}} \in \C^{n \times \brac{k+s}}$, where
\begin{align}
\brac{F_T}_{lj} & = e^{i2\pi l f_j} \quad \text{for } 1\leq l \leq n, \; 1\leq j \leq  k, \\
\brac{ \Id_{\Omega}}_{lj} & = 
\begin{cases} 1 \quad \text{if } l =i_j \\ 
0 \quad \text{otherwise}\end{cases}\quad \text{for } 1\leq l \leq n, \; 1\leq j \leq  s.
\end{align}
If $k + s \leq n$ then $\MAT{F_T & \Id_{\Omega}}$ is full rank (this follows from the fact that $F_T$ is a submatrix of a Vandermonde matrix). Equation~\eqref{eq:zero} implies
\begin{align}
\MAT{F_T & \Id_{\Omega}} \MAT{\vct{x}-\vct{x}' \\ \ml{P}_{\Omega}\vct{z} - \ml{P}_{\Omega} \vct{z'}} = 0,
\end{align}
where $\ml{P}_{\Omega} \vct{u}' \in \C^s$ is the subvector of $\vct{u}'$ containing the entries indexed by $\Omega$ and $\vct{x}' \in \C^T$ is the vector containing the amplitudes of $\mu'$ (recall that by assumption $\mu_{T^c}'=0$). We conclude that $\mu=\mu'$ and $\vct{z}=\vct{z'}$. 

\section{Proof of Lemma~\ref{lemma:c_amp}}
\label{proof:c_amp}
The vector of coefficients $\vct{c}$ equals the convolution of three rectangles of widths $2 \cdot \gammaOne \, m + 1$, $2 \cdot \gammaTwo \, m + 1$ and $2 \cdot \gammaThree \, m + 1$ and amplitudes $\brac{2 \cdot \gammaOne \, m + 1}^{-1}$, $\brac{2 \cdot \gammaTwo \, m + 1}^{-1}$ and $\brac{2 \cdot \gammaThree \, m + 1}^{-1}$. Some simple computations show that the amplitude of the convolution of three rectangles with unit amplitudes and widths $a_1 < a_2<a_3$ is bounded by $a_1 a_2$. An immediate consequence is that the amplitude of $\vct{c}$ is bounded by
\begin{align}
\normInf{ \vct{c} } & \leq \frac{ \brac{2 \cdot \gammaOne \, m + 1} \brac{2 \cdot \gammaTwo \, m + 1} }{ \brac{2 \cdot \gammaOne \, m + 1} \brac{2 \cdot \gammaTwo \, m + 1} \brac{2 \cdot \gammaThree \, m + 1}  } \\
& \leq \frac{ 1 }{ \brac{2 \cdot \gammaThree \, m + 1} } \\
& \leq \frac{1.3}{m}.
\end{align}

\section{Proof of Lemma~\ref{lemma:boundB}}
\label{proof:boundB}
To bound the operator norm of $B_{\Omega}$, we control the behavior of 
\begin{align}
H & \defeq B_{\Omega}B_{\Omega}^{\ast}\\
& = \sum_{l \in \Omega} \vct{b}\brac{l}\vct{b}\brac{l}^{\ast},
\end{align}
which concentrates around a scaled version of
\begin{align}
\bar{H} & := \sum_{l =-m}^{m} \vct{b}\brac{l}\vct{b}\brac{l}^{\ast}.
\end{align}
The following lemma bounds the operator norm of $\bar{H}$.
\begin{lemma}[Proof in Section~\ref{lemma:boundHbar}]
\label{lemma:boundHbar}
Under the assumptions of Theorem~\ref{theorem:main} 
\begin{align}
\norm{ \bar{H}} & \leq 260 \, \pi^2 n \, \log k.
\end{align}
\end{lemma}
By~\eqref{eq:cond_s} $s \leq C_s \, n \brac{ \log k \log \frac{n}{\epsilon}}^{-1}$ which together with the lemma implies 
\begin{align}
\label{eq:barH_bound_Cb}
\norm{ \frac{s}{n}\bar{H}} & \leq \frac{C_{B}^2 \, n}{2}  \brac{\log \frac{n}{\epsilon}}^{-1}
\end{align}
if we set $C_s$ small enough. The following lemma uses the matrix Bernstein inequality to control the deviation of $H$ from a scaled version of $\bar{H}$. 
\begin{lemma}[Proof in Section~\ref{proof:H_Hbar}]
\label{lemma:H_Hbar}
Under the assumptions of Theorem~\ref{theorem:main}
\begin{align}
\norm{ H - \frac{s}{n}\bar{H}} & \leq \frac{C_{B}^2 \, n}{2} \brac{\log \frac{n}{\epsilon}}^{-1}
\end{align}
with probability at least $1- \epsilon /5 $.
\end{lemma}
We conclude that 
\begin{align}
\norm{ B_{\Omega}} & \leq \sqrt{ \norm{H}} \\
& \leq \sqrt{ \frac{s}{n}\norm{ \bar{H}} + \norm{ H - \frac{s}{n}\bar{H}} } \\
& \leq C_{B}\, \sqrt{n} \brac{\log \frac{n}{\epsilon}}^{-\frac{1}{2}}
\end{align}
with probability at least $1- \epsilon /5 $ by the triangle inequality.

\subsection{Proof of Lemma~\ref{lemma:boundHbar}}
\label{proof:boundHbar}
We express the matrix $\bar{H}$ in terms of the Dirichlet kernel $\ml{D}_m$ of order $m$ defined in~\eqref{eq:Dirichlet_kernel} and its derivatives,
\begin{align}
\bar{H} = n\begin{bmatrix}
\bar{H}_0 & \bar{H}_1\\
- \bar{H}_1 & \bar{H_2}
\end{bmatrix}
\end{align}
where 
\begin{align}
 \brac{ \bar{H}_0 }_{j l} & = \ml{D}_m \brac{ f_j - f_l }, \quad 
\brac{ \bar{H}_1 }_{j l}  = \kappa \,\ml{D}_m ^{\brac{1}} \brac{ f_j - f_l }, \quad \brac{ \bar{H}_2
}_{j l} = -\kappa^2 \, \ml{D}_m^{\brac{2}} \brac{ f_j - f_l }.
\end{align}
In order to bound the operator norm of $\bar{H}$ we first establish some bounds on $\ml{D}_m ^{\brac{\ell}}$ for $\ell=0,1,2$. Due to how the kernel is normalized in~\eqref{eq:Dirichlet_kernel}, the magnitude of $\ml{D}_m$ is bounded by one. This yields a uniform bound on the magnitude of its derivatives by Bernstein's polynomial inequality.
\begin{theorem}[Bernstein's polynomial inequality~\cite{Schaeffer:1941wm}]
\label{theorem:bernstein_pol_ineq}
For any complex-valued polynomial $P$ of degree $N$
\begin{align}
\sup_{\abs{z} \leq 1} \abs{P^{\brac{1}} \brac{z}} \leq N \, \sup_{\abs{z} \leq 1} \abs{P\brac{z}}.
\end{align}
\end{theorem}
Applying the theorem, we have
\begin{align}
\label{eq:boundDirichlet}
\abs{ \ml{D}^{\brac{\ell}} _m \brac{f} } \leq \brac{2  \pi  m}^{\ell}. 
\end{align}
The following lemma allows us to control the tail of the Dirichlet kernel and its derivatives. 
\begin{lemma}[{\hspace{1sp}\cite[Section C.4]{superres_new}}]
\label{lemma:dirichlet}
If $m \geq 10^3$, for $f \geq 80 /m$
\begin{align}
\abs{ \ml{D}^{\brac{\ell}} _m \brac{f} } & \leq \frac{1.1 \, 2^{\ell - 2} \pi^{\ell} m^{\ell -1}}{f}.
\end{align}
\end{lemma}
We now combine these two bounds to control the sum of the magnitudes of $\ml{D}_m^{\brac{ \ell }}$ when evaluated at $T$ for $\ell = 0,1,2$. By the minimum-separation condition~\eqref{condition:minimum_separation}, if we fix $f_i \in T$ then there are at most 126 other frequencies in $T$ that are at a distance of $80/m$ or less from $f_i$. We bound those terms using~\eqref{eq:boundDirichlet} and deal with the rest by applying Lemma~\ref{lemma:dirichlet},
 \begin{align}
\sup_{f_i} \sum_{j = 1}^k \kappa^\ell \abs{ \ml{D}^{\brac{ \ell }} \brac{ f_i - f_j } } 
 & \leq 126 \pi^{\ell}  \kappa^\ell \sup_{f} \abs{ \ml{D}^{\brac{ \ell }} \brac{ f } } + 2\kappa^\ell  \sum_{j=1}^{ k } \sup_{\abs{f} \geq j \Delta_{\min}} \abs{ \ml{D}^{\brac{ \ell }} \brac{ f }} \\
 & \leq 126 \pi^{\ell} + \frac{1}{m^{\brac{l}}} \sum_{j=1}^{k} \frac{1.1 \, \pi^{\ell} m^{\ell -1}}{ 4j \Delta_{\min} } \quad \text{by Lemma~\ref{lemma:bound_kappa} and~\eqref{eq:boundDirichlet}}\\
 & \leq 130  \pi^{\ell} \log k \quad \text{since $\Delta_{\min} := \frac{\mindisthalf}{m}$ and $ \sum_{j=1}^{k} \frac{1}{j} \leq 1 + \log k \leq 2 \log k$}
\end{align}
as long as $k$ is larger than 2 (the argument can be easily modified if this is not the case).

By Gershgorin's circle theorem, the eigenvalues of $\bar{H}$, and consequently its operator norm, are bounded by
\begin{align}
 n \max_i \Bigg\{  & \sum_{j = 1}^k \abs{ \ml{D} \brac{ f_i - f_j } } + \sum_{j = 1}^k \kappa \abs{ \ml{D}^{\brac{ 1 }} \brac{ f_i - f_j } },  \\
&  \sum_{j = 1}^k \kappa \abs{ \ml{D}^{\brac{ 1 }} \brac{ f_i - f_j } } + \sum_{j = 1}^k \kappa^2 \abs{ \ml{D}^{\brac{ 2 }} \brac{ f_i - f_j } } \Bigg\}
    \leq 260 \, \pi^{2} n \log k.
\end{align}

\subsection{Proof of Lemma~\ref{lemma:H_Hbar}}
\label{proof:H_Hbar}
Under the assumptions of Theorem~\ref{theorem:main}
\begin{align}
H & = \sum_{l =-m}^{m} \delta_{\Omega}\brac{l} \vct{b}\brac{l}\vct{b}\brac{l}^{\ast},
\end{align}
where $\delta_{\Omega}\brac{-m}$, $\delta_{\Omega}\brac{-m+1}$, \ldots, $\delta_{\Omega}\brac{m}$ are iid Bernouilli random variables with parameter $\frac{s}{n}$. We control this sum of independent random matrices using the matrix Bernstein inequality.

\begin{theorem}[Matrix Bernstein inequality{~\cite[Theorem 1.4]{Tropp:2011vm}}]
\label{theorem:matrix_bernstein}
Let $\keys{X_l}$ be a finite sequence of independent zero-mean self-adjoint random matrices of dimension $d$ such that $\norm{X_l} \leq B$ almost surely for a certain constant $B$. For all $t \geq 0 $ and a positive constant $\sigma^2$
\begin{align}
\Pr \keys{ \norm{ \textstyle \sum_{l=-m}^{m} X_l } \geq t } \leq d \, \exp \brac{\frac{- t^2/2}{\sigma^2 + Bt / 3}} \quad \text{as long as} \quad  \norm{ \textstyle \sum_{l=-m}^{m} \E \brac{X_l^2} } \leq \sigma^2.
\end{align}
\end{theorem}

We apply the matrix Bernstein inequality to the finite sequence of independent adjoint zero-mean random matrices of the form 
\begin{align}
X_l :=  \brac{ \delta_{\Omega}\brac{l}-\frac{s}{n}} \vct{b}\brac{l} \vct{b}\brac{l}^{\ast}, \quad -m \leq l \leq m.
\end{align}
These random matrices satisfy
\begin{align}
H - \frac{s}{n} \bar{H} & = \sum_{l=-m}^{m} X_l.
\end{align}
By Lemma~\ref{lemma:bound_b}
\begin{align}
\norm{ X_l } & \leq \sup_{ -m \leq l \leq m}  \normTwo{\vct{b} \brac{ l }}^2 \\
& \leq B := 10 \, k .
\end{align}
In addition,
\begin{align}
\label{eqn:sigmaest}
\sigma^2 &:= \norm{ \textstyle \sum_{l=-m}^{m} \E \brac{X_l^2} } \\
& =  \norm{ \textstyle \sum_{l=-m}^{m} \E \brac{\brac{\bar{\delta} \brac{l}-\frac{s}{n}}^2} \normTwo{\vct{b} \brac{ l }}^2 \vct{b}\brac{l} \vct{b}\brac{l}^{\ast} } \\
& \leq 10 \, k \, \frac{s}{n}\norm{\bar{H}}  \\
& \leq 10 \, C_{B}^2 \, n \, k \brac{\log \frac{n}{\epsilon}}^{-1}
\end{align}
by Lemma~\ref{lemma:bound_b}, \eqref{eq:barH_bound_Cb} and the fact that the variance of a Bernouilli random variable of parameter $p$ equals $p \brac{1-p}$. Setting $t := \frac{C_B^2 \, n}{2} \brac{ \log \frac{n}{\epsilon}}^{-1}$ in Theorem~\ref{theorem:matrix_bernstein}, so that $\sigma^2 = 20 \, k \, t$, yields
\begin{align}
\Pr \keys{ \norm{ H - \frac{s}{n} \bar{H}} \geq t} & \leq  2 k \, \exp \brac{\frac{- t^2/2}{\sigma^2 + Bt / 3}}\\
&  \leq 2 k \, \exp \brac{\frac{- 3t}{  140 \, k }}.
\end{align}
The probability is smaller or equal to $\epsilon/5$ as long as 
\begin{align}
k \leq \frac{3 \, C_B^2 \, n}{280} \brac{ \log \frac{10 \, k}{\epsilon} \log \frac{n}{\epsilon}}^{-1}
\end{align}
which holds by~\eqref{eq:cond_k} if we set $C_k$ small enough.

\section{Proof of Lemma \ref{lemma:v_vbar}}
\label{proof:v_vbar}
The proof uses the following concentration bound that controls the deviation of a sum of independent vectors. 
\begin{theorem}[Vector Bernstein inequality~{\cite[Theorem 2.6]{candes2011probabilistic}}, {\cite[Theorem 12]{Gross:2009id}}]\label{thm:vecbern}
Let $\ml{U} \subset \R^d$ be a finite sequence of independent zero-mean random vectors with $\normTwo{ \vct{u} } \leq B$ almost surely and $\sum_{ \vct{u} \in \ml{U}} \E \normTwo{ \vct{u} }^2 \leq \sigma^2$ for all $\vct{u} \in \ml{U}$, where $B$ and $\sigma^2$ are positive constants. For all $t \geq 0$
\begin{align}
\Pr\brac{\normTwo{ \textstyle \sum_{ \vct{u} \in \ml{U}} \vct{u} } \geq t } \leq \exp\brac{-\frac{t^2}{8\sigma^2} + \frac{1}{4}} \quad \text{\ for \ } 0 \leq t \leq \frac{\sigma^2}{B} .
\end{align}
\end{theorem}
By the definitions of $\bar{K}$, $K$ and $\vct{b}$ in~\eqref{eq:kbar}, \eqref{eq:k} and~\eqref{eq:b},
\begin{align}
\vct{\bar{v}_{\ell}} \brac{f} & = \sum_{l = - m}^{ m} \brac{  i 2 \pi \kappa l}^\ell \vct{c}_{ l } \, e^{i2\pi l f}  \vct{b}\brac{l}, \\
\vct{v_{\ell}} \brac{f} & = \sum_{l = - m}^{ m} \delta_{\Omega^c}\brac{l} \brac{  i 2 \pi \kappa l}^\ell \vct{c}_{ l } \, e^{i2\pi l f}  \vct{b}\brac{l},
\end{align}
where by assumption $ \delta_{\Omega^c} \brac{-m}$, \ldots,  $ \delta_{\Omega^c} \brac{m}$ are iid Bernoulli random variables with parameter $p := \frac{n-s}{n}$. This implies that the finite collection of zero-mean random vectors of the form
\begin{align}
\vct{u} \brac{\ell,l} & :=  \brac{ \delta_{\Omega^c}\brac{l} - p  } \brac{  i 2 \pi \kappa l}^\ell \vct{c}_{ l } \, e^{i2\pi l f}  \vct{b}\brac{l},
\end{align}
satisfy
\begin{align}
\vct{v_{\ell}} \brac{f} - p \, \vct{\bar{v}_{\ell}} \brac{f} & =  \sum_{l = - m}^{ m}  \vct{u} \brac{l}.
\end{align}
We have
\begin{align}
\normTwo{\vct{u} \brac{\ell,l}} & \leq \pi^3 \normInf{ \vct{c} } \sup_{-m \leq l \leq m} \normTwo{\vct{b} \brac{ l }}  \quad \text{by Lemma~\eqref{lemma:bound_kappa} and $\ell \leq 3$} \\
& \leq B:= \frac{128 \sqrt{k}}{m} \quad \text{by Lemmas~\ref{lemma:c_amp} and~\ref{lemma:bound_b}, } 
\end{align}
as well as
\begin{align}
\sum_{l=-m}^{m} \E \normTwo{ \vct{u} \brac{\ell,l} }^2 & = \sum_{l=-m}^{m} \E \brac{\brac{ \delta_{\Omega^c}\brac{l} - p  }^2} \brac{   2 \pi \kappa l}^{2 \ell} \abs{\vct{c}_{ l }}^2  \normTwo{  \vct{b}\brac{l} }^2 \\
& \leq \pi^6  n \, \E \brac{\brac{ \delta_{\Omega^c}\brac{1} - p  }^2}  \normInf{ \vct{c} }^2 \sup_{-m \leq l \leq m} \normTwo{\vct{b} \brac{ l }}^2  \quad \text{by Lemma~\eqref{lemma:bound_kappa}} \\
& \leq \sigma^2 := \frac{3.25 \, 10^4 \, k}{m}, 
\end{align}
where the last inequality follow from Lemmas~\ref{lemma:c_amp} and~\ref{lemma:bound_b} and $\E \brac{\brac{ p - \delta_{\Omega^c}\brac{l}}^2} = p\brac{1-p}$. By the vector Bernstein inequality for $0 \leq t \leq \sigma^2/B$ and the union bound we have
\begin{align}
\Pr\brac{ \sup_{f \in \ml{G}} \left\|v_\ell(f) - p \bar{v}_\ell(f)\right\|_2 \geq t, \quad \ell \in \keys{0,1,2,3} } \leq 4 \abs{\ml{G}}\exp\brac{-\frac{t^2}{8\sigma^2} + \frac{1}{4}}. \label{eq:bound_Bernstein_vvbar}
\end{align}
To make the right-hand side smaller than $\epsilon /5$, we fix $t$ to equal
\begin{align}
t & :=  \sigma \sqrt{ 8 \brac{ \frac{1}{4} + \log \frac{20 \, \abs{\ml{G}}}{\epsilon}} }.
\end{align}
This choice of $t$ is valid because 
\begin{align}
\frac{t}{\sigma} & = \sqrt{ 8 \brac{ \frac{1}{4} + \log \frac{20 \, \abs{\ml{G}}}{\epsilon}} } \\ 
& \leq \sqrt{74 + 16 \log n + 8 \log \frac{1}{\epsilon} } \\
& \leq 0.315 \sqrt{n} + \sqrt{8 \log \frac{1}{\epsilon}} \label{eq:t_bound_1}\\
& \leq 0.32\sqrt{n}. \label{eq:t_bound_2}
, 
\end{align}
Inequality~\eqref{eq:t_bound_1} follows from the fact that $\sqrt{74 + 16 \log n} \leq 0.315 \sqrt{n}$ for $n \geq 2\, 10^3$. Inequality~\eqref{eq:t_bound_2} holds by~\eqref{eq:cond_k} and~\eqref{eq:cond_s} as long as we set $C_k$ and $C_s$ small enough and either $k \geq 1$ or $s \geq 1$. This establishes that \cfg{$t/\sigma$} is smaller than $0.32 \sqrt{n} \leq \sigma/B$.

We conclude that the desired bound holds as long as 
\begin{align}
C_{\vct{v}} \brac{\log \frac{n}{\epsilon} }^{-\frac{1}{2} } & \geq t 
 \geq \sqrt{ \frac{2 \, 10^3 \, k}{n} \brac{ \frac{1}{4} + \log \frac{8 \, 10^3 \, n^2}{\epsilon}} }, 
\end{align}
which is the case by~\eqref{eq:cond_k} if we set $C_k$ small enough.

\section{Proof of Lemma~\ref{lemma:D_bounds}}
\label{proof:D_bounds}
The proof is based on the proof of Lemma~4.4 in~\cite{tang2012offgrid}. The following lemma establishes that $\bar{D}$ is invertible and close to the identity.
\begin{lemma}[Proof in Section~\ref{proof:Dbar_bounds}]
\label{lemma:Dbar_bounds}
Under the assumptions of Theorem~\ref{theorem:main}
\begin{align}
\norm{ I - \bar{D} } & \leq 0.468,\\
\norm{\bar{D}} & \leq 1.468 ,\\
\norm{\bar{D}^{-1}} & \leq 1.88.
\end{align}
\end{lemma}
By the definition of $\bar{K}$ and $K$ in~\eqref{eq:kbar} and~\eqref{eq:k} respectively we can write $D$ and $\bar{D}$ as sums of self-adjoint matrices,
\begin{align}
\bar{D} &= \sum_{l=-m}^{m} \vct{c}_l \, \vct{b} \brac{ l }\vct{b} \brac{ l }^{\ast}, \\
D &= \sum_{l=-m}^{m} \delta_{\Omega^c} \brac{l} \, \vct{c}_l \, \vct{b} \brac{ l }\vct{b} \brac{ l }^{\ast},
\end{align}
where by assumption $ \delta_{\Omega^c} \brac{-m}$, \ldots,  $ \delta_{\Omega^c} \brac{m}$ are iid Bernoulli random variables with parameter $p := \frac{n-s}{n}$. In the following lemma we leverage the matrix Bernstein inequality to establish that $D$ concentrates around $p \, \bar{D}$.
\begin{lemma}[Proof in Section~\ref{proof:D_Dbar}]
\label{lemma:D_Dbar}
Under the assumptions of Theorem~\ref{theorem:main}
\begin{align}
\label{eq:D_barD}
\norm{D - p\bar{D} } \geq \frac{p}{4} \min \keys{ 1, \frac{C_{D}}{4} \brac{\log \frac{n}{\epsilon} }^{-\frac{1}{2} }}.
\end{align}
with probability at most $\epsilon /5$.
\end{lemma}

Applying the triangle inequality together with Lemma~\ref{lemma:Dbar_bounds} allows to lower bound the smallest singular value of $D$ under the assumption that~\eqref{eq:D_barD} holds
\begin{align}
\frac{ \sigma_{\min}\brac{D} }{p} & \geq \sigma_{\min}\brac{I} - \norm{I-\bar{D}} - \frac{1}{p}\norm{D - p\bar{D}} \\
& \geq 0.282.
\end{align}
This proves that $D$ is invertible. To complete the proof we borrow two inequalities from~\cite{tang2012offgrid}.

\begin{lemma}[{\hspace{1sp}\cite[Appendix E]{tang2012offgrid}}]
\label{lemma:AB}
For any matrices $A$ and $B$ such that $B$ is invertible and 
\begin{align}
\norm{A - B} \norm{B^{-1}} \leq \frac{1}{2}
\end{align}
we have
\begin{align}
\norm{A^{-1}} & \leq 2 \norm{B^{-1}}, \\
\norm{A^{-1} - B^{-1}} & \leq 2 \norm{B^{-1}}^2 \norm{A - B}.
\end{align}
\end{lemma}

We set $A:= D$ and $B:=p\bar{D}$. By Lemmas~\ref{lemma:Dbar_bounds} and Lemma~\ref{lemma:D_Dbar}, 
\begin{align}
\norm{D - p\bar{D}} \norm{\brac{p\bar{D}}^{-1}} & \leq \frac{1}{2}
\end{align}
with probability at least $1-\epsilon/5$. Lemmas~\ref{lemma:Dbar_bounds}, \ref{lemma:D_Dbar} and~\ref{lemma:AB} then imply
\begin{align}
\norm{D^{-1}} & \leq 2 \norm{\brac{p\bar{D}}^{-1}}\\
&  \leq \frac{4}{p} ,\\
\norm{D^{-1} - \brac{p\bar{D}}^{-1}} & \leq 2 \norm{\brac{p\bar{D}}^{-1} }^2 \norm{D - p\bar{D}}\\
& \leq \frac{C_{D}}{2p} \brac{\log \frac{n}{\epsilon} }^{-\frac{1}{2}} ,
\end{align}
with the same probability. Finally, if $s \leq n/2$, which is the case by~\eqref{eq:cond_s}, we have $1/p \leq 2$ and the proof is complete.

\subsection{Proof of Lemma~\ref{lemma:Dbar_bounds}}
\label{proof:Dbar_bounds}
The following bounds on the submatrices of $\bar{D}$ are obtained by combining Lemma~\ref{lemma:bound_kappa} with some results borrowed from~\cite{superres_new}. 
\begin{lemma}[{\hspace{1sp}\cite[Section 4.2]{superres_new}}]
Under the assumptions of Theorem~\ref{theorem:main}
\begin{align}
\normInf{ I - \bar{D}_0} & \leq 1.855 \, 10^{-2}, \\
\normInf{\bar{D}_1} &  \leq  5.148 \, 10^{-2} ,\\
\normInf{ I -\bar{D}_2} & \leq 0.416.
\end{align}
\end{lemma}
Following a similar argument as in Appendix~C of \cite{tang2012offgrid} yields the desired result:
\begin{align}
\norm{ I - \bar{D}} & \leq \normInf{ I - \bar{D}} \\
& \leq \max \keys{ \normInf{ I - \bar{D}_0} + \normInf{ \bar{D}_1} , \normInf{ I - \bar{D}_2} + \normInf{ \bar{D}_1} } \\
& \leq 0.468,\\
\norm{ \bar{D}} & \leq 1 + \norm{ I - \bar{D}} \leq 1.468, \\
\norm{ \bar{D}^{-1} } & \leq \frac{1}{1 - \normInf{ I - \bar{D}}} \leq 1.88.
\end{align}
\subsection{Proof of Lemma~\ref{lemma:D_Dbar}}
\label{proof:D_Dbar}
We define
\begin{align}
X_l :=  \brac{ p - \delta_{\Omega^c}\brac{l}} \, \vct{c}_l \, \vct{b} \brac{ l }\vct{b} \brac{ l }^T,
\end{align}
which has zero mean since
\begin{align}
\E \brac{ X_l } & =   \brac{ p - \E \brac{ \delta_{\Omega^c}\brac{l}} } \, \vct{c}_l \, \vct{b} \brac{ l }\vct{b} \brac{ l }^T \\
& = 0.
\end{align}
By the proofs of Lemmas~\ref{lemma:c_amp} and~\ref{lemma:bound_b}, for any $-m \leq l \leq m$,
\begin{align}
\norm{X_l} & \leq  \max_{-m \leq l \leq m} \norm{ \vct{c}_l \, \vct{b} \brac{ l } \vct{b} \brac{ l }^T } \\
& \leq \normInf{\vct{c}} \max_{-m \leq l \leq m} \normTwo{\vct{b} \brac{ l }}^2 \\ 
& \leq B:= \frac{12.6 \, k}{m} . 
\end{align}
Also, $\E \brac{\brac{ p - \delta_{\Omega^c}\brac{l}}^2} = p\brac{1-p}$, which implies
\begin{align}
 \E \brac{X_l^2} = p \brac{1-p} \vct{c}_l^2 \normTwo{\vct{b} \brac{ l }}^2 \vct{b} \brac{ l }\vct{b} \brac{ l }^T.
\end{align}
Since $\vct{c}_l \geq 0$ for all $l$ ($\vct{c}$ is the convolution of three positive rectangular pulses), 
\begin{align}
\sum_{l=-m}^{m} \vct{c}_l^2  \normTwo{\vct{b} \brac{ l }}^2 \vct{b} \brac{ l }\vct{b} \brac{ l }^T & \preceq  \normInf{\vct{c}} \max_{-m \leq l \leq m} \normTwo{\vct{b} \brac{ l }}^2 \sum_{l=-m}^{m} \vct{c}_l  \, \vct{b} \brac{ l }\vct{b} \brac{ l }^T \\
& \preceq  \frac{12.6 \, k}{m} \bar{D} \quad \text{by Lemma~\ref{lemma:c_amp} and~\ref{lemma:bound_b}},
\end{align}
so that
\begin{align}
\sum_{l=-m}^{m} \E \brac{X_l^2} & \leq p \norm{ \sum_{l=-m}^{m} \vct{c}_l^2  \normTwo{\vct{b} \brac{ l }}^2 \vct{b} \brac{ l }\vct{b} \brac{ l }^T } \\
& \leq  \frac{12.6 \, p \, k \norm{\bar{D}}}{m} \\
& \leq  \sigma^2 := \frac{18.5 \, p \, k }{m} \quad \text{by Lemma~\ref{lemma:Dbar_bounds}.}
\end{align}

Setting $t = \frac{p}{4}C_{\min} \brac{\log \frac{n}{\epsilon} }^{-\frac{1}{2} }$ where $C_{\min}:=\min \keys{1,C_{D}/4}$, the matrix Bernstein inequality from Theorem~\ref{theorem:matrix_bernstein} implies that
\begin{align}
\Pr \keys{\norm{ D^{-1} - p \bar{D}^{-1}} > t} & \leq 2 \, k \, \exp \brac{- \frac{C_{\min}^2 p \, m}{32 \, k} \brac{ 18.5\log \frac{n}{\epsilon}  + 1.05 \, C_{\min} \sqrt{\log \frac{n}{\epsilon}}   }^{-1} } \notag \\
 & \leq 2 \, k \, \exp \brac{- \frac{C_{D}' \brac{n-s}}{ k \log \frac{n}{\epsilon} }}  
\end{align}
for a small enough constant $C_{D}'$. This probability is smaller than $\epsilon/5$ as long as 
\begin{align}
k & \leq \frac{C_{D}' \, n}{2} \brac{\log \frac{10k}{\epsilon} \log \frac{n}{\epsilon}}^{-1}, \\
s & \leq \frac{n}{2} ,
\end{align} 
which holds by~\eqref{eq:cond_k} and~\eqref{eq:cond_s} if we set $C_k$ and $C_s$ small enough.

\section{Proof of Proposition~\ref{proposition:Qbound1}}
\label{proof:Qbound1}

We begin by expressing $Q^{\brac{\ell}}$ and $\bar{Q}^{\brac{\ell}}$ in terms of $\signx$ and $\signz$,
\begin{align}
\kappa^{\ell} \, \bar{ Q }^{\brac{\ell}} \brac{ f } & := \kappa^{\ell}  \sum_{j =1}^{k} \vct{\bar{\alpha}}_j \, \bar{ K }^{\brac{\ell}} \brac{ f - f_j } + \kappa^{\ell+1} \sum_{j =1}^{k} \vct{\bar{\beta}}_j \, \bar{ K}^{\brac{\ell+1}}  \brac{ f - f_j }\\
 & =  \vct{\bar{v}_{\ell}}\brac{f}^T \bar{D}^{-1} \MAT{ \signx \\ 0}, \\
\kappa^{\ell} \, Q ^{\brac{\ell}} \brac{ f } & :=  \kappa^{\ell} \sum_{j =1}^{k} \vct{\alpha}_j \, K^{\brac{\ell}}  \brac{ f - f_j } + \kappa^{\ell+1} \sum_{j =1}^{k} \vct{\beta}_j \, K^{\brac{\ell + 1}}  \brac{ f - f_j } + \kappa^{\ell} \, R^{\brac{\ell}} \brac{ f } \\
  & = \vct{v_{\ell}}\brac{f}^TD^{-1} \brac{\MAT{  \signx  \\ 0 } - \frac{1}{\sqrt{n}} B_{\Omega} \, \signz }  + \kappa^{\ell} \, R^{\brac{\ell}} \brac{ f }.\label{eq:Q_vl_Dinv_ur}
\end{align}
The difference between $Q^{\brac{\ell}} $ and $\bar{ Q }^{\brac{\ell}}$ can be decomposed into several terms,
\begin{align}
\kappa^{\ell} \, Q^{\brac{\ell}}\brac{f} & = \kappa^{\ell} \, \bar{ Q }^{\brac{\ell}} \brac{ f } + \kappa^{\ell} \, R^{\brac{\ell}} \brac{ f } + I^{\brac{\ell}}_1\brac{f} + I^{\brac{\ell}}_2\brac{f} + I^{\brac{\ell}}_3\brac{f} , \label{eq:Q_Qbar} \\
I_1^{\brac{\ell}} \brac{f} & :=  -\frac{1}{\sqrt{n}} \vct{v_{\ell}} \brac{f}^T D^{-1}  B_{\Omega} \, \signz  ,\\
I_2^{\brac{\ell}} \brac{f} & :=  \brac{ \vct{v_{\ell}} \brac{f} - \frac{n-s}{n} \vct{\bar{v}_{\ell}} \brac{f} }^T D^{-1} \MAT{\signx \\ \vct{0} } , \\ 
I_3^{\brac{\ell}} \brac{f} & := \frac{n-s}{n} \vct{\bar{v}_{\ell}} \brac{f}^T \brac{ D^{-1} - \frac{n}{n-s} \bar{D}^{-1} }  \MAT{\signx \\ \vct{0} } .
\end{align}
The following lemma provides bounds on these terms that hold with high probability in every point of a grid $\ml{G}$ that discretizes the unit interval. 
\begin{lemma}[Proof in Section~\ref{proof:grid_bounds}]
\label{lemma:grid_bounds}
Conditioned on $\ml{E}_{B}^{c} \cap \ml{E}_{D}^{c} \cap \ml{E}_{v}^{c}$, the events
\begin{align}
  \ml{E}_{R} & := \keys{ \sup_{f \in \ml{G}} \abs{ \kappa^{\ell} R^{\brac{\ell}}\brac{f} } \geq \frac{10^{-2}}{8} , \ell = 0, 1, 2, 3 }
  \end{align}
and 
\begin{align}
  \ml{E}_{i} & := \keys{ \sup_{f \in \ml{G}} \abs{ I_{i}^{\brac{\ell}} \brac{f} } \geq \frac{10^{-2}}{8} , \ell = 0, 1, 2, 3 } \quad i = 1,2,3 
\end{align}
where $\ml{G} \subseteq \sqbr{0,1}$ is an equispaced grid with cardinality $\abs{ \ml{G} } = 400 n^2$ occur each with probability at most $\epsilon / 20$ under the assumptions of Theorem~\ref{theorem:main}.
\end{lemma}

By the triangle inequality, Lemma~\ref{lemma:grid_bounds} implies
\begin{align}
\label{eq:Q_Qbar_grid}
\sup_{f \in \ml{G}} \abs{ \kappa^{\ell} \, Q^{\brac{\ell}}\brac{f} - \kappa^{\ell} \, \bar{ Q }^{\brac{\ell}} \brac{ f } } \leq \frac{10^{-2}}{2} 
\end{align}
with probability at least $1-\epsilon /5 $ conditioned on $\ml{E}_{B}^{c} \cap \ml{E}_{D}^{c} \cap \ml{E}_{v}^{c}$. 

We have controlled the deviation between $Q^{\brac{\ell}}$ and $\bar{ Q }^{\brac{\ell}}$ on a fine grid. The following result extends the bound to the whole unit interval. 
\begin{lemma}[Proof in Section~\ref{proof:Ql_Qbarl}]
\label{lemma:Ql_Qbarl}
Under the assumptions of Theorem~\ref{theorem:main} 
\begin{align}
\abs{\kappa^{\ell} Q^{\brac{\ell}}\brac{f} - \kappa^{\ell} \bar{Q}^{\brac{\ell}}\brac{f} } \leq  10^{-2} \quad \text{ for $\ell \in \keys{0,1,2}$}.
\end{align}
\end{lemma}
This bound suffices to establish the desired result for values of $f$ that lie away from $T$. Let us define
\begin{align}
\ml{S}_{\mathrm{near}} & := \set{f}{\abs{f-f_j} \leq 0.09 \; \text{for some } f_j \in T}, \\
\ml{S}_{\mathrm{far}} & := \sqbr{0,1}/ \ml{S}_{\mathrm{near}}.
\end{align}
Section 4 of~\cite{superres_new} provides a bound on $\bar{Q}$ which holds over all of $\ml{S}_{\mathrm{far}}$ under the minimum-separation condition~\eqref{condition:minimum_separation} (see Figure~12 in~\cite{superres_new} as well as the code that supplements~\cite{superres_new}). 
\begin{proposition}[Bound on $\bar{Q}${~\cite[Section 4]{superres_new}}]
\label{prop:Qbar_bound}
Under the assumptions of Theorem~\ref{theorem:main}
\begin{alignat}{2}
\abs{ \bar{Q}\brac{f} } & < 0.99 \quad &&  f \in \ml{S}_{\mathrm{far}}. 
\end{alignat}
\end{proposition}
Combining Lemma~\ref{lemma:Ql_Qbarl} and Proposition~\ref{prop:Qbar_bound} 
\begin{align}
\abs{Q\brac{f}} & \leq \abs{\bar{Q}\brac{f}} + 10^{-2}\\
& <1 \quad \text{for all } f \in \ml{S}_{\mathrm{far}}.
\end{align}
To bound $Q$ in $\ml{S}_{\mathrm{near}}$ we recall that by Corollary~\ref{cor:invertible} in $\ml{E}_{D}^c$ $\abs{Q\brac{f_j}}^2=1$ and  
\begin{align}
\der{f}{ \abs{Q\brac{f_j}}^2 } & = 2 \, Q_{R}^{\brac{1}} \brac{f_j} Q_{R}\brac{f_j} + 2 \, Q_{I}^{\brac{1}}\brac{f_j}Q_{I}\brac{f_j} \\
& = 0 
\end{align}
for every $f_j$ in $T$. Let $\tilde{f}$ be the element in $T$ that is closest to an arbitrary $f$ belonging to $\ml{S}_{\mathrm{near}}$. The second-order bound  
\begin{align}
\abs{Q\brac{f}}^2 & \leq 1 + \brac{f-\tilde{f}}^2 \sup_{ f \in \ml{S}_{\mathrm{near}}} \derTwo{f}{ \abs{Q\brac{f}}^2 } 
\end{align}
implies that we only need to show that $\abs{Q}^2$ is concave in $\ml{S}_{\mathrm{near}}$ to complete the proof. First, we bound the derivatives of $\bar{Q}$ and $Q$ using Bernstein's polynomial inequality.

\begin{lemma}
\label{lemma:unifbounds}
Under the assumptions of Theorem~\ref{theorem:main}, for any $\ell =0,1,2, \ldots$
\begin{align}
\sup_{f \in \sqbr{0,1}} \abs{\kappa^{ \ell } \bar{Q}^{\brac{\ell}} \brac{f} } \leq 1, \label{eq:unifbound_Qbar}\\
\sup_{f \in \sqbr{0,1}} \abs{\kappa^{ \ell } Q^{\brac{\ell}} \brac{f} } \leq 1.01.\label{eq:unifbound_Q}
\end{align}
\end{lemma}
\begin{proof}
$\bar{Q}$ is a trigonometric polynomial of degree $m$ and its magnitude is bounded by one (see Proposition 2.3 in~\cite{superres_new}). Combining Theorem~\ref{theorem:bernstein_pol_ineq} and Lemma~\ref{lemma:bound_kappa} yields~\eqref{eq:unifbound_Qbar}. The triangle inequality, Lemma~\ref{lemma:Ql_Qbarl} and~\eqref{eq:unifbound_Qbar} imply~\eqref{eq:unifbound_Q}.
\end{proof}
Section 4 of~\cite{superres_new} also provides a bound on the second derivative of $\abs{\bar{Q}}^2$ which holds over all of $\ml{S}_{\mathrm{near}}$ under the minimum-separation condition~\eqref{condition:minimum_separation} (again, see Figure~12 in~\cite{superres_new} as well as the code that supplements~\cite{superres_new}).
\begin{proposition}[Bound on the second derivative of $\abs{\bar{Q}}${~\cite[Section 4]{superres_new}}]
\label{prop:Qbar_derTwo_bounds}
Under the assumptions of Theorem~\ref{theorem:main}
\begin{alignat}{2}
\derTwo{f}{ \abs{\bar{Q}\brac{f}}^2 } &\leq -0.8 \, m^2 \quad &&  f \in \ml{S}_{\mathrm{near}}.
\end{alignat}
\end{proposition}
Combining Proposition~\ref{prop:Qbar_derTwo_bounds}, Lemma~\ref{lemma:unifbounds} and the triangle inequality, as well as the lower bound on $\kappa$ from Lemma~\ref{lemma:bound_kappa}, allows us to conclude that the second derivative of $\abs{\bar{Q}}^2$ is negative in $\ml{S}_{\mathrm{near}}$. Indeed, for any $f \in \ml{S}_{\mathrm{near}}$
\begin{align}
 \frac{\kappa^2}{2} \derTwo{f}{ \abs{Q\brac{f}}^2 } & = \kappa^2Q_{R}^{\brac{2}} \brac{f} Q_{R}\brac{f} + \kappa^2 Q_{I}^{\brac{2}}\brac{f}Q_{I}\brac{f} + \abs{\kappa \, Q^{\brac{1}} \brac{f} }^2 \\
& \leq  \frac{\kappa^2}{2} \derTwo{f}{ \abs{\bar{Q}\brac{f}}^2 } + 2 \abs{  \kappa^2 Q^{\brac{2}} \brac{f} -  \kappa^2\bar{Q}^{\brac{2}}\brac{f}  } \sup_{f'} \abs{ Q \brac{f'} }\notag \\
& \quad +2 \abs{Q\brac{f} - \bar{Q}\brac{f} }  \sup_{f'} \abs{  \kappa^2 \bar{Q}^{\brac{2}}  \brac{f'} } \notag \\
& \quad   + 2 \abs{\kappa \, Q^{\brac{1}}\brac{f} - \kappa \, \bar{Q}^{\brac{1}}\brac{f} }  \brac{ \sup_{f'}\abs{ \kappa \, Q^{\brac{1}}  \brac{f'}} +  \sup_{f'}\abs{ \kappa \, \bar{Q}^{\brac{1}}  \brac{f'}} } \\
& \leq -0.087 + 2 \cdot 10^{-2} (4+2 \cdot 10^{-2}) \\
& < 0.
\end{align}

\subsection{Proof of Lemma~\ref{lemma:grid_bounds}}
\label{proof:grid_bounds}
Following an argument used in~\cite{tang2012offgrid} (see also~\cite{Candes:2007es}), we use Hoeffding's inequality to bound the different terms.
\begin{theorem}[Hoeffding's inequality]
  \label{theorem:hoeffding} Let the components of $ \vct{\tilde{u}} $ be sampled i.i.d. from a symmetric distribution on the complex unit circle. For any $t >0$ and any vector $\vct{u} $
\begin{align}
   \Pr \brac{ \abs{ \left<\vct{\tilde{u}}, \vct{u}\right> } \geq \tilde{\epsilon} } & \leq 4 \exp \brac{- \frac{\tilde{\epsilon}^2}{4 \normTwo{ \vct{u} }^2}}.
\end{align}
\end{theorem}
\begin{corollary}
\label{cor:hoeffding_grid}
Let the components of $ \vct{\tilde{u}} $ be sampled i.i.d. from a symmetric distribution on the complex unit circle. For any finite collection of vectors $\ml{U}$ with cardinality $4 \abs{G} = 1600 n^2$ the event
\begin{align}
\ml{E} := \keys{ \abs{ \PROD{\vct{\tilde{u}}}{ \vct{u}} } >   \frac{10^{-2}}{8} \quad \text{for all } \; \vct{u} \in \ml{U} }  
\end{align}
has probability at most $ \epsilon / 20$ as long as
\begin{align}
\normTwo{\vct{u}}^2 \leq C_{\ml{U}}^2 \brac{\log \frac{n}{\epsilon}}^{-1} \quad \text{for all } \; \vct{u} \in \ml{U},
\end{align}
where $ C_{\ml{U}} : =1/5000$.
\end{corollary}
\begin{proof}
The result follows directly from the proposition and the union bound.
\end{proof}

\subsubsection*{Bound on $\Pr \brac{\ml{E}_{R} | \ml{E}_{B}^{c} \cap \ml{E}_{D}^{c} \cap \ml{E}_{v}^{c}}$}
We consider the family of vectors
\begin{align}
\vct{u}\brac{\ell, f} :=  \frac{ \kappa^{\ell} }{\sqrt{n}} \MAT{ \brac{i2\pi l_1}^\ell e^{i 2 \pi l_1 f} &  \brac{i2\pi l_2}^\ell e^{i 2 \pi l_2 f} & \cdots &  \brac{i2\pi l_s}^\ell e^{i 2 \pi l_s f} }^T
\end{align}
where $\ell \in \keys{0,1,2,3}$ and $f$ belongs to $\ml{G}$, so that $\abs{\ml{U}} = 4 \abs{\ml{G}}$. We have
\begin{align}
\normTwo{ \vct{u} \brac{\ell, f}}^2 & \leq \frac{ \kappa^{2\ell} \brac{2\pi m}^{2\ell} s}{n} \\
& \leq  \frac{ \pi^6 s}{n} \quad \text{by Lemma~\ref{lemma:bound_kappa}} \\
& \leq C_{\ml{U}}^2 \brac{\log \frac{n}{\epsilon}}^{-1} \quad \text{by~\eqref{eq:cond_s} if we set $C_s$ small enough}.
\end{align}
The desired result follows by Corollary~\ref{cor:hoeffding_grid} because
\begin{align}
\kappa^{\ell} R^{\brac{\ell}}\brac{f} = \PROD{ \signz }{ \vct{u}\brac{\ell, f} } .
\end{align}

\subsubsection*{Bound on $\Pr \brac{\ml{E}_{1} | \ml{E}_{B}^{c} \cap \ml{E}_{D}^{c} \cap \ml{E}_{v}^{c}}$}
We have 
\begin{align}
I_1^{\brac{\ell}} \brac{f} & = \PROD{ \vct{u}\brac{ \ell , f } }{ \signz}, \qquad 
\vct{u}\brac{\ell, f} := -\frac{1}{\sqrt{n}} \; B_{\Omega}^{\ast}\, D^{-1} \vct{v_{\ell}} \brac{f},
\end{align}
where $\ell \in \keys{0,1,2,3}$ and $f$ belongs to $\ml{G}$, so that $\abs{\ml{U}} = 4 \abs{\ml{G}}$. 

To bound $\normTwo{ \vct{u}\brac{\ell, f} }$ we leverage a bound on the $\ell_2$ norm of $\vct{ v_{\ell}}$ which follows from Lemma~\ref{lemma:v_vbar} and the following bound on the $\ell_2$ norm of $\vct{\bar{v}_{\ell}}$ .
\begin{lemma}[Proof in Section~\ref{proof:vbar_bound}]
\label{lemma:vbar_bound}
Under the assumptions of Theorem~\ref{theorem:main}, there is a fixed numerical constant $C_{\vct{\bar{v}}}$ such that for any $f$
\begin{align}
\normTwo{\vct{\bar{v}_{\ell}} \brac{f}} & \leq C_{\vct{\bar{v}}}.
\end{align}
\end{lemma}
\begin{corollary}
\label{cor:v_bound}
In $\ml{E}_{v}^c$ for any $f \in \ml{G}$
\begin{align}
\normTwo{\vct{ v_{\ell}} \brac{f}} & \leq C_{\vct{\bar{v}}} + C_{\vct{v}}.
\end{align}
\end{corollary}
\begin{proof}
The result follows from the lemma, the triangle inequality and Lemma~\ref{lemma:v_vbar}.
\end{proof}

Combining Lemma~\ref{lemma:D_bounds} and Corollary~\ref{cor:v_bound} yields
\begin{align}
\normTwo{ \vct{u}\brac{\ell, f} } & \leq \frac{1}{\sqrt{n}} \norm{B_{\Omega}} \norm{D^{-1}} \normTwo{\vct{v_{\ell}}\brac{f}} \\
& \leq \frac{8 \brac{C_{\vct{\bar{v}}} + C_{\vct{v}} } \norm{B_{\Omega}}}{\sqrt{n}}
\end{align}
in $\ml{E}_{D}^{c} \cap \ml{E}_{v}^{c}$. Corollary~\ref{cor:hoeffding_grid} implies the desired result if 
\begin{align}
\label{eq:C_B_def}
\norm{B_{\Omega}} & \leq  C_{B} \brac{\log \frac{n}{\epsilon} }^{-\frac{1}{2} } \sqrt{ n }, \quad C_{B} := \frac{C_{\ml{U}}}{8\brac{C_{\vct{\bar{v}}} + C_{\vct{v}} }},
\end{align}
which is the case in $\ml{E}_{B}^{c}$ by Lemma~\ref{lemma:boundB}.

\subsubsection*{Bound on $\Pr \brac{\ml{E}_{2} | \ml{E}_{B}^{c} \cap \ml{E}_{D}^{c} \cap \ml{E}_{v}^{c}}$}

We have 
\begin{align}
I_2^{\brac{\ell}} \brac{f} &  = \PROD{ \vct{u}\brac{ \ell , f } }{ \signx}, \qquad 
\vct{u}\brac{\ell, f} := P D^{-1}\brac{ \vct{v_{\ell}} \brac{f} - \frac{n-s}{n} \vct{\bar{v}_{\ell}} \brac{f} }
\end{align}
where $P \in \R^{k \times 2k}$ is the projection matrix that selects the first $k$ entries in a vector, $\ell \in \keys{0,1,2,3}$ and $f$ belongs to $\ml{G}$, so that $\abs{\ml{U}} = 4 \abs{\ml{G}}$. 

Since $\norm{P}=1$, by Lemma~\ref{lemma:D_bounds} in $\ml{E}_{D}^{c}$
\begin{align}
\normTwo{ \vct{u}\brac{\ell, f} } & \leq \norm{P} \norm{D^{-1}} \normTwo{ \vct{v_{\ell}} \brac{f} - \frac{n-s}{n} \vct{\bar{v}_{\ell}} \brac{f} } \\
& \leq 8 \normTwo{ \vct{v_{\ell}} \brac{f} - \frac{n-s}{n} \vct{\bar{v}_{\ell}} \brac{f} }.
\end{align}
The desired result holds if 
\begin{align}
\label{eq:C_v_def}
\normTwo{ \vct{v_{\ell}} \brac{f} - \frac{n-s}{n} \vct{\bar{v}_{\ell}} \brac{f} } & \leq C_{\vct{v}}\brac{\log \frac{n}{\epsilon} }^{-\frac{1}{2} }, \quad C_{\vct{v}} := \frac{C_{\ml{U}}}{8},
\end{align}
which is the case in $\ml{E}_{v}^{c}$ by Lemma~\ref{lemma:v_vbar}.

\subsubsection*{Bound on $\Pr \brac{\ml{E}_{3} | \ml{E}_{B}^{c} \cap \ml{E}_{D}^{c} \cap \ml{E}_{v}^{c}}$}

We have 
\begin{align}
I_3^{\brac{\ell}} \brac{f} &  = \PROD{ \vct{u}\brac{ \ell , f } }{ \signx}, \qquad 
\vct{u}\brac{\ell, f} := \frac{n-s}{n} P \brac{ D^{-1} - \frac{n}{n-s} \bar{D}^{-1} } \vct{\bar{v}_{\ell}} \brac{f}
\end{align}
where $\ell \in \keys{0,1,2,3}$ and $f$ belongs to $\ml{G}$, so that $\abs{\ml{U}} = 4 \abs{\ml{G}}$. 

Since $\norm{P}=1$, by Lemma~\ref{lemma:v_vbar}
\begin{align}
\normTwo{ \vct{u}\brac{\ell, f} } & \leq \norm{P} \norm{ D^{-1} - \frac{n}{n-s} \bar{D}^{-1}} \normTwo{\vct{\bar{v}_{\ell}} \brac{f} } \\
& \leq C_{\bar{v}} \norm{ D^{-1} - \frac{n}{n-s} \bar{D}^{-1}}.
\end{align}
The desired result holds if 
\begin{align}
\label{eq:C_D_def}
\norm{ D^{-1} - \frac{n}{n-s} \bar{D}^{-1}} & \leq C_{D} \brac{\log \frac{n}{\epsilon} }^{-\frac{1}{2} }, \quad C_{D} := \frac{ C_{\ml{U}} }{ C_{\bar{v}} },
\end{align}
for a fixed numerical constant $C_{D}$, which is the case in $\ml{E}_{D}^{c}$ by Lemma~\ref{lemma:D_bounds}.

\subsection{Proof of Lemma~\ref{lemma:vbar_bound}}
\label{proof:vbar_bound}
We use the $\ell_1$ norm to bound the $\ell_2$ norm of $\vct{\bar{v}_{\ell}}\brac{f}$:
\begin{align}
    \normTwo{ \vct{\bar{v}_{\ell}}\brac{f} } & \leq \normOne{ \vct{\bar{v}_{\ell}}\brac{f} }\\
    &  = \sum_{j = 1}^k \kappa^\ell \left| \bar{K}^{\brac{ \ell }} \brac{ f
    - f_j } \right| + \sum_{j = 1}^k \kappa^{\ell+1} \left| \bar{K}^{\brac{ \ell + 1}} \brac{ f
    - f_j } \right|  .
\end{align}
To bound the sum on the right we leverage some results from~\cite{superres_new}. 
\begin{lemma}
  \begin{align}
     \label{eq:barK_bound}
     \kappa^\ell \left|
    \bar{K}^{\brac{ \ell }} \brac{ f } \right|   \leq \begin{cases}
    C_1 \quad & \forall f \in [-\frac{1}{2}, \frac{1}{2}], \\
    C_2 \, m^{-3} \abs{ f}^{-3}  \quad &  \text{ if } \quad \frac{80}{m} \leq \abs{f} \leq \frac{1}{2},
    \end{cases}
  \end{align}
  for suitably chosen numerical constant $C_1$ and $C_2$. 
\end{lemma}
\begin{proof}
The constant bound on the kernel follows from Corollary 4.5, Lemma 4.6 and Lemma C.2 (see also Figures 14 and 15) in~\cite{superres_new}. The bound for large $f$ follows from Lemma~C.2 in~\cite{superres_new}.
\end{proof}
By the minimum-separation condition~\eqref{condition:minimum_separation}, there are at most 127 elements of $T$ that are at a distance of $80/m$ or less from $f$. We use the first bound in~\eqref{eq:barK_bound} to control the contribution of those elements and the second bound to deal with the remaining terms,
 \begin{align}
 \sum_{j = 1}^k \kappa^\ell \abs{ \bar{K}^{\brac{ \ell }} \brac{ f - f_j } } 
 & \leq \sum_{j: \abs{ f-f_j } < \frac{80}{m}} C_1 + \sum_{j: \frac{80}{m} \leq \abs{ f-f_j } \leq \frac{1}{2}}\frac{C_2}{m^3 \abs{ f-f_j }^3} \\
 & \leq 127\, C_1 + 2\, C_2 \sum_{j = 1}^\infty\frac{1}{m^3 (j \Delta_{\min})^3} \\
 &\leq 127\,C_1 + 2\,C_2 \sum_{j=1}^\infty \frac{1}{j^3} \\
 & = 127\,C_1 + 2\,C_2 \, \zeta \brac{3} ,
  \end{align}
where $\zeta \brac{3}$ is Ap\'ery's constant, which is bounded by 1.21. This completes the proof.

\subsection{Proof of Lemma~\ref{lemma:Ql_Qbarl}}
\label{proof:Ql_Qbarl}
The proof follows a similar argument to the proof of Proposition 4.12 in~\cite{tang2012offgrid}. We begin by bounding the deviations of $Q^{\brac{\ell}}$ and $\bar{Q}^{\brac{\ell}}$ on neighboring points.
\begin{lemma}[Proof in Section~\ref{proof:diff_Q}]
\label{lemma:diff_Q}
Under the assumptions of Theorem~\ref{theorem:main}, for any $f_1$, $f_2$ in the unit interval
\begin{align}
\abs{ \kappa^{\ell} Q^{\brac{\ell}}\brac{f_2} -  \kappa^{\ell} Q^{\brac{\ell}}\brac{f_1} } & \leq  n^2 \abs{f_2 - f_1}, \\ 
\abs{ \kappa^{\ell} \bar{Q}^{\brac{\ell}}\brac{f_2} -  \kappa^{\ell} \bar{Q}^{\brac{\ell}}\brac{f_1} } & \leq  n^2 \abs{f_2 - f_1}.
\end{align} 
\end{lemma}
For any $f$ in the unit interval there exists a grid point $f_{\ml{G}}$ such that the distance between the two points is smaller than the step size $\brac{400 \, n^2}^{-1}$. This allows to establish the desired result by combining~\eqref{eq:Q_Qbar_grid} with Lemma~\ref{lemma:diff_Q} and the triangle inequality,
\begin{align}
\abs{ \kappa^{\ell} Q^{\brac{\ell}}\brac{f} -  \kappa^{\ell} \bar{Q}^{\brac{\ell}}\brac{f} } & \leq \abs{ \kappa^{\ell} Q^{\brac{\ell}}\brac{f} -  \kappa^{\ell} Q^{\brac{\ell}}\brac{f_{\ml{G}} } } + \abs{ \kappa^{\ell} Q^{\brac{\ell}}\brac{f_{\ml{G}}} -  \kappa^{\ell} \bar{Q}^{\brac{\ell}}\brac{f_{\ml{G}}} }\\
& \quad + \abs{ \kappa^{\ell} \bar{Q}^{\brac{\ell}}\brac{f_{\ml{G}}} -  \kappa^{\ell} \bar{Q}^{\brac{\ell}}\brac{f} } \\
& \leq 2 \, n^2 \abs{f - f_{\ml{G}} } + 5 \, 10^{-3}\\
& \leq 10^{-2}.
\end{align}

\subsubsection{Proof of Lemma~\ref{lemma:diff_Q}}
\label{proof:diff_Q}
We first derive a coarse uniform bound on $Q^{\brac{\ell}}$ for $\ell \in \keys{0,1,2,3}$. For this we need bounds on the $\ell_2$ norm of $\vct{v_{\ell}} \brac{f}$ and the magnitude of $R^{\brac{\ell}} \brac{ f }$ that hold over the whole unit interval, not only on a discrete grid. By the definitions of $K$ and $\vct{b}\brac{j}$ in~\eqref{eq:k} and~\eqref{eq:b}, for any $f$
\begin{align}
\normTwo{ \vct{v_{\ell}} \brac{f}} & = \normTwo{ \sum_{l \in \Omega^c} \brac{  i 2 \pi \kappa l}^\ell \vct{c}_{ l } e^{i2\pi l f}  \vct{b}\brac{l} } \\
& \leq \pi^{\ell} n \normInf{c} \sup_{-m \leq l \leq m} \normTwo{\vct{b} \brac{ l }} \quad \text{by Lemma~\ref{lemma:bound_kappa}}\\
& \leq \frac{1.3 \, \pi^{3} n \sqrt{10 k} }{m}   \quad \text{by Lemmas~\ref{lemma:c_amp} and~\ref{lemma:bound_b}} \\
& \leq 256 \sqrt{k}.
\end{align}
Similarly, for any $f$
\begin{align}
\abs{ \kappa^{\ell} R^{\brac{\ell}} \brac{ f } } & = \abs{\lambda \, \kappa^{\ell} \sum_{l \in \Omega} \brac{- i 2 \pi l}^{\ell} \signz_l \, e^{- i 2 \pi l f} } \\
& \leq \frac{ \kappa^{\ell} \brac{2 \pi}^{\ell}}{\sqrt{n}} \sum_{l \in \Omega}  l^{\ell} \\
& \leq \frac{ \kappa^{\ell} \brac{2 \pi}^{\ell} s \, m^{\ell} }{\sqrt{n}}\\
& \leq  \frac{ 4 \pi^3 s}{\sqrt{n}} \quad \text{by Lemma~\ref{lemma:bound_kappa}}. 
\end{align}
We also derive a coarse bound on the operator norm $B_{\Omega}$
\begin{align}
\norm{B_{\Omega}} & \leq \sqrt{ \norm{\bar{H}} } \\
& \leq  \sqrt{260 \, \pi^2 n \, \log k} \quad \text{by Lemma~\ref{lemma:boundHbar} }
\end{align}
which holds because $B_{\Omega}$ is a submatrix of a matrix $\bar{B}$ such that $\bar{H}=\bar{B}\bar{B}^{\ast}$. These bounds together with~\eqref{eq:Q_vl_Dinv_ur}, the Cauchy-Schwarz inequality and the triangle inequality imply that in $\ml{E}_{D}^c$
\begin{align}
\abs{ \kappa^{\ell} \, Q^{\brac{\ell}} \brac{f} } & \leq \normTwo{\vct{v_{\ell}}\brac{f}} \norm{ D^{-1}} \brac{ \normTwo{ \signx  } +  \frac{1}{\sqrt{n}} \norm{ B_{\Omega}} \normTwo{ \signz } } + \abs{ \kappa^{\ell}  R^{\brac{\ell}} \brac{ f } } \\ 
& \leq 5 \, 10^5 \, \brac{k + \sqrt{ks \, \log k}}\\
& \leq \frac{ n }{7} \quad \text{by~\eqref{eq:cond_k} and~\eqref{eq:cond_s} if we set $C_k$ and $C_s$ small enough }. \label{eq:boundunif_Q}
\end{align}
Finally, if we interpret $Q^{\brac{\ell}}\brac{z}$ as a function of $z \in \C$, a generalization of the mean-value theorem yields 
\begin{align}
\abs{ \kappa^{\ell} Q^{\brac{\ell}}\brac{f_2} -  \kappa^{\ell} Q^{\brac{\ell}}\brac{f_1} } & \leq \kappa^{\ell}  \abs{e^{i 2 \pi f_2} - e^{i 2 \pi f_1}} \sup_{z'} \abs{ \der{z}{Q^{\brac{\ell}}\brac{z'}}} \\
& \leq \frac{2 \pi \abs{f_2 - f_1}}{\kappa} \sup_{f} \abs{\kappa^{\ell + 1} Q^{\brac{\ell +1}}\brac{f} } \\ 
& \leq n^2 \abs{f_2 - f_1} \quad \text{by \eqref{eq:boundunif_Q} for $\ell \in \keys{0,1,2}$}.
\end{align}
The bound on the deviation of $\bar{Q}^{\ell}$ is obtained using exactly the same argument together with the bound~\eqref{eq:unifbound_Qbar}. In the case of $\bar{Q}$ the bound is extremely coarse, but it suffices for our purpose.

\section{Proof of Proposition~\ref{proposition:qboundgamma}}
\label{proof:qboundgamma}
Let $l$ be an arbitrary element of $\Omega^c$. We express the corresponding coefficient $\vct{q}_{l}$ in terms of the sign patterns $\signx$ and $\signz$,
\begin{align}
\vct{q}_{l} & =  \vct{c}_l \brac{\sum_{j =1}^{k} \vct{\alpha}_j \, e^{i2\pi l f_j} + i2\pi l   \kappa \sum_{j =1}^{k} \vct{\beta}_j \, e^{i2\pi l f_j}} \\
& =  \vct{c}_l \, \vct{b}\brac{l}^{\ast} \MAT{\vct{\alpha} \\ \vct{\beta} } \\
& =  \vct{c}_l \, \vct{b}\brac{l}^{\ast} D^{-1} \brac{\MAT{  \signx  \\ 0 } - \frac{1}{\sqrt{n}} B_{\Omega} \, \signz } \\
& =  \vct{c}_l \, \brac{ \PROD{ P D^{-1} \vct{b}\brac{l} }{\signx} +  \frac{1}{\sqrt{n}} \PROD{B_{\Omega}^{\ast} D^{-1} \vct{b}\brac{l} }{ \signz }},
\end{align}
where $P \in \R^{k \times 2k}$ is the projection matrix that selects the first $k$ entries in a vector. 

The bounds
\begin{align}
\normTwo{P D^{-1} \vct{b}\brac{l}}^2 & \leq \norm{P}^2 \norm{D^{-1}}^2   \normTwo{\vct{b}\brac{l}}^2 \\
& \leq 640 k \quad \text{in $\ml{E}_{D}^{c}$ by Lemmas~\ref{lemma:bound_b} and~\ref{lemma:D_bounds}} \\
& \leq \frac{ 0.18^2 n}{ \log \frac{40}{\epsilon}} \text{by~\eqref{eq:cond_k} if we set $C_k$ small enough,}
\end{align}
and 
\begin{align}
\normTwo{B_{\Omega}^{\ast} D^{-1} \vct{b}\brac{l}}^2 & \leq \norm{B_{\Omega} }^2  \norm{D^{-1}}^2   \normTwo{\vct{b}\brac{l}}^2 \\
& \leq 640 \, C_{B}^2 \, k n \quad \text{in $\ml{E}_{B}^{c} \cap \ml{E}_{D}^{c}$ by Lemmas~\ref{lemma:boundB} and~\ref{lemma:D_bounds}} \\
& \leq \frac{ 0.18^2 n^2}{ \log \frac{40}{\epsilon}} \quad \text{by~\eqref{eq:cond_k} if we set $C_k$ small enough,}
\end{align}
imply by Hoeffding's inequality (Theorem~\ref{theorem:hoeffding}) that the probability of each of the events 
\begin{align}
\abs{  \PROD{ P D^{-1} \vct{b}\brac{l} }{\signx}} & > 0.18 \sqrt{n} , \\
\abs{ \PROD{B_{\Omega}^{\ast} D^{-1} \vct{b}\brac{l} }{ \signz }} & > 0.18 n
\end{align}
is bounded by $\epsilon / 10$. 
By Lemma~\ref{lemma:c_amp} and the union bound, this implies 
\begin{align}
\abs{\vct{q}_l} & \leq \normInf{ \vct{c} } \brac{ \abs{ \PROD{D^{-1} \vct{b}\brac{l} }{ \MAT{  \signx  \\ 0 } }} + \frac{\abs{ \PROD{B_{\Omega}^{\ast} D^{-1} \vct{b}\brac{l} }{ \signz }}}{\sqrt{n}} } \\
& \leq \frac{2.6}{n} \brac{0.18 \sqrt{n} + 0.18 \sqrt{n}}\\
& < \frac{1}{\sqrt{n}}
\end{align}
with probability at least $1-\epsilon/5$.

%
%
%

\section{Algorithms}

\subsection{Proof of Lemma~\ref{lemma:dual_sdp_noise}}
\label{proof:dual_sdp_noise}
The problem is equivalent to
\begin{align}
\min_{ \tilde{\mu}, \vct{ \tilde{z} } ,\vct{u} }    \normTV{ \tilde{\mu}} + \lambda \normOne{ \vct{\tilde{z}} } \qquad
  \text{\ subject\ to\ }  & \normTwo{ \vct{y} - \vct{u} }^2 \leq \sigma^2 \\
   & \ml{F}_n \, \tilde{\mu} + \vct{\tilde{z}} = \vct{u} ,
\end{align}
where we have introduced an auxiliary primal variable $\vct{u}\in \C^{n}$. Let us define the dual variables $\vct{\eta} \in \C^{n}$ and $\nu  \geq 0$. The Lagrangian is equal to
\begin{align}
\mathcal{L}\brac{\tilde{ \mu },\vct{\tilde{ z }},\vct{\eta}} & = \normTV{\tilde{ \mu }} + \lambda \normOne{\vct{\tilde{ z }}} + \PROD{ \vct{u} -\mathcal{F}_{n} \, \tilde{ \mu } - \vct{\tilde{ z }} }{\vct{\eta}} + \nu \brac{ \normTwo{ \vct{y} - \vct{u} }^2 - \sigma^2 }
\\
& =\normTV{\tilde{ \mu }}-  \PROD{\tilde{ \mu }}{\mathcal{F}_{n}^{\ast} \,\vct{\eta}} + \lambda \normOne{ \vct{\tilde{ z }}} - \PROD{ \vct{\tilde{ z }}}{\vct{\eta}} +  \PROD{\vct{u}}{\vct{\eta}} + \nu \brac{ \normTwo{ \vct{y} - \vct{u} }^2 - \sigma^2 }
\end{align}
where $\eta \in \C^{n}$ is the dual variable. 

To compute the Lagrange dual function we minimize the value of the Lagrangian over the primal variables~\cite{Boyd:2004uz}. The minimum of
\begin{align}
 \normTV{\tilde{ \mu }}-  \PROD{\tilde{ \mu }}{\mathcal{F}_{n}^{\ast} \,\vct{\eta}} 
\end{align}
over $\tilde{ \mu }$ is $-\infty$ unless~\eqref{eq:cond_one_primaldual} holds.
Moreover, if~\eqref{eq:cond_one_primaldual} holds then the minimum is at $\tilde{\mu}=0$ by H\"older's inequality. Similarly, minimizing 
\begin{align}
\lambda \normOne{ \vct{\tilde{ z }}} - \PROD{ \vct{\tilde{ z }}}{\vct{\eta}}
\end{align}
over $\vct{z}$ yields $-\infty$ unless~\eqref{eq:cond_lambda_primaldual} holds,
whereas if~\eqref{eq:cond_lambda_primaldual} holds the minimum is attained at $\vct{\tilde{z}}=0$. All that remains is to minimize
\begin{align}
 \PROD{\vct{u}}{\vct{\eta}} + \nu \brac{ \normTwo{ \vct{y} - \vct{u} }^2 - \sigma^2 }
\end{align}
with respect to $\vct{u}$ (note that~\eqref{eq:cond_one_primaldual} and~\eqref{eq:cond_lambda_primaldual} do not involve $\vct{u}$). The function is convex with respect to $\vct{u}$ so we set the gradient to zero to deduce that the minimum is at $\vct{u} = \vct{y} - \frac{1}{2\nu} \eta $. Plugging in this value yields the Lagrange dual function
 \begin{align}
 \label{eq:dual_long}
\PROD{\vct{y}}{\vct{\eta}} -\frac{1}{4 \nu} \normTwo{ \eta }^2 - \nu \sigma^2 .
\end{align}
The dual problem consists of maximizing the Lagrange dual function subject to $\nu \geq 0 $, \eqref{eq:cond_one_primaldual} and~\eqref{eq:cond_lambda_primaldual}. For any fixed value of $\tilde{\eta}$, maximizing over $\nu$ is easy, the expression is convex in the half plane $\nu \geq 0$ and the derivative is zero at $\normTwo{ \eta } / 2\sigma$. Plugging this into~\eqref{eq:dual_long} yields the dual problem~\eqref{eq:dual_noise}.

The reformulation of~\eqref{eq:dual_noise} as a semidefinite program is an immediate consequence of the following proposition.

\begin{proposition}[Semidefinite characterization~{\cite[Theorem 4.24]{Dumitrescu:2007vw}}, {\cite[Proposition 2.4]{superres_new}}]
\label{prop:sdp-charact}
Let $\vct{\eta} \in \C^{n }$,
\begin{align*}
\abs{ \brac{ \mathcal{F}_{n}^{\ast} \, \vct{\eta}} (f) } & \leq 1 \quad \text{for all } f \in \sqbr{0,1}
\end{align*}
if and only if there exists a Hermitian matrix $\Lambda \in \C^{n \times n}$ obeying
 \begin{equation}
\label{eq:sdp-charact}
   \MAT{\Lambda & \vct{\eta} \\ \vct{\eta}^{\ast} & \Id } \succeq 0, \qquad \mathcal{T}^{\ast}\brac{\Lambda}= \MAT{1 \\ \vct{0}},
\end{equation}
where $\vct{0} \in \C^{n-1}$ is a vector of zeros.
\end{proposition}

\subsection{Proof of Lemma~\ref{lemma:primaldual_noise}}
\label{proof:primaldual_noise}
The interior of the feasible set of Problem~\eqref{eq:dual_noise} contains the origin and is
therefore non empty, so strong duality holds by a generalized Slater condition \cite{rockafellar1974conjugate} and we have
\begin{align}
\sum_{f_j \in \widehat{T}} \abs{\vct{\hat{x}}_{j}} +\lambda \sum_{l \in \widehat{\Omega}} \abs{\vct{\hat{z}}_{l}}   = \normTV{\hat{\mu}} + \lambda \normOne{\vct{\hat{z}}} 
& = \PROD{\vct{\hat{\eta}}}{ \vct{y} } - \sigma \normTwo{\vct{\eta}} \\
& \leq \PROD{\vct{\hat{\eta}}}{ \vct{y} } -  \PROD{\vct{\hat{\eta}}}{\vct{y}-\mathcal{F}_{n} \, \hat{\mu} - \vct{\hat{z}}} \label{eq:CauchySchwarzineq}\\
& = \PROD{\vct{\hat{\eta}}}{\mathcal{F}_{n} \, \hat{\mu}+ \vct{\hat{z}} }\\
& = \operatorname{Re}\sqbr{\sum_{f_j \in \widehat{T}} \abs{\vct{\hat{x}}_{j}}  \overline{\brac{\mathcal{F}_{n}^{\ast} \, \vct{\hat{\eta}}}\brac{f_j}} \frac{\vct{\hat{x}}_j}{\abs{\vct{\hat{x}}_j}} +\sum_{l \in \widehat{\Omega}} \abs{\vct{\hat{z}}_{l}} \overline{\vct{\hat{\eta}}_l}\frac{\vct{\hat{z}}_l}{\abs{\vct{\hat{z}}_l}} }.
\end{align}
The inequality~\eqref{eq:CauchySchwarzineq} follows from the Cauchy-Schwarz inequality because $\keys{\hat{\mu}, \vct{\hat{z}}}$ is primal feasible and hence $\normTwo{\vct{y}-\mathcal{F}_{n} \, \hat{\mu} - \vct{\hat{z}}} \leq \sigma$. Due to the constraints~\eqref{eq:cond_one_primaldual} and~\eqref{eq:cond_lambda_primaldual} and H\"older's inequality, the inequality that we have established is only possible if~\eqref{eq:eta_mu_noise} and \eqref{eq:eta_z_noise} hold. The proof is complete.

\subsection{Atomic-noise denoising via the alternating direction method of multipliers}
\label{sec:admm}
We rewrite Problem~\eqref{eq:denoising_sdp_noise} as
\begin{align}
	\label{eq:sdp2}
	 \min_{ \substack{ t \in \R, \, \vct{u} \in \C^n, \\ \vct{\tilde{g}} \in \C^n , \, \vct{\tilde{z}} \in \C^n \\ \Psi \in \C^{n+1 \times n+1} } } \;   
	\frac{ \xi }{2}   \brac{n \, \vct{u}_1 + t} + \lambda'  \normOne{ \vct{\tilde{z}} }  + \frac{1}{2}  \normTwo{ \vct{y} - \vct{\tilde{g}} - \vct{\tilde{z}} }^2  \quad
 \text{subject to} \quad &  \Psi = \MAT{ \ml{T}\brac{ \vct{u} } & \vct{\tilde{g}} \\ \vct{\tilde{g}}^{\ast} & t} , \\
  &  \Psi \succeq 0,
\end{align}
where $ \xi := \frac{1}{\gamma\sqrt{n}}$ and $\lambda' := \frac{\lambda}{\gamma}$. The augmented Lagrangian for this problem is of the form
\begin{align}
	\label{eq:AL}
	\ml{ L}_\rho \brac{ t,\vct{u},\vct{\tilde{g}},\vct{\tilde{z}},\Upsilon,\Psi }  := &  \frac{ \xi }{2}   \brac{n \, \vct{u}_1 + t} + \lambda'  \normOne{ \vct{\tilde{z}} }  + \frac{1}{2}  \normTwo{ \vct{y} - \vct{\tilde{g}} - \vct{\tilde{z}} }^2 + \PROD{\Upsilon}{ \Psi - \MAT{ \ml{T}\brac{ \vct{u} } & \vct{\tilde{g}} \\ \vct{\tilde{g}}^{\ast} & t} }\\
	 & + \frac{\rho}{2} \normF{ \Psi - \MAT{ \ml{T}\brac{ \vct{u} } & \vct{\tilde{g}} \\ \vct{\tilde{g}}^{\ast} & t} }^2,
\end{align}
where $\rho > 0$ is a parameter. The alternating direction method of multipliers (ADMM) minimizes the augmented Lagrangian by iteratively applying the updates:
\begin{align}
\label{eq:UPDA1}
t^{\brac{l+1}} & :=  \arg\min_{t}  \ml{ L}_\rho \brac{ t,\vct{u}^{\brac{l}},\vct{\tilde{g}}^{\brac{l}},\vct{\tilde{z}}^{\brac{l}},\Upsilon^{\brac{l}},\Psi^{\brac{l}}} , \\
\label{eq:UPDA2}
\vct{u}^{\brac{l+1}} & :=  \arg\min_{\vct{u}} \ml{ L}_\rho \brac{ t^{\brac{l}},\vct{u},\vct{\tilde{g}}^{\brac{l}},\vct{\tilde{z}}^{\brac{l}},\Upsilon^{\brac{l}},\Psi^{\brac{l}}} , \\
\vct{\tilde{g}}^{\brac{l+1}}  & :=  \arg\min_{\vct{\tilde{g}}}  \ml{ L}_\rho \brac{ t^{\brac{l}},\vct{u}^{\brac{l}},\vct{\tilde{g}},\vct{\tilde{z}}^{\brac{l}},\Upsilon^{\brac{l}}, \Psi^{\brac{l}}} , 
\label{eq:UPDA3}\\
\vct{\tilde{z}}^{\brac{l+1}} & :=  \arg\min_{ \vct{\tilde{z}} }  \ml{ L}_\rho \brac{ t^{\brac{l}},\vct{u}^{\brac{l}},\vct{\tilde{g}}^{\brac{l}},\vct{\tilde{z}},\Upsilon^{\brac{l}}, \Psi^{\brac{l}}} , 
\label{eq:UPDA4} \\
\Psi^{\brac{l+1}} & := \arg\min_{\Psi} \ml{ L}_\rho \brac{ t^{\brac{l}},\vct{u}^{\brac{l}},\vct{\tilde{g}}^{\brac{l}},\vct{\tilde{z}}^{\brac{l}},\Upsilon^{\brac{l}},\Psi },\\
\label{eq:UPDAUpsilon}
\Upsilon^{\brac{l+1}} &:= \Upsilon^{\brac{l}} + \rho \brac{ \Psi^{\brac{l+1}}  -  \MAT{ \ml{T}\brac{ \vct{u}^{\brac{l+1}} } & \vct{\tilde{g}}^{\brac{l+1}} \\ \brac{\vct{\tilde{g}}^{\brac{l+1}}}^{\ast} & t^{\brac{l+1}}} },
\end{align}
where $l$ indicates the iteration number. We refer the interested reader to the tutorial~\cite{boyd2011distributed} and references therein for a justification of these steps and more information on ADMM. 

For the method to be practical, we need an efficient implementation of all the updates. The augmented Lagrangian is convex and differentiable with respect to $t$, $\vct{u}$ and $\vct{\tilde{g}}$, so for these variables we just need to compute their gradient and set it to zero. This yields the closed-form updates:
\begin{align}
\label{eq:UPDAtheta}
t^{\brac{l+1}} &= \Psi_{n+1}^{\brac{l}} + \frac{1}{\rho}\left(\Upsilon_{n+1}^{\brac{l}} - \frac{\xi}{2} \right),\\
\vct{u}^{\brac{l+1}} &= M \, \ml{T}^{\ast}  \brac{ \Psi_{0}^{\brac{l}} + \frac{\Upsilon_{0}^{\brac{l}}}{\rho}} - \frac{\xi}{2 \rho} \vct{e}\brac{1} ,\\
\vct{\tilde{g}}^{\brac{l+1}} &= \frac{1}{2\rho+1} \brac{ \vct{y} - \vct{\tilde{z}} ^{\brac{l}} + 2\rho \, \vct{\psi}^{\brac{l}} + 2 \vct{\upsilon}^{\brac{l}}},
\end{align}
where $\vct{e}\brac{1}: = [1,0,0,...,0]^T$, $\ml{T}^{\ast}$ outputs a vector whose $j$-th element is the trace of the $(j-1)$-th subdiagonal of the input matrix, $M$ is a diagonal matrix such that
\begin{align}
M_{j,j} =\frac{1}{{n - j + 1}}, \quad j = 1,...n,
\end{align}
and
\begin{align}
\Psi^{\brac{l}} &:= \left[ {\begin{array}{*{20}{c}}
		{\Psi_{0}^{\brac{l}}}&{\vct{\psi}^{\brac{l}}}\\
		{(\vct{\psi}^{\brac{l}})^*}&{\Psi_{n+1}^{\brac{l}}}
		\end{array}} \right], \qquad
\Upsilon^{\brac{l}} := \left[ {\begin{array}{*{20}{c}}
		{\Upsilon _{0}^{\brac{l}}}&{\vct{\upsilon}^{\brac{l}}}\\
		{(\vct{\upsilon}^{\brac{l}})^*}&{\Upsilon_{n+1}^{\brac{l}}}
	\end{array}} \right].
\end{align}
$\Psi _{0}^{\brac{l}}$ and $\Upsilon_{0}^{\brac{l}}$ are $n \times n$ matrices, $\vct{\psi}^{\brac{l}}$ and $\vct{\upsilon}^{\brac{l}}$ are $n$-dimensional vectors and $\Psi_{n+1}^{\brac{l}}$ and $\Upsilon _{n+1}^{\brac{l}}$ are scalars. 

Updating $\vct{\tilde{z}}$ requires solving the problem 
\begin{align}
 \min_{\vct{\tilde{z}}} \lambda' \|\vct{\tilde{z}}\|_1 + \frac{1}{2} \|\vct{y}-\vct{\tilde{g}}^{\brac{l}}-\vct{\tilde{z}}\|_2^2,
\end{align}
which is easily achieved by the applying a proximal operator
\begin{align}
\vct{\tilde{z}}^{\brac{l+1}} := \op{prox}_{\lambda'}(\vct{y} - \vct{\tilde{g}}^{\brac{l}}),
\end{align}
where for $1\leq j \leq n$
\begin{align}
\label{eq:prox}
\op{prox}_{\lambda'} \brac{ \vct{\tilde{z} }}_j := 
\begin{cases}
\op{sign} \brac{\vct{\tilde{z}_j}} \brac{ \abs{ \vct{\tilde{z}}_j }- \lambda'}  \quad & \text{if $\abs{ \vct{\tilde{z}}_j } > \lambda'$ } \\
0 & \text{otherwise}.
\end{cases}
\end{align}	
Finally, the update of $\Psi^{\brac{l}}$ amounts to a projection onto the positive semidefinite cone
\begin{align}
\label{eq:UPDAPhi}
\Psi^{\brac{l+1}} = \arg\min_{\Psi \succeq 0} \left \| {\Psi  - \MAT{ \ml{T}\brac{ \vct{u}^{\brac{l}} } & \vct{\tilde{g}}^{\brac{l}} \\ \brac{\vct{\tilde{g}}^{\brac{l}}}^{\ast} & t^{\brac{l}}} } + \frac{1}{\rho} \Upsilon^{\brac{l}} \right\|_F^2,
\end{align}
which can be accomplished by computing the eigenvalue decomposition of the matrix and setting all negative eigenvalues to zero. 

\end{document}